\documentclass[12pt]{article}

\usepackage[doi=false,isbn=false,url=false,eprint=false]{biblatex}
\usepackage{geometry}
\usepackage{amsfonts}
\usepackage{amsmath}
\usepackage{amssymb}
\usepackage{amsthm}
\usepackage{mathrsfs}
\usepackage{stmaryrd}
\usepackage{mathtools}
\usepackage{xfrac}
\usepackage{tikz-cd}
\usepackage{enumerate}
\usepackage{tabularx}
\usepackage{hyperref}
\usepackage{adjustbox}

\numberwithin{equation}{section}

\DeclareMathOperator{\id}{Id}

\newcommand{\norm}[1]{\left\|#1\right\|}
\newcommand{\crochet}[1]{\left\langle{#1}\right\rangle}
\newcommand{\curly}[1]{\left\{{#1}\right\}}

\newcommand{\verts}[1]{\left|{#1}\right|}
\newcommand{\linf}[1]{|{#1}|_{\ell^\infty}}
\newcommand{\normm}[1]{{\left\vert\kern-0.15ex\left\vert\kern-0.15ex\left\vert#1\right\vert\kern-0.15ex\right\vert\kern-0.15ex\right\vert}}

\newcommand{\ol}{\overline}
\newcommand{\wt}{\widetilde}
\newcommand{\wh}{\widehat}
\newcommand{\inv}{^{-1}}

\newcommand{\eu}{\mathrm{e}}
\newcommand{\iu}{\mathrm{i}}
\newcommand{\du}{\mathrm{d}}

\newcommand{\Z}{\mathbb{Z}}
\newcommand{\R}{\mathbb{R}}
\newcommand{\C}{\mathbb{C}}
\newcommand{\Q}{\mathbb{Q}}
\newcommand{\N}{\mathbb{N}}
\newcommand{\T}{\mathbb{T}}
\newcommand{\D}{\mathbb{D}}

\newcommand{\CC}{\mathcal{C}}
\newcommand{\CU}{\mathcal{U}}
\newcommand{\CP}{\mathcal{P}}
\newcommand{\CD}{\mathcal{D}}
\newcommand{\CG}{\mathcal{G}}
\newcommand{\CH}{\mathcal{H}}
\newcommand{\CQ}{\mathcal{Q}}
\newcommand{\CT}{\mathcal{T}}
\newcommand{\CO}{\mathcal{O}}
\newcommand{\CA}{\mathcal{A}}
\newcommand{\CR}{\mathcal{R}}

\newcommand{\sF}{\mathscr{F}}
\newcommand{\sH}{\mathscr{H}}
\newcommand{\sS}{\mathscr{S}}

\newcommand{\fg}{\mathfrak{g}}
\newcommand{\fF}{\mathfrak{F}}
\newcommand{\fN}{\mathfrak{N}}

\newtheorem{theorem}{Theorem}[section]
\newtheorem{lemma}[theorem]{Lemma}
\newtheorem{corollary}[theorem]{Corollary}
\newtheorem{proposition}[theorem]{Proposition}

\theoremstyle{definition}

\newtheorem{question}{Question}
\theoremstyle{remark}
\newtheorem*{remark}{Remark}

\title{Deformations of the Standard Map
       with Prescribed Actions and Lyapunov Exponents}

\author{Yunzhe Li} 

\date{}

\begin{document}

\maketitle

\begin{abstract}
We construct nontrivial deformations of the standard map which preserve the symplectic actions, respectively the Lyapunov exponents, of infinitely many periodic orbits accumulating to an invariant curve. 
The proof uses a resonant normal-form construction to obtain a sequence of periodic orbits accumulating on an invariant curve with a Liouville rotation number.
Within these normal forms we capture the dependence of these periodic orbits on the resonant Fourier coefficients of the dynamics on the invariant curve and, using the contraction mapping principle, obtain a suitable deformation achieving the prescribed spectral data associated with this sequence of orbits.

The result can be viewed as a symplectic twist-map analogue of a length spectral \textit{nonrigidity} phenomenon for Riemannian manifolds and convex billiards, and it motivates the existence problem for similar ``partially length-isospectral'' deformations of strictly convex billiard tables.
\end{abstract}

\section{Introduction}\label{sec:main results}
Let $\T = \R / \Z$ and let $V: \T \to \R$ be a $\CC^1$ function.
The \textit{(Chirikov) standard map} $\mathscr{F}$ associated with the \textit{potential} $V$ is the dynamical system on $\T \times \R$ defined by 
\begin{equation}\label{eq:std map}
    \begin{aligned}
        \mathscr{F}: \T \times \R &\to \T \times \R,\\
        (x,y) &\mapsto \bigl(x + y + \dot V(x),\, y + \dot V(x)\bigr),
    \end{aligned}
\end{equation}
where $\dot V := \frac{dV}{dx}$.

Suppose $(x,y)$ is periodic under $\mathscr{F}$ with (minimal) period $q \ge 1$ and write $(x_k, y_k) = \sF^k(x,y)$ for $k = 0,1,\ldots, q-1$.
Then the \textit{symplectic action} of this periodic orbit is defined as 
\begin{equation}\label{eq:action functional}
     \sum_{k = 0}^{q-1} \Bigl[\frac{1}{2}\bigl(x_k - x_{k+1}\bigr)^2 + V(x_k)\Bigr]
\end{equation}
with indices taken modulo~$q$.
Let $\du_{(x,y)} \sF^q$
denote the differential of the iterated map $\sF^q$ at $(x,y)$, viewed as a linear operator on the tangent space of $\T \times \R$ at $(x,y)$.
The \textit{Lyapunov exponent} of the periodic orbit $(x_k, y_k)_{k = 0}^{q-1}$ is defined to be
\begin{equation}\label{eq:lyap exp}
    \det\bigl( \du_{(x,y)}\sF^q - \id \bigr).
\end{equation}

The primary goal of this paper is to construct examples of nontrivial deformations of the potential $V$ that preserve the symplectic actions (respectively, the Lyapunov exponents) of infinitely many periodic orbits.
In particular, we prove:
\begin{theorem}\label{thm:main1}
    There exists a nontrivial analytic family of standard maps $(\sF_\tau)_{\tau}$ such that, for each $\tau$, the map $\sF_\tau$ has infinitely many dynamically distinct 
    \footnote{We say that two periodic orbits are \textit{dynamically distinct} if the orbits do not coincide as sets (they are not just reparametrizations of one another, and one is not a repetition of the other).}
    periodic orbits depending analytically on $\tau$, and the symplectic action of each of these periodic orbits is constant in~$\tau$.
    The same conclusion holds with `symplectic action' replaced by `Lyapunov exponent'.
\end{theorem}

The standard map is a prototypical example of an \textit{exact symplectic twist map}.
It is a classical fact that for every coprime pair $(p,q)$ of integers with $q > 0$, a $(p,q)$-periodic orbit of an exact symplectic twist map can be given by a minimum of the \textit{action functional} 
\begin{equation*}
    (x_k)_{k = 0}^{q-1} \mapsto \sum_{k = 0}^{q-1} S(x_k ,x_{k+1}), \qquad (x_q = x_0 + p),
\end{equation*}
where $S$ is a generating function of the exact symplectic map.
The minimizing sequence $(x_k)_{k = 0}^{q-1}$ corresponds to the first coordinate of a periodic orbit of period $q$ lifted to a suitable universal cover.
Such orbits are called \textit{minimal orbits}.

For the standard map with potential $V$, the generating function can be taken as
\begin{equation}\label{eq:std map gen func intro}
    S(x,x') = \frac{1}{2}(x - x')^2 + V(x).
\end{equation}
We refer to Section~\ref{sec:notation} for more details on exact symplectic maps and the definition of $(p,q)$-periodic orbits.

Given an exact symplectic map $F$ with a choice of generating function, \textit{Mather's beta function} is defined to be the continuous function $\beta: \R \to \R$ such that
\begin{equation}\label{eq:beta func intro}
    \beta\Big( \frac{p}{q}\Big) = \frac{1}{q} \min\{\;\text{symplectic actions of $(p,q)$-periodic orbits}\;\}.
\end{equation}
For a potential $V: \T \to \R$, let $\beta_V$ denote Mather's beta function associated with the standard map defined by the generating function \eqref{eq:std map gen func intro}.
We now formulate the second version of our main result.
\begin{theorem}\label{thm:main beta}
    Let $\omega \in \R \setminus \Q$ be such that 
    \begin{equation*}
        \verts{\omega - \frac{p}{q}} \leq q^{-70} \eu^{-2\pi q}
    \end{equation*}
    for infinitely many coprime pairs of integers $(p,q)$.
    Then there exists a nontrivial real-analytic family of potentials $(V_\tau: \T \to \R)_{\tau}$ such that, for a sequence $p_j/ q_j \to \omega$, the value 
    \begin{equation*}
        \beta_{V_\tau}\Big(\frac{p_j}{q_j}\Big)
    \end{equation*}
    is constant in $\tau$ for each $j$.
\end{theorem}

Theorems~\ref{thm:main1} and~\ref{thm:main beta} are immediate consequences of Theorem~\ref{thm:main}, which we now state.
We refer to Section~\ref{sec:notation} for the notation.

\begin{theorem}\label{thm:main}
    Let $\omega \in \R \setminus \Q$ and let $(p_j, q_j)$ be a sequence of coprime pairs of integers such that $p_j/q_j \to \omega$ as $j \to \infty$, with $q_{j+1} > 4 q_j$ and $q_1$ sufficiently large.
    Then the following statements hold.
    \begin{enumerate}
        \item Suppose
        \begin{equation}\label{eq:main dioph condition}
            \verts{\omega - \frac{p_j}{q_j}} \leq q_j^{-70}.
        \end{equation}
        Then there exists a real-analytic map 
        \begin{equation*}
            \alpha \mapsto \phi_\alpha
        \end{equation*}
        sending any sequence $\alpha = (\alpha_j)_{j \geq 1}$ of real numbers with $|\alpha_j| < 1/4$ to a real-analytic function $\phi_\alpha: \T \to \R$ with $\|\dddot \phi_\alpha\| \leq 1$, where $\|\;\cdot\;\|$ is the Fourier-weighted norm defined in Section~\ref{sec:notation}, such that the standard map associated with a potential $V_\alpha$ satisfying
        \begin{equation}\label{eq:V_alpha}
            \dot V_\alpha\bigl(t+\phi_\alpha(t)\bigr) \;=\; \phi_\alpha(t+\omega)-2\phi_\alpha(t)+\phi_\alpha(t-\omega)
        \end{equation}
        has a sequence of $(p_j, q_j)$-periodic orbits $(\xi_j)_{j \geq 1}$, each depending analytically on $\alpha$, and the Lyapunov exponent of $\xi_j$ is
        \begin{equation}\label{eq:main thm lyap}
            - (2 + \alpha_j)\, q_j \Bigl( \frac{p_j}{q_j} - \omega \Bigr)^2 \eu^{-2\pi q_j}.
        \end{equation}

        \item Suppose
        \begin{equation}\label{eq:main Liouv condition}
            \verts{\omega - \frac{p_j}{q_j}} \leq q_j^{-70} \eu^{-2\pi q_j}.
        \end{equation}
        Then there exist $\epsilon_0 > 0$ and a real-analytic map 
        \begin{equation*}
            \alpha \mapsto \phi_\alpha
        \end{equation*} 
        sending any sequence $\alpha = (\alpha_j)_{j \geq 1}$ of real numbers with $|\alpha_j| < \epsilon_0$ to a real-analytic function $\phi_\alpha: \T \to \R$ with analytic norm $\|\dddot\phi_\alpha\| \leq 1$ such that the unique standard map associated with a potential $V_\alpha$, which depends analytically on $\alpha$ and satisfies \eqref{eq:V_alpha},
        has a sequence of $(p_j, q_j)$-periodic orbits $(\xi_j)_{j \geq 1}$, each depending analytically on $\alpha$, and the symplectic action of $\xi_j$ is 
        \begin{equation}\label{eq:main thm action}
            \frac{q_j \bigl(\frac{p_j}{q_j} - \omega\bigr)^2}{2(1 + f_1^*)} \bigg[ 1 - \frac{(2 + \alpha_j)\eu^{-2\pi q_j} }{2\pi^2 q_j^3} \bigg]
             + \bigl[\omega - (\phi_\alpha \dot\phi_\alpha^+)^*\bigr]\, p_j,
        \end{equation}
        where $\phi_\alpha^+ := \phi_\alpha(\cdot + \omega)$ and $f_1^*$ and $(\phi_\alpha \dot\phi_\alpha^+)^*$ denote the averages over $\T$ of the functions
        \begin{equation*}
           f_1 = (1 + \dot\phi_\alpha)^{-1}(1 + \dot\phi_\alpha^+)^{-1} - 1
           \quad\text{and}\quad
           \phi_\alpha \dot\phi_\alpha^+,
        \end{equation*}
        respectively.
        Moreover, the map $\alpha \mapsto \phi_\alpha$ can be constructed so that the quantities $f_1^*$ and $(\phi_\alpha \dot\phi_\alpha^+)^*$ are both constant in~$\alpha$ and for each $j$ the orbit $\xi_j$ can be chosen as a minimal orbit, i.e., an orbit which minimizes the symplectic action among all $(p_j, q_j)$-orbits.
    \end{enumerate}
\end{theorem}

We provide some informal clarification of the statement of Theorem~\ref{thm:main}.
The analyticity of the map $\alpha \mapsto \phi_\alpha$ means the following: if $\alpha^\tau = (\alpha_j^\tau)_{j\ge1}$ is a family of sequences parametrized by $\tau \in \C$ such that, for each $j$, the function $\tau \mapsto \alpha_j^\tau$ is analytic, then the resulting function $\phi_{\alpha^\tau}$ depends analytically on~$\tau$.
Moreover, if $\alpha_j \in \R$ for all $j$, then $\phi_\alpha$ is real-valued on $\T$.
For complex-valued sequences $(\alpha_j)$, analogous results hold for the complexified standard map with a complex-valued potential.

Given any real-analytic $\phi_\alpha$ satisfying the smallness assumptions in Theorem~\ref{thm:main}, one can always solve \eqref{eq:V_alpha} for $V_\alpha$ up to an additive constant (see Section~\ref{sec:RIC coord}).
Therefore, there is indeed a unique standard map associated with each sequence $\alpha$ via the map $\alpha \mapsto \phi_\alpha$.
In part (2) of Theorem \ref{thm:main}, the additive constant is fixed by the condition \eqref{eq:main thm action} and this constant depends analytically on $(\alpha_j)$.

By our construction in Section~\ref{sec:RIC coord}, each such standard map admits a rotational invariant curve (RIC) of rotation number $\omega$.
The sequence of $(p_j,q_j)$-periodic orbits accumulates onto this invariant curve as $j \to \infty$.
The function $\phi_\alpha$ gives the restricted dynamics on this RIC in the sense that the map $\id + \phi_\alpha: \T \to \T$ is a circle diffeomorphism that conjugates the rigid rotation $R_{\omega}$ to the restriction of the dynamics to this RIC.

In our construction, the function $\phi_\alpha$ will be chosen as a \textit{lacunary} Fourier series in the sense that, for all $|k| \geq q_1$, the $k$th Fourier coefficient $\hat \phi_k$ vanishes unless $k \in \pm\{ q_j, 2q_j\}$ for some $j \geq 1$.

The proof of Theorem~\ref{thm:main} uses Picard iterations to construct the function $\phi_\alpha$ as the fixed point of a certain contracting map.
Section~\ref{sec:outline} gives a more detailed outline of its proof and explains the proofs of Theorems~\ref{thm:main1}–\ref{thm:main beta} using Theorem~\ref{thm:main}.
After the preparatory Sections~\ref{sec:notation}–\ref{sec:RIC coord}, the main technical step is the construction of resonant normal forms (Theorem~\ref{thm:rnf} in Section~\ref{sec:rnf}), which are suited to the study of $(p,q)$-periodic orbits when $p/q$ approximates $\omega$ sufficiently well.
The proof of the main theorem, Theorem~\ref{thm:main}, is formally completed in Section~\ref{sec:equidistr} modulo the proof of Theorem~\ref{thm:rnf}, which is in turn proved in the last Section~\ref{sec:rnf proof}.

\section{Motivation and background}\label{sec:backgrounds}

The results in this paper are motivated by questions of \emph{spectral rigidity} in dynamical systems. 
Broadly speaking, one would like to understand to what extent a dynamical system is determined, up to an appropriate notion of conjugacy, by various spectral invariants attached to its periodic orbits.
In this paper, we focus on the \textit{deformational} rigidity, namely, we study if a dynamical system can be nontrivially deformed without changing certain spectral invariants.

\paragraph{Laplace spectral rigidity and its dynamical analogues}

The modern spectral rigidity problem goes back to Kac's famous question ``Can one hear the shape of a drum?''~\cite{Kac}.
Concretely, if two bounded domains in $\R^2$ have identical Laplace spectra (with appropriate boundary conditions), must the domains be congruent?
More generally, one asks how much geometric information about a Riemannian manifold $(M,g)$ is encoded in the spectrum of its Laplace operator.

It is now known that the answer to Kac's question, in full generality, is negative: for instance, Gordon, Webb and Wolpert~\cite{GWW}, using the method of Sunada~\cite{Sunada}, constructed isospectral nonisometric pairs of planar domains.
However, the question remains open for smooth strictly convex planar domains.

There are natural dynamical counterparts to the Laplace spectral rigidity.
In the setting of exact symplectic maps, a central object is the \emph{action spectrum} of the map, while in the setting of geodesic flows, the analogous object is the (marked) \emph{length spectrum} of closed geodesics.
Another dynamically relevant object for a general differentiable dynamical system is the \emph{Lyapunov spectrum}.

More precisely, given a differentiable dynamical system $T \colon M \to M$, we define its \emph{Lyapunov spectrum} to be the collection of Lyapunov exponents of all periodic orbits of $T$.
This can be regarded as an analogue in spirit to the length spectrum for geodesic flows.
This analogy appears explicitly in holomorphic dynamics, where McMullen~\cite{McMullen} considers holomorphic maps $f : \Delta \to \Delta$ on the unit disk and studies the quantities
\begin{equation*}
    \bigl|\log (f^q)'(z)\bigr|,
\end{equation*}
where $z$ belongs to a $q$-periodic cycle of $f$.
These log-multipliers are shown to exhibit properties closely analogous to the \emph{lengths} of closed geodesics on hyperbolic surfaces.

In the setting of expanding maps on the circle, a classical result of Shub and Sullivan~\cite{ShubSullivan} shows that the \emph{marked} Lyapunov spectrum determines the map up to \emph{smooth} conjugacy.
It is well known that every expanding map on the circle is topologically conjugate to the linear model $x \mapsto d x$ for some degree $d \in \Z^\times$.
Such a conjugacy provides a natural symbolic coding of periodic orbits, and hence a canonical \emph{marking} of the Lyapunov spectrum.
We also refer to the work of Drach and Kaloshin \cite{DK} for recent progress on an \emph{unmarked} version of this rigidity phenomenon and for a detailed historical review of the problem therein.

Our setting is quite different from these uniformly expanding or hyperbolic examples.
The standard map is a conservative dynamical system given by an \textit{exact symplectic twist map} defined on a cylinder, analogous to the dynamics of strictly convex billiards.
One of the goals of this paper is to exhibit a weak form of \emph{nonrigidity} for the Lyapunov spectrum in the sense that the Lyapunov exponents from certain sequences of periodic orbits of the standard map can be preserved under nontrivial deformations.

\paragraph{Symplectic maps and the action spectrum}
Another major aspect of this paper concerns the \emph{action spectrum}.
Let $I \subset \R$ be an interval and $\sF: \T \times I \to \T \times I$ be an exact symplectic map with respect to the symplectic $1$-form $y \du x$.
With a fixed choice of generating function, each periodic orbit of $\sF$ carries a well-defined \emph{symplectic action} obtained by evaluating the corresponding action functional over the orbit. 
Explicitly, this is obtained by summing the generating function over the points of the orbit.
The \emph{action spectrum} of $\sF$ is then the collection of symplectic actions of all periodic orbits of $\sF$.
In the case of the standard map, the action of a periodic orbit is given explicitly by the action functional~\eqref{eq:action functional} introduced in Section~\ref{sec:main results}.
Likewise, each periodic orbit can be associated with the eigenvalues of its linearized Poincar\'e map; in our two-dimensional setting this can be conveniently encoded by the \emph{Lyapunov exponent}~\eqref{eq:lyap exp}.

Guillemin and Melrose~\cite{GM} proposed to study such spectral data for exact symplectic maps and asked, in particular, whether the action spectrum (or jointly with the eigenvalues of the linearized Poincar\'e maps of periodic orbits) determines the map, at least locally and up to symplectic conjugation.
This is a symplectic analogue of Kac's question.

Our main result shows that the answer is negative if one only prescribes a \emph{subset} of the spectral data.
More precisely, Theorem~\ref{thm:main1} constructs a nontrivial analytic family of standard maps $(\sF_\tau)_{\tau\in(-\epsilon,\epsilon)}$ for which there exists an \emph{infinite} sequence of dynamically distinct periodic orbits whose actions or Lyapunov exponents~\eqref{eq:lyap exp} are independent of~$\tau$.

\paragraph{Length spectrum rigidity}

The action spectrum is closely related to the analogous notion of the \emph{length spectrum} of a Riemannian manifold.
Let $(M,g)$ be a compact Riemannian manifold.
Its \emph{length spectrum} is the collection of lengths of closed geodesics of $(M,g)$, while the \emph{marked length spectrum} records, for each free homotopy class, the length of the corresponding minimal closed geodesic (when it exists).
The classical length spectral rigidity problem asks whether the (marked) length spectrum determines $(M,g)$ up to isometry.

There has been substantial progress in settings where the dynamics is uniformly hyperbolic.
For example, for closed surfaces of negative curvature, the works of Croke and Otal~\cite{Croke,Otal} show that the marked length spectrum determines the metric up to isometry.
More recently, Guillarmou and Lefeuvre~\cite{GuillarmouLefeuvre} established \emph{local} marked length spectral rigidity for a broad class of Anosov manifolds in any dimension.
These results provide strong evidence that, for systems with uniformly hyperbolic dynamics, the marked length spectrum is a very robust invariant.

\paragraph{Spectral rigidity of convex billiards}
The dynamics of the billiard map associated with a strictly convex planar domain $\Omega \subset \R^2$ with smooth boundary provides a rather different setting from uniformly hyperbolic systems.
The billiard map is an exact symplectic twist map on a phase cylinder $\T \times [0, \pi]$ in suitable coordinates, and its action functional coincides with (minus) the Euclidean length of the corresponding broken geodesic in $\Omega$.
Consequently, up to sign, the action spectrum of the billiard map agrees with the length spectrum of periodic billiard trajectories, and one may ask whether this spectrum determines the domain $\Omega$.
We refer to the monograph of Kozlov and Treshchëv~\cite{kozlov1991billiards} for general background on billiard dynamics.

To describe a natural marking of the billiard spectrum, we recall the notion of \emph{rotation numbers} for twist maps.

Let $I \subset \R$ be an interval and $F: \T \times I \to \T \times I$ be continuous.
For a periodic orbit $(x,y)$ under $F$ with minimal period $q$, and a lift $\wt F$ of $F$ to the universal cover $\R \times I$, there exists an integer $p\in\Z$ such that
\begin{equation*}
  \tilde F^{q}(\tilde x,y)=(\tilde x+p,\,y).
\end{equation*}
Another lift $\tilde F'$ of $F$ satisfies $\tilde F'=\tilde F+(n,0)$ for some $n\in\Z$, hence
\begin{equation*}
  \tilde F'^{\,q}(\tilde x,y)=(\tilde x+p+qn,\, y),
\end{equation*}
so $p$ is determined modulo $q$ by the orbit.
There is therefore a unique choice of lift $\tilde F$ for which $p\in\{0,1,\dots,q-1\}$.
We refer to such an orbit as a \emph{$(p,q)$–periodic orbit}, or simply a \emph{$(p,q)$–orbit} of~$F$.

When $F$ is an exact symplectic twist map and $(p,q)$ is a coprime pair, a $(p,q)$–orbit that minimizes the symplectic action among all $(p,q)$–orbits is called a \textit{minimal orbit}, or sometimes an \textit{Aubry--Mather orbit}, due to its central role in Aubry--Mather theory.
The number $p/q$ is then called the \textit{rotation number} of this minimal orbit.\footnote{
However, the notion of rotation number does not apply to nonminimal periodic orbits in general: for instance, there may exist $(2p,2q)$–orbits that are not minimal.
In this paper we will therefore use the more flexible notion of $(p,q)$–orbits most of the time, without referring to rotation numbers.
}

In the billiard setting, one should keep in mind that a minimal orbit \emph{maximizes} the total Euclidean length of the broken geodesic representing the periodic billiard trajectory.
The \textit{marked length spectrum} of a billiard domain is then defined to be the map that associates to each rational number $p/q$ the length of a minimal orbit with rotation number $p/q$.

Several rigidity results are known in this billiard setting.  
In particular, a theorem of De Simoi, Kaloshin and Wei~\cite{DKW} proves \emph{deformational} length spectral rigidity within the class of nearly circular, $\Z^2$-symmetric domains.
Their proof relies on a detailed analysis of the first-order variation, under boundary deformations, of the lengths of a sequence of symmetric periodic orbits with rotation number $1/q$, $q \to \infty$, together with a delicate inversion of the resulting linearized operator.

However, the method of~\cite{DKW} strongly exploits the symmetry assumptions.
In particular, any deformation of the boundary which is ``antisymmetric'' to first order produces vanishing first-order variations of the lengths \textit{and} Lyapunov exponents of every symmetric periodic orbit, and thus cannot be detected at the linear level.
This leaves open the possibility that there might exist nontrivial deformations which preserve the lengths of infinitely many periodic trajectories.
Our construction for the standard map can be interpreted as an explicit realization of such a phenomenon in the symplectic twist-map setting.

\paragraph{Rigidity by spectral data of accumulating periodic orbits}

Besides the result of De Simoi, Kaloshin and Wei, we review some other circumstances where the spectral data of an accumulating sequence of periodic orbits is successfully exploited to yield rigidity results.

In the work of Huang, Kaloshin and Sorrentino~\cite{HKS}, they show that, for a generic billiard domain, the eigendata of minimal orbits can be recovered from the marked length spectrum.
In fact, one recovers the Lyapunov exponent of a minimal orbit of rotation number $p/q$ using only the lengths of minimal orbits whose rotation numbers converge to $p/q$.

We also mention work of De Simoi, Kaloshin and Leguil~\cite{DKL}, which considers dispersing billiard systems consisting of three scatterers with axial symmetry and shows that the marked length spectrum determines the geometry of the scatterers.
Here, the marking is given by the symbolic sequence of scatterers hit by the orbit over one period.
In that work, the authors consider the lengths of a sequence of periodic orbits converging to a homoclinic orbit of a particular $2$-periodic orbit.
Using the spectral data from this sequence, one reconstructs the Taylor series of the boundaries of the scatterers at the end-points of the $2$-periodic orbit.
This result has been extended by Osterman~\cite{Osterman}, who considers the same setting without symmetry assumptions and shows that the marked length spectrum determines the analytic conjugacy class of the billiard map.

Notably, these positive results mentioned above make use not of the \emph{full} spectrum, but only of a countable subset corresponding to a sequence of orbits accumulating on some invariant ``boundary'' or ``axis of symmetry'' of the dynamics.
On the other hand, our main result shows that, in the standard-map setting, spectral data associated with certain accumulating sequences of periodic orbits are \emph{not} deformationally rigid.
This motivates the following subtle question about which countable subsets of the spectrum suffice to determine the underlying dynamics.

\begin{question}\label{Q:accumulating isospec}
    Does there exist a nonisometric deformation of a smooth strictly convex planar billiard domain which preserves the lengths of infinitely many periodic orbits accumulating to the boundary, to a periodic orbit, or to an invariant curve?
\end{question}

There is a positive answer to a \emph{nondeformative} version of Question~\ref{Q:accumulating isospec}: Buhovsky and Kaloshin~\cite{BK} constructed a pair of billiard domains such that the marked length spectra of the two domains coincide along a sequence of rotation numbers $1/q_j$ with $q_j \to \infty$.

The methods in this current paper do not answer Question~\ref{Q:accumulating isospec}.
The key issue is that the proof of our main theorem, Theorem~\ref{thm:main}, relies crucially on the special structure of the standard map: one can explicitly construct potentials for a standard map admitting an invariant curve with a prescribed dynamics restricted on that curve.
In the setting of billiards, the standard tool for obtaining invariant curves is KAM theory, but the Liouville-type arithmetic condition~\eqref{eq:main Liouv condition} required in our construction is incompatible with the usual Brjuno-type conditions needed for KAM theory.

\paragraph{Mather's beta function}

Another spectral invariant often considered is Mather's beta function and its infinite jet at particular points.
Given a general exact symplectic twist map and a choice of generating function, Mather's beta function $\beta$ is defined, as in~\eqref{eq:beta func intro}, to be a continuous function whose value at a rational number $p/q$ is $1/q$ times the action of the minimal orbits of rotation number $p/q$.
It is known that $\beta$ is convex on $\Q$ and therefore extends continuously, by density, to an interval of $\R$ called the \emph{rotation interval} of the map.

If Mather's beta function admits an asymptotic expansion at $\omega$ of the form
\begin{equation}\label{eq:inf jets}
    \beta(\alpha) \sim \sum_{k \geq 0} c_k (\alpha - \omega)^k,
\end{equation}
we call the sequence $(c_k)$ the \emph{infinite jet} of $\beta$ at $\omega$, by analogy with the jets of differentiable functions.
More precisely, the asymptotics~\eqref{eq:inf jets} means that for any $N \in \N$ and any sequence $\alpha_j \to \omega$,
\begin{equation*}
    \frac{1}{(\alpha_j - \omega)^{N}}
    \bigg(\beta(\alpha_j) -  \sum_{k = 0}^{N} c_k (\alpha_j - \omega)^k\bigg) \to 0
    \quad \text{as } j \to \infty.
\end{equation*}
We refer to~\cite{Siburg,MatherForni} for more details on Mather's beta function.

In the billiard setting, Mather's beta function admits an infinite jet at $0$, which is commonly known as the Marvizi--Melrose invariants~\cite{MarviziMelrose}.
Hence, if a deformation of a billiard domain preserves a sequence of orbits accumulating to the boundary, then this deformation preserves all Marvizi--Melrose invariants.
In particular, Buhovsky and Kaloshin~\cite{BK} used this observation to construct nonisometric pairs of domains with identical Marvizi--Melrose invariants.

We recall that a \emph{rotational invariant curve (RIC)} of a map $F: \T \times I \to \T \times I$ is the graph
\begin{equation*}
  C=\{(x,h(x)):\ x\in\T\},
\end{equation*}
of a continuous function $h:\T\to\R$, such that $C$ is invariant under $F$.
Using Birkhoff normal forms, one can show that if the symplectic map is sufficiently smooth and it has an RIC whose restricted dynamics conjugates to a rigid rotation $x \mapsto x + \omega$ by a \emph{Diophantine} number (i.e.\ $\omega$ is irrational and $\sup_{p\in\Z} |q\omega-p|^{-1}$ grows at most polynomially in $q$), then Mather's beta function admits an infinite jet at $\omega$.
Thus, if a deformation of the system preserves a sequence of orbits accumulating to such an invariant curve, then this deformation preserves the infinite jet of Mather's beta function at~$\omega$.

This leads to the following question, which is in some sense weaker than Question~\ref{Q:accumulating isospec}.

\begin{question}\label{Q:jet isospec}
    Does there exist a nonisometric deformation of a smooth strictly convex planar billiard domain which preserves the infinite jet of Mather's beta function (when it exists) at some $\omega \in \R$?
\end{question}

We remark that Mather \cite{Mather1990} proved that Mather's beta function is differentiable at every \emph{irrational} $\omega$ and generically fails to be differentiable at rational $\omega$.
In fact, Mather's beta function is differentiable at $p/q$ if and only if there exists an RIC conjugate to a rigid rotation by $p/q$.
However, it is not clear in general whether Mather's beta function has an infinite jet at a frequency $\omega$ that is irrational but not Diophantine.
Hence, our result Theorem~\ref{thm:main beta} does not provide information relevant to Question~\ref{Q:jet isospec}, due to the strong arithmetic condition \eqref{eq:main Liouv condition}.
This motivates a further regularity question for Mather's beta function.

\begin{question}\label{Q:beta-jet-existence}
    Under what conditions does Mather's beta function admit an infinite jet at a given $\omega \in \R$?
\end{question}

\section{An outline of the proof}\label{sec:outline}

We first explain how Theorem~\ref{thm:main1} and Theorem~\ref{thm:main beta} follow from Theorem~\ref{thm:main}.

To obtain Theorem~\ref{thm:main1} from Theorem~\ref{thm:main}, we consider a one-parameter analytic family $\alpha^\tau = (\alpha_j^\tau)_{j\ge1}$ of sequences such that $\alpha_j^\tau$ is constant in $\tau$ for all $j\ge2$, and only the first parameter $\alpha_1^\tau$ varies.
Applying Theorem~\ref{thm:main} to this family with a choice of $\phi_\alpha$ keeping the averaged quantities $f_1^*$ and $(\phi_\alpha \dot\phi_\alpha^+)^*$ constant in $\alpha$, we deduce from the explicit formula~\eqref{eq:main thm action} that the action (or Lyapunov exponent) of each $(p_j,q_j)$–orbit with $j\ge2$ is constant in~$\tau$.
Since the pairs $(p_j,q_j)$ are coprime and the $q_j$ are strictly increasing, these periodic orbits are dynamically distinct, yielding Theorem~\ref{thm:main1}.

Finally, Theorem~\ref{thm:main beta} follows by combining Theorem~\ref{thm:main} with the definition of Mather's beta function.
Indeed, by Theorem~\ref{thm:main}, the sequence of $(p_j, q_j)$-orbits can be chosen as minimal orbits. Hence, they control the value of Mather's beta function at $p_j/q_j$. For a suitable choice of the family $\alpha^\tau$ as above (again varying only $\alpha_1^\tau$), the values of $\beta_{V_\tau}(p_j/q_j)$ remain constant in~$\tau$ for the whole subsequence $p_j/q_j\to\omega$, which is precisely the statement of Theorem~\ref{thm:main beta}.

We now outline the proof of the main result, Theorem~\ref{thm:main}.

\medskip
\noindent\textbf{Step 1. Parameterization by a rotational invariant curve.}
The starting point is the elementary observation that \emph{a rotational invariant curve (RIC) with rotation number $\omega$ determines the standard map}. 
More precisely, suppose there exist a circle diffeomorphism $u:\T\to\T$ and a smooth map $v:\T\to\R$ such that
\begin{equation*}
    \sF(u(t), v(t)) = (u(t+\omega),\, v(t+\omega))
\end{equation*}
for all $t\in\T$.
Writing $u = \id + \phi$ with $\phi:\T\to\R$, and expressing the standard map in the coordinates $(u,v)$, one obtains an explicit relation
\begin{equation}\label{eq:RIC potential outline}
    \dot V\bigl(t+\phi(t)\bigr) 
    \;=\; \phi(t+\omega)-2\phi(t)+\phi(t-\omega)
\end{equation}
between the potential $V$ and the function $\phi$ describing the RIC (see Section~\ref{sec:RIC coord} for details).
In particular, once $\phi$ and $\omega$ are fixed, the potential $V$ is determined by \eqref{eq:RIC potential outline} up to an additive constant.

Thus, for fixed $\omega$, the class of standard maps admitting an RIC of rotation number $\omega$ and whose dynamics on the RIC is conjugate to the rigid rotation $R_\omega$ is parametrized by the function $\phi:\T\to\R$.
In the proof of Theorem~\ref{thm:main} we work inside this class: we first construct a family of functions $\phi_\alpha$, and then define the corresponding potentials $V_\alpha$ by solving~\eqref{eq:RIC potential outline}, choosing the additive constant appropriately.
To facilitate computation, we first consider coordinates adapted to this RIC, where the standard map writes 
\begin{equation}\label{eq:outline adapted coord}
    \begin{pmatrix}
        x \\ y 
    \end{pmatrix}
    \mapsto
    \begin{pmatrix}
        x + \omega + y + f(x,y) \\ y + g(x,y)
    \end{pmatrix}
\end{equation}
where the RIC corresponds to $\{y= 0\}$ and the functions $f$ and $g$ depend implicitly on $(\omega, \phi)$.

\medskip
\noindent\textbf{Step 2. The heuristic and a $q$–resonant normal form near the RIC}
Fix a suitable $\omega\in\R\setminus\Q$.
We choose a sequence of coprime integer pairs $(p_j,q_j)$ with $q_j>0$ and $p_j/q_j\to\omega$ as $j\to\infty$, subject to the arithmetic condition \eqref{eq:main dioph condition} in Theorem~\ref{thm:main}.

The guiding heuristic is that, for each such rational approximation $p_j/q_j$, the \emph{$q_j$-th Fourier mode of $\phi$ should dominate the spectral data of $(p_j,q_j)$–periodic orbits} when $p_j/q_j$ approximates $\omega$ sufficiently well.
Here, by \emph{spectral data} we mean either the symplectic actions or the Lyapunov exponents of these periodic orbits.
The idea is that if $q_j$ and $q_k$ are very far apart, then the contribution of the $q_k$–th Fourier mode to the dynamics of $(p_j,q_j)$–orbits should be averaged out because the corresponding periods are incompatible.
This viewpoint is inspired by resonant averaging techniques in Hamiltonian dynamics (see, for example,~\cite{BounemouraNiederman}).
The heuristic above is made precise by a resonant normal-form construction described in Section~\ref{sec:rnf}.
Adapting the method of~\cite{Martín_2016} to the standard map, we show that, in a complex neighbourhood of the line $\{y = p/q - \omega\}$, the standard map \eqref{eq:outline adapted coord} can be conjugated to a \emph{$q$–resonant normal form}.
More concretely, in these $q$-dependent coordinates the map is conjugate to a map (defined after re-centering $y$-coordinate at $p/q - \omega$) of the form
\begin{equation}\label{eq:outline rnf}
    \begin{pmatrix}
        x \\ y 
    \end{pmatrix}
    \mapsto
    \begin{pmatrix}
        x + p/q + y \\ y + g_q(x,y)
    \end{pmatrix},
\end{equation}
where $g_q(x,0)$ is \emph{almost} $1/q$–periodic in $x$.

We consider the (dominant) \emph{resonant part} $\langle g_q\rangle_q^N$ of $g_q$, obtained by truncating its Fourier series to modes of order at most $N$ that are multiples of~$q$.
Theorem~\ref{thm:rnf} shows that, for $N$ not too large and $|\omega - p/q|$ small enough, the resonant part $\langle g_q\rangle_q^N$ can be well-approximated by the resonant part of an explicit expression in terms of the leading Taylor expansion $f_1(x)$ of $f(x,y)$ at $y = 0$, which in turn depends rather explicitly on~$\phi$.
By further analysis in Section \ref{sec:picard}, we formalize the heuristic in Lemma~\ref{lem:CP}: the dynamics governing the $(p,q)$–periodic orbits depend mainly on the Fourier coefficients of $\phi$ at orders $kq$ and only weakly on the remaining coefficients.

\medskip
\noindent\textbf{Step 3. A fixed-point argument for the resonant data.}
In Section~\ref{sec:picard} we use a Picard iteration scheme to construct $\phi_\alpha$ so that the resonant normal forms $g_{q_j}$ associated with the approximants $(p_j,q_j)$ have a prescribed structure.

Using the estimates from Theorem~\ref{thm:rnf}, we show that for each $q$ the map $\phi \mapsto \langle \partial_x g_q\rangle_q^N$ is close, in a suitable analytic norm and with appropriate scaling, to the map $\phi \mapsto \langle \dddot\phi\rangle_q^N$ (the resonant part of the third derivative of $\phi$), provided $q$ is large.
Because the periods $q_j$ grow sufficiently fast, the influence of the $q_k$–th resonant block on the $q_j$–th resonant block is very small when $k\neq j$.
This near-decoupling allows us to prescribe the resonant parts $\langle g_{q_j}\rangle_{q_j}^{N_j}$ \emph{simultaneously} for the infinite sequence $(p_j,q_j)$ by solving a fixed-point equation for $\phi$.

This fixed-point problem is encoded in a nonlinear operator on a Banach space of analytic functions, and we prove that this operator is a contraction.
The resulting contraction mapping theorem is formulated as Theorem~\ref{thm:fx pt} and yields an \textit{inversion} of the map
\begin{equation}\label{eq:outline rnf to phi}
    \phi \mapsto \Big(\langle \partial_x g_{q_j}(\cdot, 0) \rangle_{q_j}^{N_j} \Big)_{j \geq 1}.
\end{equation}
In other words, Theorem~\ref{thm:fx pt} allows us to control the resonant parts of the $q_j$–resonant normal forms $g_{q_j}$ essentially independently by tuning~$\phi$.

\medskip
\noindent\textbf{Step 4. Construction of periodic orbits and computation of spectral data.}
In Section~\ref{sec:equidistr} we apply the general fixed-point scheme proved in Section~\ref{sec:picard} to construct and analyse $(p_j,q_j)$–periodic orbits explicitly.
For a sequence of the parameters $\sigma_j$, we apply the inverse of \eqref{eq:outline rnf to phi} to the choice $N_j = 2q_j$ and 
\begin{equation*}
    \langle \partial_x g_{q_j}(\cdot, 0) \rangle_{q_j}^{N_j} \;=\; \sigma_j\,\eu^{-2\pi q_j}\bigl(p_j/q_j - \omega\bigr)^2 \cos(2\pi q_j x).
\end{equation*}
Using $g_{q_j} \approx \langle g_{q_j}\rangle_{q_j}^{N_j}$, one can construct a $(p_j,q_j)$–periodic orbit that closely shadows the \textit{equidistributed} orbits $(k p_j/q_j, 0)_{k=0,\ldots,q_j-1}$.

By analysing directly the Jacobian matrix of the normal form~\eqref{eq:outline rnf} along this periodic orbit, we obtain an explicit expression for the Lyapunov exponent in terms of the parameter~$\sigma_j$.
Furthermore, using the near-equidistribution property of the orbit together with the structure of the normal form, we derive a corresponding estimate for the symplectic action of the orbit, again expressed in terms of~$\sigma_j$.
In the final step, we perform a second Picard-type iteration, this time in a suitable Banach of sequences $(\sigma_j)$, to adjust the parameters $\sigma_j$ so that the Lyapunov exponent (or, in the second part of Theorem~\ref{thm:main}, the action) of the $j$–th periodic orbit attains the prescribed value.

\section{Notations and basic definitions}\label{sec:notation}

For $b>0$, let $\D_b$ denote the open disc in $\C$ of radius $b$ centred at the origin.  
We write $\T := \R/\Z$ for the real $1$-torus. For $a>0$, define the complex neighbourhood
\begin{equation*}
  \T_a := \{\, z \in \C/\Z : |\Im z| < a \,\}.
\end{equation*}
If $x\in\T_a$ and $z\in\C$, we write $x+z$ for the class in $\C/\Z$ obtained by adding representatives of $x$ and $z$. In particular, adding a \emph{real} number preserves~$\T_a$.

For $\omega\in\R$, the \emph{rigid rotation by $\omega$} is
\begin{equation*}
    \begin{aligned}
        R_\omega : \T_a &\longrightarrow \T_a, \\ 
        x &\longmapsto x + \omega .
    \end{aligned}
\end{equation*}
For a function $u: \T_a \to \C$, we shall denote by
\begin{equation*}
    u^\pm := u \circ R_{\pm \omega}
\end{equation*}
the translations of $u$ by an $\pm\omega$-rotation.

We write $A\lesssim B$ to mean that $A \le c\,B$ for some universal constant $c>0$ (independent of any parameters involved in $A$ and $B$). 

For a function $u:\T_a \to \C$, we denote by $\dot u$, $\ddot u$, and $\dddot u$ its first, second, and third derivatives with respect to~$x$.

\paragraph{Analytic norms.}

Let $u:\T_a \to \C$ be analytic. Its supremum norm on $\T_a$ is
\begin{equation*}
    |u|_a := \sup\{\,|u(x)| : x \in \T_a\,\}.
\end{equation*}
Expanding $u$ into its Fourier series
\begin{equation*}
  u(x) \;=\; \sum_{k\in\Z} \wh u_k\, \eu^{2\pi \iu k x},
  \qquad
  \wh u_k \;=\; \int_{\T} u(x)\,\eu^{-2\pi \iu k x}\,\du x,
\end{equation*}
(where $\du x$ denotes the normalized Lebesgue measure on $\T$), we define the \emph{Fourier-weighted norm}
\begin{equation*}
  \|u\|_{a} := \sum_{k \in \Z} |\wh u_k|\, \eu^{2\pi a |k|}.
\end{equation*}
The spaces
\begin{equation*}
  \{u:\T_a\to\C : |u|_a<\infty\}
  \quad\text{and}\quad
  \{u:\T_a\to\C : \|u\|_a<\infty\}
\end{equation*}
are Banach spaces. Moreover, the $|\cdot|_a$-space is a Banach algebra, i.e.\ $|fg|_a \le |f|_a\,|g|_a$, and (see Lemma~\ref{lem:banach-alg} in the Appendix) the $\|\cdot\|_a$-space is also a Banach algebra.
We also introduce 
\begin{equation*}
    \|u \|_a' := \|\dddot u\|_a.
\end{equation*}

Since we often take $a = 1$, we suppress the subscript in this case and write simply
\begin{equation*}
    |u| := |u|_1, \qquad  \|u\| := \|u\|_1 \qquad \text{and} \qquad \|u\|' := \|u \|_1'. 
\end{equation*}

For $f : \T_a \times \D_b \to \C$ analytic, set
\begin{equation*}
    |f|_{a,b} := \sup \{\,|f(x,y)| : x\in \T_a,\ y \in \D_b\,\}.
\end{equation*}

\paragraph{Resonant projections in Fourier space.}

Let $q\in\N$ and $u: \T_a \to \C$ be analytic. We write $\crochet{u}_q$ for the projection of $u$ onto the $q\Z$-modes (the $1/q$-periodization), and $\curly{u}_q := u - \crochet{u}_q$ for the complementary part:
\begin{equation*}
    \crochet{u}_q(x) \;=\; \sum_{\substack{k \in \Z \\ q \mid k}} \wh u_k \,\eu^{2\pi \iu k x}
    \;=\; \frac{1}{q}\sum_{j = 0 }^{q-1} u\Bigl(x + \frac{j}{q}\Bigr),
\end{equation*}
\begin{equation*}
    \curly{u}_q(x) \;=\; \sum_{\substack{k \in \Z \\ q \nmid k}} \wh u_k \,\eu^{2\pi \iu k x}.
\end{equation*}
For $N\in\N$, define the low/high-frequency truncations
\begin{equation*}
    \crochet{u}^{N}(x) := \sum_{\substack{k \in \Z \\ |k| \le N}} \wh u_k \,\eu^{2\pi \iu k x},
    \qquad
    \crochet{u}^{>N}(x) := \sum_{\substack{k \in \Z \\ |k| > N}} \wh u_k \,\eu^{2\pi \iu k x},
\end{equation*}
and combine them as
\begin{equation*}
    \crochet{u}_q^{N}(x) := \sum_{\substack{k \in \Z \\ q \mid k,\ |k| \le N}} \wh u_k \,\eu^{2\pi \iu k x}.
\end{equation*}
We also use
\begin{equation*}
    u^* := \int_\T u(x)\,\du x, 
    \qquad 
    u^\bullet := u - u^*,
\end{equation*}
for the average and zero-average parts of $u$.

For $f: \T_a \times \D_b \to \C$ analytic and fixed $y\in\D_b$, let $\wh f(y)_k$ be the $k$th Fourier coefficient of $x\mapsto f(x,y)$:
\begin{equation*}
  \wh f(y)_k \;=\; \int_{\T} f(x,y)\,\eu^{-2\pi \iu k x}\,\du x.
\end{equation*}
We apply the same resonant projection notations $\crochet{f}_q$, $\crochet{f}^N$, $f^*$, $f^\bullet$ etc., in the $x$-variable (with $y$ fixed), e.g.
\begin{equation*}
    \crochet{f}_q^N(x,y) \;=\; \sum_{\substack{k \in \Z \\ q \mid k,\ |k| \le N}} \wh f(y)_k \,\eu^{2\pi \iu k x}; \qquad f^*(y) = \int_\T f(x,y) \du x .
\end{equation*}
\paragraph{Exact symplectic maps and actions.}
We recall that a $\CC^1$ map $\sF: \T \times \R \to \T \times \R$ is said to be \textit{exact symplectic} if the differential $1$-form $\sF ^*(y \du x) - y\du x$ is exact, i.e. 
\begin{equation}\label{eq:exact symp maps}
    \sF ^*(y \du x) - y\du x = \du \sS
\end{equation}
for some function $\sS$ on $\T \times \R$.
The function $\sS$, which is unique up to an additive constant, is called a \textit{generating function} of $\mathscr{F}$.
In particular, taking the exterior derivative of \eqref{eq:exact symp maps} shows that $\sF$ preserves the canonical area form $\du x \du y$.
Let a generating function $\sS$ be fixed.
Let $(\sF^k(x_0, y_0))_{k = 0}^{q-1}$ be a $q$-periodic orbit of $\sF$.
Then the \textit{(symplectic) action} of this orbit is defined as 
\begin{equation*}
    \sum_{k = 0}^{q-1} \sS(\sF^k(x_0, y_0)).
\end{equation*}
Denote $(x', y') = \sF(x,y)$. 
If $x'$ is a monotone function in $x$ for each fixed $y$, then $\sF$ is called a \textit{twist map} and, in this case, the lift of the map $\T \times \R \ni (x,y) \mapsto (x, x') \in \T^2$ is a diffeomorphism $\R^2 \to \R^2$.
Hence, one can lift $\sS: \T \times \R \to \R$ to $\R^2$ and then consider $\sS$ as a function in the $(x, x')$-variables. 
The $(x,x')$-coordinates are sometimes called the \textit{Lagrangian coordinates}.
If $\sF$ is the standard map associated with a potential $V(x)$, then its generating function can be given by 
\begin{equation}\label{eq:std map gen func}
    \sS(x,x') = \frac{1}{2} (x - x')^2 + V(x).
\end{equation}
We recall that the notion of $(p,q)$-periodic orbits and minimal orbits have been defined in Section~\ref{sec:backgrounds}.
Given a coprime pair $(p,q)$, Birkhoff proved that there exists at least two dynamically distinct $(p,q)$-orbits.
Besides a minimal orbit with rotation number $p/q$, another $(p,q)$-periodic orbit can be obtained using a `mountain pass' argument (see for e.g., \cite[Ch. 9]{Katok_Hasselblatt_1995}).
It is customary to call this additional $(p,q)$-periodic orbit a \textit{minimax} orbit.

\section{Invariant curves and adapted coordinates}\label{sec:RIC coord}

\paragraph{Recovering the potential from an RIC.}
Let $\sF$ be the standard map \eqref{eq:std map} associated with a potential $V$.
A special feature of~\eqref{eq:std map} is that it is uniquely determined by any one of its rotational invariant curves that conjugate to rigid rotations (if it has any). 
We explain this now.

Fix $\omega\in\R$. 
Suppose $\sF$ has a rotational invariant curve (RIC) in the sense defined in Section~\ref{sec:backgrounds} and assume furthermore that the dynamics on $C$ is conjugate to the rigid rotation $R_\omega: x \mapsto x + \omega$. 
Concretely, there exists a circle diffeomorphism $u:\T\to\T$ and a map $v:\T\to\R$ such that the map $t\mapsto (u(t),v(t))$ parametrizes $C$ and
\begin{equation}\label{eq:RIC}
    \sF(u(t), v(t)) \;=\; \bigl(u(t+\omega),\, v(t+\omega)\bigr)\quad\text{for all }t\in\T.
\end{equation}

Write $u(t)=t+\phi(t)$ with a $\CC^1$ function $\phi:\T\to\R$; then $u$ is an orientation preserving circle diffeomorphism iff $\dot\phi>-1$. 
By choosing a lift, we consider $u$ and $v$ as functions on $\R$ with $ u(t+1)= u(t)+1$ and $v(t + 1) = v(t)$, and interpret~\eqref{eq:RIC} using lifts so that equalities are in $\R$ for the $x$-component. From~\eqref{eq:std map} and~\eqref{eq:RIC},
\begin{equation}\label{eq:v-u-relation}
  v(t+\omega) \;=\;  u(t+\omega)- u(t), 
  \qquad 
  \dot V\bigl(u(t)\bigr) \;=\; v(t+\omega)-v(t).
\end{equation}

Substituting $u=\id+\phi$ and eliminating $v$, we obtain
\begin{equation}\label{eq:RIC potential}
    \dot V\bigl(t+\phi(t)\bigr) \;=\; \phi(t+\omega)-2\phi(t)+\phi(t-\omega)\,.
\end{equation}
Thus, the pair $(\omega,\phi)$ determines $\dot V$ on $\T$, and hence determines $V$ up to an additive constant.

\paragraph{Existence of $V$ for a given conjugacy.}
Let $u=\id+\phi$ with $\dot\phi>-1$ and set $\phi^\pm:=\phi\circ R_{\pm\omega}$ as in Section \ref{sec:notation}. Define
\begin{equation*}
  w(x)\;:=\; \bigl(\phi^+ -2\phi + \phi^-\bigr)\circ u^{-1}(x).
\end{equation*}
We claim $\int_\T w(x)\,\du x=0$, which implies that there exists $V:\T\to\R$ (unique up to an additive constant) such that $\dot V=w$ and hence~\eqref{eq:RIC potential} holds.

Indeed, using the change of variables $x=u(t)$ and that $\int_\T(\phi^+ -2\phi+\phi^-)\,\du t=0$ and $\int_\T 2 \phi(t) \dot\phi(t) \du t = 0$,
\begin{align*}
\int_\T w(x)\,\du x
&= \int_\T \bigl(\phi^+(t) -2\phi(t)+\phi^-(t)\bigr)\,(1+\dot\phi(t))\,\du t \\
&= \int_\T \bigl(\phi^+(t) -2\phi(t)+\phi^-(t)\bigr)\,\dot\phi(t)\,\du t \\
&= \int_\T \phi(t+\omega)\,\dot\phi(t)\,\du t \;+\; \int_\T \phi(t - \omega)\,\dot\phi(t)\,\du t \\
&= -\int_\T \dot\phi(t+\omega)\,\phi(t)\,\du t \;+\; \int_\T \phi(t)\,\dot\phi(t + \omega)\,\du t\;=\; 0,
\end{align*}
where in the last line we used an integration by parts on the first term.

We summarize:

\begin{lemma}\label{lem:std map with RIC}
For any $\omega \in \R$ and any $\phi:\T\to\R$ with $\dot \phi>-1$, there exists a potential $V$ (unique up to an additive constant) such that the standard map associated with $V$ admits an RIC whose restricted dynamics conjugates to $R_\omega$ by the diffeomorphism  $u=\id+\phi$.
\end{lemma}

\paragraph{Standard maps in adapted coordinates.}
Let $\sF$ be the unique standard map having an RIC whose restricted dynamics conjugates to $R_\omega$ by $\id + \phi$ as in Lemma~\ref{lem:std map with RIC}.

To facilitate later calculations, we introduce coordinates on $\T \times \R$ adapted to this RIC.
We parametrize the RIC by $t \mapsto (u(t), v(t))$ as in \eqref{eq:RIC}.
Consider the area-preserving change of coordinates $\sH$ that `straightens' the invariant curve,
\begin{equation}\label{eq:RIC coord}
    \sH(x,y) = \Big(u(x), v(x) + \frac{y}{1 + \dot \phi(x)}\Big).
\end{equation} 
Hence $\sF \circ \sH(x,0) = \sH(x + \omega, 0)$.
Define $F := \sH\inv \circ \sF \circ \sH$ and functions $f, g: \T \times \R \to \R$ so that
\begin{equation}\label{eq:F RIC}
 F(x,y) = (x + \omega + y + f(x,y),\ y + g(x,y)).    
\end{equation}
A straightforward substitution of $v(\cdot + \omega) = u(\cdot + \omega) - u(\cdot)$ and $u = \id + \phi$ yields a set of equations characterizing $f$ and $g$:
\begin{equation}\label{eq:fg}
    \begin{cases}    
        y + f(x,y)  = \frac{y}{1+ \dot\phi(x)} + \phi(x+\omega) - \phi(x + \omega + y + f(x,y)) \\
        y+g(x,y) = \big(y + f(x,y) + \phi(x + y + f(x,y)) - \phi(x)\big)\big(1 + \dot\phi(x + \omega + y + f(x,y))\big).
    \end{cases}
\end{equation}
Setting $y = 0$ in \eqref{eq:fg} gives $f(x,0) = g(x,0) = 0$, as expected since $F(x,0) = (x + \omega, 0)$.

In the sequel, we will call $F$ the \textit{standard map associated with $(\omega, \phi)$ in the adapted coordinates}.

\paragraph{Actions of periodic orbits}
Next, we consider the generating function of the standard map $ F =\sH\inv \circ \sF \circ \sH$ associated with $(\omega, \phi)$ in the adapted coordinates as defined in the preceding paragraph.
Since $\sH$ and $\sF$ are area-preserving, the map $F$ is area preserving.
Since $F$ preserves $\T \times \{0\}$, the integral of the differential 1-form $F^*(y \du x) - y \du x$ over the circle $\T \times \{0\}$ vanishes.
Thus $F^*(y \du x) - y \du x$ belongs to the zero cohomology class by Poincar\'e duality, and is therefore exact.
It follows that $F$ is an exact symplectic map with respect to the symplectic 1-form $y \du x$.
Let us now find a generating function of $F$ in terms of $\phi$.

Note that the differential $1$-form $\sH^*(y \du x) - y \du x$ is closed by area preservation.
Since the de Rham cohomology group $H^1(\T \times \R)$ is generated by the class of $\du x$ where $x$ parametrizes $\T$, we can write 
\begin{equation}\label{eq:pullback 1 form}
    \sH^*(y \du x) = y \du x + \lambda \du x + \du \Sigma
\end{equation}
for some $\lambda \in \R$ and $\Sigma: \T \times \R \to \R$.
Let $\sS$ be a generating function of the standard map $\sF$.
Using $\sF^*(y \du x) - y \du x = \du \sS$ and \eqref{eq:pullback 1 form}, we compute:
\begin{equation}\label{eq:F*ydx}
    \begin{aligned}
        F^*(y \du x)
        =& F^*(\sH^*(y \du x) - (\lambda \du x + \du \Sigma)) \\
        =& \sH^* \sF^* (y \du x) - F^*(\lambda \du x + \du \Sigma) \\
        =& \sH^* (y \du x) + \sH^* (\du \sS) - F^*(\lambda \du x + \du \Sigma) \\
        =& y \du x + \lambda \du x + \du \Sigma + \sH^* (\du \sS) - F^*(\lambda \du x + \du \Sigma)
    \end{aligned}
\end{equation}
whereas $\lambda$ can be computed by integrating $\sH^*(y \du x)$ over the circle $\T \times \{0\}$ and using \eqref{eq:RIC coord}, \eqref{eq:v-u-relation} and the relation $u = \id + \phi$:
\begin{equation*}
    \begin{aligned}
        \lambda =& \int_\T \lambda \du x = \int_{\T \times \{0\}}  \sH^*(y \du x)
        = \int_\T v(x) (1 + \dot \phi(x)) \du x \\
        =& \int_\T (\omega + \phi(x) - \phi(x - \omega)) (1 + \dot \phi(x)) \du x 
        =  \omega - \int_\T \phi(x - \omega)\dot \phi(x) \du x
        =  \omega - \int_\T\phi \dot\phi^+.
    \end{aligned}
\end{equation*}
Substituting into \eqref{eq:F*ydx}, we obtain a generating function $S$ of $F$:
\begin{equation}\label{eq:RIC gen func}
    S = \sS \circ \sH + \bigg(\omega - \int_\T\phi \dot\phi^+\bigg)(x - x\circ F) + (\Sigma - \Sigma \circ F)
\end{equation}
where $x - x\circ F$ is interpreted using a lift of $F$. By \eqref{eq:F RIC}, we have simply $x - x\circ F = -(\omega + y + f(x,y))$.
Suppose $(F^k(x_0, y_0))_{k = 0}^{q-1}$ is a $(p, q)$-periodic orbit of $F$.
Then the sum of $x - x\circ F$ over this orbit is exactly $-p$ by the definition of $(p,q)$-periodic orbits (see Section \ref{sec:notation}), and the sum of the third term in \eqref{eq:RIC gen func} vanishes by a telescopic sum.
Hence, the symplectic action of this orbit with respect to the generating function \eqref{eq:RIC gen func} writes 
\begin{equation}\label{eq:RIC action}
    \sum_{k = 0}^{q-1} S \circ F^k(x_0, y_0) = -\bigg(\omega - \int_\T\phi \dot\phi^+\bigg)p + \sum_{k = 0}^{q-1} \sS \circ \sH  \circ F^k(x_0, y_0).
\end{equation}
Observe that the term $\sum_{k = 0}^{q-1} \sS \circ \sH (F^k(x_0, y_0))$ is the symplectic action of the corresponding periodic orbit under the original standard map $\sF$, with a generating function $\sS$ in the form of \eqref{eq:std map gen func} for a potential $V$ satisfying \eqref{eq:RIC potential}.
Looking forward, we will actually evaluate via \eqref{eq:RIC action} the $\sS$-action $\sum_{k = 0}^{q-1} \sS \circ \sH (F^k(x_0, y_0))$ in terms of the $S$-action $\sum_{k = 0}^{q-1} S(F^k(x_0, y_0))$ associated with the map $F$ in adapted coordinates.
To this end, we now derive a separate formula for $S$.

By taking the derivative of the first line of \eqref{eq:fg} in $y$, it is easy to see that $\partial f / \partial y > -1 $ as long as $\dot\phi > -1$ over $\T$.
Therefore, the map $F$ is a twist map.
After lifting $S$ to $\R^2$ and transform to the Lagrangian coordinates $(x, x')$, we deduce from $F^* (y \du x) - y \du x = \du S$ that, in a neighbourhood of $\{y = 0\}$,
\begin{equation}\label{eq:S def}
    F(x,y) = (x', y') \quad \text{if and only if} \quad \frac{\partial S}{\partial x}(x,x') = -y \quad \text{and} \quad \frac{\partial S}{\partial x'}(x,x') = y'.
\end{equation}
Using \eqref{eq:S def}, and regarding $x'(x,y)$ as a function in terms of $(x,y)$, we can calculate the derivatives of $S(x,x'(x,y))$ in $x$ and $y$ as follows.
\begin{equation*}
    \frac{\partial S}{\partial x}(x,y) = \frac{\partial S}{\partial x} + \frac{\partial S}{\partial x'} \frac{\partial x'}{\partial x} = y' \frac{\partial x'}{\partial x} - y 
    \quad \text{and} \quad 
    \frac{\partial S}{\partial y}(x,y) = \frac{\partial S}{\partial x'} \frac{\partial x'}{\partial y} = y' \frac{\partial x'}{\partial y}.
\end{equation*}
Substituting $x' = x + \omega + y + f(x,y)$ and $y' = y + g(x,y)$, we then obtain explicitly 
\begin{equation}\label{eq:RIC gen func-fg}
    \frac{\partial S}{\partial x }(x,y) = (y + g(x,y))\Big( 1 + \frac{\partial f}{\partial x}(x,y) \Big) - y; \qquad  \frac{\partial S}{\partial y }(x,y) = (y + g(x,y))\Big( 1 + \frac{\partial f}{\partial y}(x,y) \Big).
\end{equation}
Hence, the $F$-generating function $S$ can be obtained by integrating the second equation of \eqref{eq:RIC gen func-fg} in $y$.

In particular, since $g(x,0) = 0$, we see from the first equation in \eqref{eq:RIC gen func-fg} that $S$ is constant along $\{y=0\}$.
By adjusting the additive constant of the potential $V$ solving \eqref{eq:RIC potential}, we may adjust $\sS$ (and hence $S$) by an additive constant so that \eqref{eq:RIC gen func} vanishes identically on the RIC.
This choice fixes $\sS$ and $S$ for a given pair $(\omega,\phi)$; we call $\sS$ the \emph{normalized generating function} associated with $(\omega,\phi)$ in the original coordinates, and $S$ the corresponding normalized generating function in the adapted coordinates.

\section{Resonant coordinates}\label{sec:rnf}

The main result of this section is Theorem~\ref{thm:rnf}, where we obtain a `$q$–resonant normal form' for the standard map $F$ associated with $(\omega,\phi)$ in the adapted coordinates constructed in Section~\ref{sec:RIC coord}.
The normal form is valid when $\omega$ is sufficiently well approximated by a rational number $p/q$ with large denominator~$q$.

The purpose of this normal form is to capture the dynamics of $F$ in a small neighbourhood of the circle
\begin{equation*}
  \T \times \{p/q - \omega\}
\end{equation*}
and to show that, roughly speaking, the dynamics there depends strongly on the resonant Fourier modes of $\phi$, i.e.\ on the Fourier coefficients $\hat\phi_k$ with $q\mid k$, and is only weakly influenced by the nonresonant modes.
Localization near $\T \times \{p/q - \omega\}$ is important because the dynamics there is close to a rigid rotation in $x$ by the rational angle $p/q$ and, at least when $\phi$ is small, all $(p,q)$–periodic orbits of $F$ are expected to lie nearby.

The following theorem is the technical core of the paper.
For technical reasons, we consider the complexified standard maps in this theorem: we can obviously define an analytic map $\sF$ using the formula \eqref{eq:std map} with a complex analytic potential $V$, which can in turn be given by a real $\omega$ and a complex analytic function $\phi:\T_1 \to \C$ using the construction of Section~\ref{sec:RIC coord}.
By abuse of notation we will still call the resulting map $F$ the standard map associated with $(\omega, \phi)$ in adapted coordinates, even though one may not be able to iterate $F$ on the real cylinder $\T \times \R$ indefinitely unless $\phi$ is real-valued.
The formulae \eqref{eq:fg} describing the map $F$ still hold, of course.
The corresponding normalized generating function $S$ will be in general complex valued and the `symplectic actions' obtained by summing over periodic orbits are complex quantities.
We refer to Section~\ref{sec:notation} for the notation.

\begin{theorem}\label{thm:rnf}
    Fix $\nu_0 := 64$, $\nu \geq \nu_0 + 6$ and $\kappa \in \Z^+$.
    Let $\omega \in \R \setminus \Q$ and let $\phi: \T_1 \to \C$ be analytic, of zero average, with $\|\dddot\phi\| \le 1$.
    Let $F$ be the standard map associated with $(\omega,\phi)$ in adapted coordinates.
    Suppose $q$ is sufficiently large and
    \begin{equation}\label{eq:om arith cond}
        \Big|\omega - \frac{p}{q}\Big| \leq q^{-\nu} \qquad \text{for some } p \in \Z.
    \end{equation}
    Then there exists an analytic change of coordinates
    \begin{equation*}
        H_q : \T_a \times \D_b \longrightarrow \T_1 \times \C,
        \qquad a = 1 - \frac{3}{q}, \quad b = \frac{1}{5} \Big|\omega - \frac{p}{q}\Big|,
    \end{equation*}
    with the following properties.
    \begin{enumerate}
        \item The map $H_q$ is close to the translation by $p/q - \omega$ in $y$ in the sense that 
        \begin{equation}\label{eq:h12}
            H_q: (x,y) \longmapsto \bigl(x + h_{q,1}(x,y),\, y + p/q - \omega + h_{q,2}(x,y)\bigr)
        \end{equation}
        where 
        \begin{equation}\label{eq:h12 est}
            |h_{q,1}|_{a, b} \lesssim q \verts{\omega - \frac{p}{q}}; \quad |h_{q,2}|_{a,b} < \frac{1}{10} \verts{\omega - \frac{p}{q}}.
        \end{equation}

        \item Defining $F_q = H_q^{-1} \circ F \circ H_q$, we have
        \begin{equation}\label{eq:rnf gq}
            F_q :
            \begin{pmatrix}
                x \\ y 
            \end{pmatrix}
            \longmapsto
            \begin{pmatrix}
                x + p/q + y \\
                y + g_q(x,y)
            \end{pmatrix}
        \end{equation}
        for some analytic $g_q : \T_a \times \D_b \to \C$ such that, with $N := \kappa q$,
        \begin{align}
            &|g_q|_{a,b} \;\lesssim\; \Big|\omega - \frac{p}{q}\Big|^2, \label{eq:rnf concl 1}\\
            &\bigl|g_q - \langle g_q \rangle_q^N\bigr|_{1/q,b}
                \;\lesssim\;
                \eu^{6\pi \kappa}\,\eu^{-2\pi (\kappa+1) q}
                \Big|\omega - \frac{p}{q}\Big|^2. \label{eq:rnf concl 2}
        \end{align}

        \item For $|y| < b$, the averaged part of $g_q(\cdot,\; y)$ satisfies
        \begin{equation}\label{eq:gq star}
            |g_q^*(y)|
            \;\lesssim\;
            q^{\nu_0 - \nu}\;
            \sup_{\Im x = 0}|g_q^\bullet(x, y)|.
        \end{equation}

        \item Writing
        \begin{equation}\label{eq:f1}
            f_1(x)
             := \frac{1}{(1 + \dot\phi(x))(1 + \dot\phi(x + \omega))} - 1,
        \end{equation}
        we have the comparison estimate
        \begin{equation}\label{eq:rnf res compare}
            \Big\|
              \Big\langle
                 g_q(\cdot , 0)
                 + \frac{1}{2} \Big(\frac{p}{q} - \omega \Big)^2
                   \frac{\langle \dot f_1 \rangle_q}{1 + \langle f_1 \rangle_q}
              \Big\rangle_q^{N}
            \Big\|
            \;\lesssim\;
            \kappa\,\eu^{4\pi \kappa}\,q^{\nu_0 - \nu}
            \Big|\omega - \frac{p}{q}\Big|^2.
        \end{equation}

        \item Let $S$ be the normalized generating function associated with $(\omega,\phi)$ in adapted coordinates (see the end of Section~\ref{sec:RIC coord}), and define
        \begin{equation*}
          S_q := S \circ H_q.
        \end{equation*}
        Then
        \begin{equation}\label{eq:rnf S compare}
            \bigg|
               S_q(x,y)
               - \Big(\frac{p}{q} - \omega + y\Big)^2
                 \frac{1+f_1(x)}{2\bigl(1 + \langle f_1 \rangle_q\bigr)^2}
            \bigg|
            \;\lesssim\;
            q^{\nu_0}\,\Big|\omega - \frac{p}{q}\Big|^3,
            \qquad (x,y) \in \T_{a} \times \D_b.
        \end{equation}
    \end{enumerate}
    Moreover, the change of coordinates $H_q$ depends analytically on $\phi$ and is real-analytic when $\phi$ is real-valued.
\end{theorem}

\begin{remark}
    The bound ``$q$ sufficiently large'' depends only on $\nu$ and $\kappa$.
    In particular, it is independent of $(\omega,\phi)$.
\end{remark}

\begin{remark}
    The analytic dependence of $H_q$ on $\phi$ means more precisely as follows. Suppose $\phi^\tau(x)$ is jointly analytic in $x \in \T_1$ and a complex parameter $\tau$ such that $\phi = \phi^\tau$ satisfies the conditions of Theorem~\ref{thm:rnf}.
    Then the corresponding coordinate transform $H_q = H^\tau_q$ satisfies that $H_q^\tau(x,y)$ is jointly analytic in $(\tau, x, y)$.
\end{remark}

In the rest of this paper, we adopt the following terminology.
Given a standard map $F$ associated with $(\omega,\phi)$ in adapted coordinates and a $q$ satisfying the assumptions of Theorem~\ref{thm:rnf}, we call the coordinates defined by $H_q$ the \emph{$q$–resonant coordinates}.
We call the function $g_q$ in \eqref{eq:rnf gq} the \emph{$q$–resonant normal form} of $F$.
More generally, for a function $g$ we call $\langle g \rangle_q^N$ the \emph{$(q,N)$–resonant part} (or simply the \emph{resonant part}) of $g$, and we call $g - \langle g \rangle_q^N$ the \emph{nonresonant part}.
We call the map $F_q = H_q\inv \circ F \circ H_q$ the \emph{standard map associated with $(\omega, \phi)$ in $q$-resonant coordinates}.

The proof of Theorem~\ref{thm:rnf} is rather technical and will be given in Section~\ref{sec:rnf proof}.

\section{Controlling resonant parts by Picard iteration}\label{sec:picard}

We shall apply Theorem \ref{thm:rnf} to a sequence of coprime pairs of integers $(p_j, q_j)$ with $q_j > 0$ such that $p_j/q_j$ approximates $\omega$ sufficiently fast, obtaining accordingly a sequence of resonant normal forms $g_{q_j}$ from Theorem \ref{thm:rnf}.
With $N_j = \kappa q_j$ for a fixed $\kappa$, we shall approximate the  zero-average $q_j$-resonant part $\langle g^\bullet_{q_j}\rangle_{q_j}^{N_j}$ by the corresponding $q_j$-resonant part $\langle \phi \rangle_{q_j}^{N_j}$ of the function $\phi$ which defines the standard map.
Then, using a Picard iteration we find a choice of $\phi$ for which the individual resonant part $\langle g^\bullet_{q_j}\rangle_{q_j}^{N_j}$ of the map takes desired forms.

\begin{remark} 
    To prove the results in this section, the only conclusions needed from Theorem~\ref{thm:rnf} are \eqref{eq:rnf res compare} and the analytic dependence of $g_q$ on $\phi$.
    The other conclusions of Theorem~\ref{thm:rnf} are not used in this Section.
\end{remark}

We begin by defining certain  Banach spaces suitably adapted to our problem.
Let $\CA$ denote the Banach space of analytic functions on $\T_1$ with zero average and finite $\|\cdot\|$-norm.
For an integer $n > 0$, define a Banach subspace $\CA_n \subset \CA$ by 
\begin{equation*}
    \CA_n = \{u \in \CA \mid \hat u_k = 0 \text{ for all } |k| < n\}.
\end{equation*}
Let $\CA'$ be the subspace of $\CA$ consisting of functions $u$ such that $\dddot u \in \CA$ and put $\|u\|' := \|\dddot u\|$.
Since elements of $\CA'$ have zero average, the space $\CA'$ is a Banach space with respect to the $\|\cdot \|'$-norm.
Similarly, we also define the Banach subspaces $\CA'_n = \CA_n \cap \CA'$ with induced  $\|\cdot\|'$-norm.

We now state the main result of this section.
\begin{theorem}\label{thm:fx pt}
    For any $\epsilon > 0$ and $\kappa \in \Z^+$, 
    there exists $n \in \Z^+$ such that the following holds.
    For any $\omega \in \R \setminus \Q$ and any sequence $(p_j, q_j)$ of coprime integer pairs satisfying, for some $\nu \geq 70$, that 
    \begin{equation}\label{eq:CP assumptions}
        \verts{\omega - \frac{p_j}{q_j}} \leq q_j^{-\nu}, \quad q_{j+1} > 2\kappa q_j \quad \text{and} \quad q_j > n,
    \end{equation}
    and for any $\phi_0 \in \CA'_n$ and $h \in \CA$ such that $h = \sum_{j} \crochet{h}_{q_j}^{N_j}$ with $N_j = \kappa q_j$ and
    \begin{equation}\label{eq:h+phi0 0}
        \|h\| + 2\|\phi_0\|' < 1-2\epsilon,
    \end{equation}
    there exists a unique analytic function $\phi_1 \in \CA'_n$ depending analytically on $(h, \phi_0)$ such that $\phi_1 = \sum_j \langle \phi_1 \rangle_{q_j}^{N_j}$, $\|\phi_1\|' \leq 1 - \epsilon - \|\phi_0\|'$ and
    \begin{equation}\label{eq:fx pt 0}
        \bigg\langle\frac{\partial g_{q_j}}{\partial x}(\cdot, 0)\bigg \rangle_{q_j}^{N_j} = \Big(\frac{p_j}{q_j} - \omega \Big)^2 \crochet{h}_{q_j}^{N_j}  \quad \text{for} \quad j \geq 1
    \end{equation}
    where $g_{q_j}$ is the $q_j$-resonant normal form as constructed in Theorem~\ref{thm:rnf} with $\phi := \phi_0 + \phi_1$ and the chosen $\omega$, $\kappa$ and $\nu$.
\end{theorem}

\begin{remark}
    The meaning of analytic dependence of $\phi_1$ on $(h, \phi_0)$ is in the sense analogous to the second remark following Theorem~\ref{thm:rnf}, i.e., if $h^\tau(x)$ and $\phi_0^\tau(x)$ are both analytic jointly in $x \in \T_1$ and a complex parameter $\tau$ such that $h = h^\tau$ and $\phi_0 = \phi_0^\tau$ satisfy the conditions of Theorem~\ref{thm:fx pt}, then the corresponding function $\phi_1 = \phi_1^\tau$ has the property that $\phi_1^\tau(x)$ is jointly analytic in $x$ and $\tau$.
\end{remark}

Theorem~\ref{thm:fx pt} will be proved by combining the usual Picard iteration (Banach fixed point theorem) with the quantitative estimates from the following lemma.

\begin{lemma}\label{lem:CP}
    For any $\epsilon > 0$ and $\kappa \in \Z^+$, 
    there exists $n \in \Z^+$ such that the following holds. Let $\omega$ and $(p_j, q_j)$ be as in Theorem~\ref{thm:fx pt} such that \eqref{eq:CP assumptions} holds for some $\nu \geq 66$.
    For any $\phi \in \CA_n'$ with $\|\phi\|' \leq 1$,
    let
    \begin{equation*}
        F_{q_j}:
        \begin{pmatrix}
            x \\y
        \end{pmatrix} \mapsto 
        \begin{pmatrix}
            x + p_j/q_j + y \\ y + g_{q_j}(x,y)
        \end{pmatrix}, \qquad g_{q_j}: \T_{1 - 3/q_j} \times \D_{|\omega - p_j/q_j|/5} \to \C
    \end{equation*}
    be the standard map associated with $(\omega, \phi)$ in the $q_j$-resonant coordinates as constructed by applying Theorem~\ref{thm:rnf} with $q = q_j$ and define
    \begin{equation}\label{eq:CP def}
        \CP(\phi):= \sum_{j = 1}^\infty \crochet{\Big(\frac{p_j}{q_j} - \omega \Big)^{-2} \frac{\partial g_{q_j}}{\partial x} (\cdot, 0 ) - \dddot\phi}_{q_j}^{N_j} \quad \text{where} \quad N_j = \kappa q_j.
    \end{equation}
    Then it holds that $\CP(\phi): \T_1 \to \C$ is an analytic function with
    \begin{equation}\label{eq:CP eps}
        \|\CP(\phi)\| < \epsilon,   
    \end{equation}
    and the operator $\CP: \CA_n' \to \CA$ is analytic.
    Moreover, for $\|\phi\|' < 1$, the operator norm of the derivative $d_\phi \CP: \CA_n' \to \CA$ satisfies
    \begin{equation}\label{eq:CP der bound}
        \|d_\phi \CP\| < \frac{\epsilon}{1 - \|\phi\|'}.
    \end{equation}
\end{lemma}

\begin{proof}[Proof of Theorem~\ref{thm:fx pt} assuming Lemma~\ref{lem:CP}]
    Let $n$ be given by Lemma~\ref{lem:CP} and take $\omega$ and $(p_j, q_j)$ satisfying the conditions of Theorem~\ref{thm:fx pt}.
    Note that the required conclusion \eqref{eq:fx pt 0} is equivalent to 
    \begin{equation}\label{eq:fx pt eq prelim}
        \dddot \phi_1 + \sum_{j \geq 1} \langle \dddot \phi_0 \rangle_{q_j}^{N_j} = h - \CP(\phi_0 + \phi_1)
    \end{equation}
    where $\CP$ is as defined in \eqref{eq:CP def}.
    In order to cast the problem into a fixed point problem, let $\mathscr{D}: u \mapsto \dot u$ denote the derivative map. 
    Then $\mathscr{D}^3: \CA' \to \CA$ is a Banach space isometry. 
    Take the inverse $\mathscr{D}^{-3}: \CA \to \CA'$ and define
    \begin{equation}\label{eq:fx pt map}
        \begin{aligned}
            \CH_{h, \phi_0}: \CA_n' \to & \CA_n'\\
            \phi_1 \mapsto& \mathscr{D}^{-3} \big[ h - \CP(\phi_0 + \phi_1) - \sum_{j \geq 1} \langle \dddot \phi_0 \rangle_{q_j}^{N_j} \big].
        \end{aligned}
    \end{equation}
    Then finding $\phi_1$ satisfying \eqref{eq:fx pt eq prelim} is equivalent to finding a fixed point of $\CH_{h, \phi_0}$.

    Fix $\phi_0 \in \CA'_n$ and $h \in \CA$ satisfying \eqref{eq:h+phi0 0} and fix $r = 1 - \epsilon - \|\phi_0\|'$.
    
    Then, for $\|\phi_1\|' \leq r$ we have $\|\phi_0 + \phi_1\|' \leq 1 - \epsilon < 1$. 
    By the definition of $\|\cdot\|'$-norm, the estimate \eqref{eq:CP eps} of Lemma~\ref{lem:CP} and the assumption \eqref{eq:h+phi0 0}, 
    \begin{equation*}
        \|\CH_{h, \phi_0}(\phi_1)\|' \leq \| h\| + \| \CP(\phi_0 + \phi_1)\| + \|\phi_0\|' < \|h\| + \epsilon + \|\phi_0\|' < r.
    \end{equation*}
    Taking the derivative of~\eqref{eq:fx pt map} in $\phi_1$ and using the linearity of $\mathscr{D}^{-3}$, we have $d_{\phi_1}\CH_{h, \phi_0} = -\mathscr{D}^{-3} \circ d_{\phi_0 + \phi_1} \CP $ and hence we deduce from \eqref{eq:CP der bound} and $\|\phi_1\|' \leq r$ that the operator norm of the linear map $d_{\phi_1}\CH_{h, \phi_0}: \CA'_n \to \CA'_n$ satisfies 
    \begin{equation*}
        \|d_{\phi_1}\CH_{h, \phi_0}\| = \sup_{\|\delta\|' = 1} \|\mathscr{D}^{-3}[ d_{\phi_0 + \phi_1} \CP(\delta)]\|'  = \| d_{\phi_0 + \phi_1} \CP\| < \frac{\epsilon}{1-\|\phi_0 + \phi_1\|'} \leq 1.
    \end{equation*}
    Note also that by definition of $\CP$ and the condition on $h$, we have $\CH_{h, \phi_0}(\phi_1) = \sum_{j \geq 1} \langle \CH_{h, \phi_0}(\phi_1) \rangle_{q_j}^{N_j}$.
    Hence, the contracting mapping principle applies to $\CH_{h, \phi_0}$ on the closed set 
    \begin{equation*}
        \bigg\{u \in \CA_{n}' \ : \ \|u\|' \leq r, \ u = \sum_{j \geq 1} \langle u \rangle_{q_j}^{N_j}   \bigg\}
    \end{equation*}
    and yields a unique fixed point therein.
    Moreover, by analytic dependence of $\CP$ on $\phi$ as shown in Lemma~\ref{lem:CP}, the fixed point depends analytically on $h$ and $\phi_0$.
    The proof is finished.
\end{proof}

In the rest of this section, we focus on the proof of Lemma~\ref{lem:CP}.
Recall that Theorem~\ref{thm:rnf} gives a comparison estimate \eqref{eq:rnf res compare}.
The main idea for the proof of Lemma~\ref{lem:CP} will be to give a similar comparison estimate between the resonant part $\langle\dddot\phi\rangle_{q_j}^{N_j}$ of $\dddot\phi$ and the function
\begin{equation*}
    -\frac{1}{2}\frac{\partial}{\partial x} \bigg\langle \frac{\langle \dot f_1 \rangle_{q_j}}{1 + \langle f_1 \rangle_{q_j}}\bigg\rangle_{q_j}^{N_j}.
\end{equation*}
Since $f_1$ depends explicitly on $\phi$ via \eqref{eq:f1}, we can compute an estimate directly.

\begin{proof}[Proof of Lemma \ref{lem:CP}]
    Let $\epsilon > 0$ and $\kappa \in \Z_{> 0}$ be given.
    Consider a sequence $(p_j, q_j)$ of coprime integer pairs with $q_j > 0$ satisfying the assumptions \eqref{eq:CP assumptions} with some $n > 0$ to be fixed later and $\nu \geq 70$.
    Taking $n$ sufficiently large, we can apply Theorem~\ref{thm:rnf} to any $\phi \in \CA'$ with $\|\phi\|' \leq 1$ and each $q_j$ and accordingly we can define $\CP(\phi)$ as in \eqref{eq:CP def} using the sequence $(g_{q_j})_{j \geq 1}$ of $q_j$-resonant normal forms furnished by Theorem~\ref{thm:rnf}.
    We will take $\phi \in \CA_n'$.
    To study the series on the right hand side of the definition \eqref{eq:CP def}, we shall take $M_j = 2N_j = 2\kappa q_j$ and decompose $\CP(\phi) = \wt \CP(\phi) + \sum_{j \geq 1} \big(\CP_j(\phi) + \CP'_j(\phi)\big)$
    where
    \begin{equation*}
        \begin{aligned}
            \wt \CP(\phi) =& - \sum_{j \geq 1} \crochet{ \frac{1}{2}\ddot f_1  + \dddot\phi}_{q_j}^{N_j}, \\
            \CP_j(\phi) =& - \frac{1}{2}\crochet{ \frac{\du}{\du x} \frac{\langle \dot f_1 \rangle_{q_j}^{M_j}}{1 + \langle f_1 \rangle_{q_j}^{M_j}} -  \ddot f_1 }_{q_j}^{N_j}, \\
        \CP_j'(\phi) =& \frac{\du}{\du x} \crochet{ \Big(\frac{p_j}{q_j} - \omega\Big)^{-2} g_{q_j}(\cdot, 0) + \frac{1}{2} \frac{\langle \dot f_1 \rangle_{q_j}^{M_j}}{1 + \langle f_1 \rangle_{q_j}^{M_j}}}_{q_j}^{N_j}
        \end{aligned}
    \end{equation*}
    and $f_1: \T_1 \to \C$ is as defined in \eqref{eq:f1}.
    We shall analyse $\wt \CP$, $\CP_j$ and $\CP'_j$ separately.

    We start with $\wt \CP$.
    By definition of $\|\cdot\|'$-norm, the assumptions $\|\phi\|' \leq 1$ and $\phi \in \CA_n'$ imply
    \begin{equation*}
        \|\ddot \phi\| = \sum_{|k| \geq  n} (2\pi |k|)^2 |\hat \phi_k| e^{2\pi |k|} 
        \leq \frac{1}{2\pi n} \sum_{|k| \geq  n} (2\pi |k|)^3 |\hat \phi_k| e^{2\pi |k|} 
        = \frac{\|\phi\|'}{2\pi n}.
    \end{equation*}
    Similarly, we have $\|\dot\phi\| \leq (2\pi n)^{-2}$.
    From the explicit expression $(1 + f_1) = ((1+\dot\phi)(1 + \dot\phi^+))\inv$, a direct calculation shows that
    \begin{equation*}
        \frac{1}{2}f_1 + \dot\phi + \frac{1}{2}\CD_\omega \dot\phi = \frac{1}{2}(f_1  + \dot\phi + \dot\phi^+)
        = \frac{\dot\phi^2 + (\dot\phi^+)^2 + \dot\phi \dot\phi^+ + \dot\phi^2 \dot\phi^+  + \dot\phi ( \dot\phi^+)^2}{2(1+\dot\phi) (1+\dot\phi^+)}.
    \end{equation*}
    By differentiating the above expression twice and using $|\dot\phi| \leq (2\pi n)^{-2}$, $|\ddot\phi| \leq (2\pi n)^{-1}$ and $\|\dddot\phi\| \leq 1$, it is straightforward to verify that the right hand side of the above equation has $\|\cdot\|$-norm bounded by $\lesssim 1/n$.
    On the other hand, by Lemma~\ref{lem:fourier MVT},
    \begin{equation*}
        \|\langle \CD_\omega \dddot\phi \rangle_{q_j}^{N_j}\| = \|\langle \CD_{\omega - p_j/q_j} \dddot\phi\rangle_{q_j}^{N_j}\| \leq \verts{\omega - \frac{p_j}{q_j}} \norm{\Big \langle \frac{\du}{\du x} \dddot\phi \Big\rangle_{q_j}^{N_j}} \lesssim N_j \verts{\omega - \frac{p_j}{q_j}} \|\dddot\phi\|.
    \end{equation*}
    Thus, using $N_j = \kappa q_j$ we obtain
    \begin{equation*}
        \|\tilde \CP(\phi) \| \lesssim \frac{1}{n} + \sum_j N_j \verts{\omega -\frac{p_j}{q_j}} \|\dddot\phi\| \lesssim \frac{1}{n} + \kappa \sum_j q_j \verts{\omega -\frac{p_j}{q_j}}.
    \end{equation*}
    Since $|\omega - p_j/q_j| \leq q_j^{-\nu}$ and $q_{j+1} > 2\kappa q_j > n$, we can bound $q_j |\omega - p_j/q_j| < [(2\kappa)^{j-1} n]^{-\nu + 1}$ and then bound the series above by a converging geometric series to deduce 
    \begin{equation}\label{eq:barCP ineq}
        \|\tilde \CP(\phi) \| \lesssim \frac{1}{n} + \frac{\kappa n^{-\nu + 1}}{1 - (2\kappa)^{-\nu + 1}} .
    \end{equation}

    Next, we consider $\CP_j$.
    By a straightforward computation, we have 
    \begin{equation*}
        \CP_j(\phi) = \frac{1}{2}\crochet{ \frac{(\langle \dot f_1\rangle_{q_j}^{M_j})^2}{(1+\langle f_1 \rangle_{q_j}^{M_j})^2}  + \frac{\langle \ddot f_1\rangle_{q_j}^{M_j}\langle  f_1\rangle_{q_j}^{M_j}}{1 + \langle f_1\rangle_{q_j}^{M_j}}}_{q_j}^{N_j}.
    \end{equation*}
    Using $\|\dot\phi\| \leq (2\pi n)^{-2}$ and the formula \eqref{eq:f1} for $f_1$, the denominators above are uniformly bounded away from $0$.
    Since $\|\dot f_1\| \lesssim \|\ddot\phi\|$ and  $\|\ddot f_1\| \lesssim \|\dddot\phi\| \leq 1$, 
    $$\|\CP_j(\phi)\| \lesssim \|\langle \dot f_1\rangle_{q_j}^{M_j}\|.$$
    By assumption $q_{j+1} > M_j$, the orders of nonzero Fourier coefficients of $\langle \dot f_1\rangle_{q_j}^{M_j}$ are \textit{disjoint} subsets of $\Z$. 
    Hence, it follows from the definition of $\|\cdot\|$-norm and the definition of $\CP_j$ that 
    \begin{equation}\label{eq:CPj ineq}
        \Big\|\sum_{j \geq 1} \CP_j(\phi)\Big\| = \sum_{j \geq 1} \|\CP_j(\phi)\|  \lesssim \sum_{j \geq 1} \|\langle \dot f_1\rangle_{q_j}^{M_j}\| \lesssim \| \dot f_1\| \lesssim \|\ddot\phi\| \lesssim \frac{1}{n}.
    \end{equation}
    
    To analyse $\CP'_j(\phi)$, we start with the following decomposition 
    \begin{equation*}
        \CP'_j(\phi) = \frac{\du}{\du x} \crochet{\Big(\omega - \frac{p_j}{q_j}\Big)^{-2} g_{q_j}(\cdot, 0) + \frac{1}{2} \frac{\langle \dot f_1 \rangle_{q_j}}{1 + \langle f_1 \rangle_{q_j}}}_{q_j}^{N_j} + \frac{1}{2}\frac{\du}{\du x} \crochet{\frac{\langle \dot f_1 \rangle_{q_j}^{M_j}}{1 + \langle f_1 \rangle_{q_j}^{M_j}} - \frac{\langle \dot f_1 \rangle_{q_j}}{1 + \langle f_1 \rangle_{q_j}}}_{q_j}^{N_j}.
    \end{equation*}
    By the estimate \eqref{eq:rnf res compare} of Theorem \ref{thm:rnf} and $N_j = \kappa q_j$, the first term is bounded by 
    \begin{equation}\label{eq:CP' ineq 1}
    \begin{aligned}
        &\norm{\frac{\du}{\du x} \crochet{\Big(\omega - \frac{p_j}{q_j}\Big)^{-2} g_{q_j}(\cdot, 0) + \frac{1}{2} \frac{\langle \dot f_1 \rangle_{q_j}}{1 + \langle f_1 \rangle_{q_j}}}_{q_j}^{N_j}} \\
        \leq& 2\pi N_j \norm{\crochet{\Big(\omega - \frac{p_j}{q_j}\Big)^{-2} g_{q_j}(\cdot, 0) + \frac{1}{2} \frac{\langle \dot f_1 \rangle_{q_j}}{1 + \langle f_1 \rangle_{q_j}}}_{q_j}^{N_j}}\\
        \lesssim& \kappa^2 \eu^{4\pi \kappa} q_j^{\nu_0 - \nu + 1}
    \end{aligned}
    \end{equation}
    with $\nu_0 = 64$ as in Theorem \ref{thm:rnf},
    whereas the second term writes 
    \begin{equation*}
    \begin{aligned}
        \frac{1}{2}\frac{\du}{\du x} \crochet{\frac{\langle \dot f_1 \rangle_{q_j}^{M_j}}{1 + \langle f_1 \rangle_{q_j}^{M_j}} - \frac{\langle \dot f_1 \rangle_{q_j}}{1 + \langle f_1 \rangle_{q_j}}}_{q_j}^{N_j}
        = &\frac{1}{2} \frac{d^2}{dx^2} \crochet{\log(1 + \langle f_1 \rangle_{q_j}^{M_j}) -  \log(1 + \langle f_1 \rangle_{q_j})}_{q_j}^{N_j}\\
    =& \frac{1}{2} \frac{d^2}{dx^2} \int_0^1 \crochet{ \frac{\langle f_1 \rangle_{q_j}^{>M_j}}{1 + \langle f_1 \rangle_{q_j}^{M_j} + t \langle f_1 \rangle_{q_j}^{>M_j}} }_{q_j}^{N_j}dt.
    \end{aligned}
    \end{equation*}
    By the first statement of Lemma~\ref{lem:Fourier sep} and the definition  $M_j = 2N_j = 2\kappa q_j$,
    \begin{equation*}
        \begin{aligned}
            \Bigg\|\crochet{ \frac{\langle f_1 \rangle_{q_j}^{>M_j}}{1 + \langle f_1 \rangle_{q_j}^{M_j} + t \langle f_1 \rangle_{q_j}^{>M_j}} }_{q_j}^{N_j}\Bigg\| 
            \leq& (2\kappa + 1) \eu^{-4\pi (M_j - N_j)} \Bigg\|\frac{1}{1 + \langle f_1 \rangle_{q_j}^{M_j} + t \langle f_1 \rangle_{q_j}^{>M_j}}\Bigg\| \|\langle f_1 \rangle_{q_j}^{>M_j}\| \\
            \lesssim& \kappa \eu^{-4\pi \kappa q_j}.
        \end{aligned}
    \end{equation*}
    Substituting into the integral above, we obtain
    \begin{equation*}
        \norm{\frac{1}{2}\frac{\du}{\du x} \crochet{\frac{\langle \dot f_1 \rangle_{q_j}^{M_j}}{1 + \langle f_1 \rangle_{q_j}^{M_j}} - \frac{\langle \dot f_1 \rangle_{q_j}}{1 + \langle f_1 \rangle_{q_j}}}_{q_j}^{N_j}} \lesssim N_j^2 \kappa \eu^{-4\pi \kappa q_j} \lesssim \kappa^3 q_j^2 \eu^{-4\pi \kappa q_j}.
    \end{equation*}
    Combining with \eqref{eq:CP' ineq 1} and summing over $j \geq 1$, we obtain
    \begin{equation*}
    \Big\|\sum_{j \geq 1} \CP'_j(\phi)\Big\| \lesssim \sum_{j \geq 1} \kappa^2 \eu^{4\pi \kappa} q_j^{\nu_0 - \nu + 1} + \kappa^3 q_j^2 \eu^{-4\pi \kappa q_j} .
    \end{equation*}
    Using the assumption $\nu \geq \nu_0 + 6$ and $q_j > (2\kappa)^{j-1} n$ and bounding the series in the right hand side above by convergent geometric series,
    \begin{equation}\label{eq:CP' ineq}
        \Big\|\sum_{j \geq 1} \CP'_j(\phi)\Big\| \lesssim \frac{\kappa^2 \eu^{4\pi \kappa} }{1 - (2\kappa)^{-\nu + \nu_0 + 1}} n^{-\nu + \nu_0 + 1} + \kappa^3 n^2 \eu^{-4\pi \kappa n}.
    \end{equation}
    Consolidating \eqref{eq:barCP ineq}, \eqref{eq:CPj ineq} and \eqref{eq:CP' ineq}, we arrive at 
    \begin{equation}\label{eq:CP ineq}
        \|\CP(\phi)\| \lesssim \frac{1}{n} + \frac{\kappa^2 \eu^{4\pi \kappa} }{1 - (2\kappa)^{-\nu + \nu_0 + 1}} n^{-\nu + \nu_0 + 1} + \kappa^3 n^2 \eu^{-4\pi \kappa n}.
    \end{equation}
    In particular, the right hand side of \eqref{eq:CP def} is an absolutely convergent series in $\CA$. Thus $\CP(\phi)$ is well-defined as an element in $\CA$.
    Taking $n$ sufficiently big, we can make the right hand side of \eqref{eq:CP ineq} smaller than any given positive number and thus obtain \eqref{eq:CP eps}.
    Since $\nu_0 = 64$ and we assume $\nu \geq 70$, we have $n^{-\nu + \nu_0 + 1} \leq n^{-5}$ and thus the `sufficiently big' condition on $n$ depends only on $\epsilon$ and $\kappa$, not on $\nu$.

    Next, we consider the analyticity $\phi \mapsto \CP(\phi)$.
    Let $\phi^\tau \in \CA'$ depend analytically on some complex parameter $\tau \in \C$.
    Suppose $\|\phi^\tau\| '\leq 1$ for all $|\tau| \leq R$ so that $\CP(\phi^\tau)$ is defined.
    We consider $\CP(\phi^\tau)$ as the point-wise limit of the partial sum of its Fourier series
    \begin{equation*}
        \CP^K(\phi^\tau)(x) := \sum_{|k| \leq K} \widehat{\CP(\phi^\tau)}_k \eu^{2\pi \iu k x}.
    \end{equation*}
    Indeed, by the absolute convergence of \eqref{eq:CP def} established above, if $|\tau| \leq R$ then the sequence $\CP^K(\phi^\tau)$ converges in $\CA$ as $K \to \infty$.
    In particular, for each fixed $\tau$, the set $\{\CP^K(\phi^\tau) \mid K \geq 1\}$ is a  relatively compact subset in $\CA$ .
    On the other hand,  the function $\CP^K(\phi^\tau): \T_1 \to \C$ has finitely many nonzero Fourier coefficients and, by analytic dependence of $\langle g_{q_j}(\cdot, 0)\rangle_{q_j}^{N_j}$ on $\phi^\tau$, each of these coefficients depends analytically on $\tau$.
    It follows that $\CP^K(\phi^\tau)$ is an analytic function in $\tau$ taking values in $\CA$.
    Hence, by Montel's theorem, the family  of analytic functions $\{\tau \mapsto \CP^K(\phi^\tau) \}_{K \geq 1}$ is equicontinuous, and, by Ascoli's theorem, the convergence $\CP^K(\phi^\tau) \to \CP(\phi^\tau)$ as $K \to \infty$ is uniform over $\tau$.
    It follows that $\CP(\phi^\tau)$  depends analytically in $\tau$ since it is the uniform limit of a sequence of analytic functions.
    We have thus shown that $\CP: \CA_n' \to \CA$ is analytic.

    Finally, we show \eqref{eq:CP der bound}. 
    Put $\phi^\tau = \phi + \tau \delta$ for some $\phi, \delta \in \CA'_n$ with $\|\phi\|' < 1$.
    Then we have that $\CP(\phi^\tau)$ is analytic for $|\tau| < \frac{1 - \|\phi\|'}{\|\delta\|'}$.
    Applying Cauchy's estimate to the derivative $\frac{d}{d\tau} \CP(\phi^\tau)|_{\tau = 0}$ and using \eqref{eq:CP eps}, we obtain 
    \begin{equation*}
        \|d_{\phi} \CP(\delta)\| 
        < \Big(\frac{1 - \|\phi\|'}{\|\delta\|'}\Big)\inv \epsilon = \frac{\epsilon}{1 - \|\phi\|'} \|\delta\|'.
    \end{equation*}
    Since $\delta$ can be  chosen arbitrarily from an open ball, it follows that $\|d_\phi \CP\| \leq \epsilon (1 - \|\phi\|')\inv$ in operator norm.
    The proof of the Lemma is finished.
\end{proof}

\begin{remark}
        In the proof of Theorem \ref{thm:fx pt}, one can avoid dealing with analytic functions taking values in an infinite dimensional Banach space  by employing a slightly different argument.
        Let $\hat u_k(z)$ be the $k$-th Fourier coefficient of $\CP(\phi + z \delta)$ where $\phi, \delta \in \CA'$ such that $\|\phi\|' + \|\delta\|' < 1$.
        By analytic dependence of $g_{q_j}$'s, each $\hat u_k(z)$ is an analytic function in $z$ for $\|\phi + z \delta\|' < 1$.
        Let us denote
        \begin{equation*}
            G(\phi, \delta): x \mapsto \sum_k \hat u_k'(0) e^{2\pi \iu k x}.
        \end{equation*}
        We shall prove that $G(\phi, \cdot):\CA' \to \CA$ is a bounded linear map and it is in fact equal to $\du_{\phi} \CP$.
        
        Fix $\theta_k\in \R$ such that $|\hat u_k(1) - \hat u_k(0) -  \hat u_k'(0)| = (\hat u_k(1) - \hat u_k(0) -  \hat u_k'(0)) e^{\iu \theta_k}$.
        Consider 
        \begin{equation}\label{eq:H series}
            H(z) = \sum_{k \in \Z} \hat u_k(z) e^{\iu \theta_k} e^{2\pi |k|}.
        \end{equation}
        By the same argument using Montel's theorem and Ascoli's theorem as in the proof of Lemma~\ref{lem:CP}, one can show that the series \eqref{eq:H series}  converges uniformly on compact subsets of the open disc of radius $(1 - \|\phi\|')/\|\delta\|'$ and therefore defines a complex analytic function in $z$, which is uniformly bounded by $\epsilon$.
        Similarly, the function $z \mapsto \sum_k \hat u_k(z) e^{2\pi |k|}$ is analytic and uniformly bounded by $\epsilon$ over the same domain.
        Then, by Cauchy's estimate  we have $\|G(\phi, \delta)\| \leq \epsilon (1 - \|\phi\|')\inv \|\delta\|'$.
        Hence, $G(\phi, \cdot)$ is bounded in operator norm by $\epsilon (1 - \|\phi\|')\inv$.
        Next, by the choice of $\theta_k$, the mean value theorem and Cauchy's estimate, 
        \begin{equation*}
            \begin{aligned}
                \|\CP(\phi + \delta) - \CP(\phi) - G(\phi, \delta)\| 
                =& \sum_k |\hat u_k(1) - \hat u_k(0) - \hat u_k'(0)| e^{2\pi |k|} \\
                =& \sum_k (\hat u_k(1) - \hat u_k(0) - \hat u_k'(0)) \eu^{\iu \theta_k} e^{2\pi |k|} \\
                =& H(1) - H(0) - H'(0) \\
                =& \int_0^1 H''(t) (1 - t) dt\\
                \leq& \frac{\epsilon}{2(\frac{1 - \|\phi\|'}{\|\delta\|'} - 1)^2}\\
                \lesssim& \frac{\epsilon \|\delta\|'^2}{(1 - \|\phi\|' - \|\delta\|')^2}.
            \end{aligned}
        \end{equation*}
        Assuming $\|\delta\|'$ is sufficiently small ($\|\delta\|' < (1 - \|\phi\|')/2$ suffices), we obtain $\|\CP(\phi + \delta) - \CP(\phi) - G(\phi, \delta)\| \lesssim \epsilon (1 - \|\phi\|')^{-2} \|\delta\|'^2$.
        This shows that $G(\phi, \cdot)$ is indeed the derivative $d_\phi \CP$ and thus we have $\|d_\phi \CP\| \leq \epsilon (1 - \|\phi\|')\inv$ as required.
\end{remark}

\section{Analysis of almost equidistributed periodic orbits: the end of the proof of Theorem~\ref{thm:main}}\label{sec:equidistr}

We will apply Theorem~\ref{thm:fx pt} to construct a sequence of $(p_j, q_j)$-periodic orbits of a standard map with well-controlled symplectic actions and eigendata, and then complete the proof of Theorem~\ref{thm:main}.

More precisely, we start with the following lemma. 
We recall that the functional space $\CA_n'$ and the norm $\|\cdot\|'$ have been defined at the beginning of Section~\ref{sec:picard}.

\begin{lemma}\label{lem:2 orb}
    The following statements hold for all sufficiently large $n > 0$.

    Let $\omega \in \R \setminus \Q$ and assume there exists a sequence
    $(p_j, q_j)_{j \geq 1}$ of coprime integer pairs such that, for some
    $\nu \geq 70$,
    \begin{equation}\label{eq:equidistr arithm}
        \Bigl|\omega - \frac{p_j}{q_j}\Bigr|
        \leq q_j^{-\nu}, 
        \qquad  
        q_{j+1} > 4q_j,
        \qquad 
        q_1 > n.
    \end{equation}
    Then, for every $\phi_0 \in \CA_n'$ and every sequence
    $(\sigma_j)_{j \geq 1}$ of complex numbers satisfying
    \begin{equation*}
      |\sigma_j| q_j \geq 1
      \quad\text{for all } j
      \qquad\text{and}\qquad
      \sum_{j \geq 1} |\sigma_j|
      + 2\|\phi_0\|' \leq \frac12,
    \end{equation*}
    there exists $\phi_1 \in \CA_n'$, depending analytically on $\phi_0$ and on
    $(\sigma_j)_{j \geq 1}$, with the following properties.

    Let $\phi := \phi_0 + \phi_1$ and let $F$ be the standard map associated
    with $(\omega,\phi)$ in the adapted coordinates.

    For each $j \geq 1$ there exist two distinguished $(p_j, q_j)$-periodic orbits of $F$ in a complex neighborhood of $\T \times \R$. These orbits depend analytically on $\phi_0$ and $(\sigma_j)_{j \geq 1}$
    and satisfy:

    \begin{enumerate}
        \item 
        The two orbits are geometrically distinct.
        If $(\sigma_j)_j$ and $\phi_0$ are real valued, then the two orbits lie in $\T\times \R$ and they are the \emph{only} $(p_j,q_j)$-periodic orbits of $F$ in $\T \times \R$.
        In particular, one of them is a (globally) minimal orbit and the other
        is a minimax orbit.

        \item 
        Let $S$ be the normalized generating function associated with $F$ in adapted coordinates, and
        define the rescaled symplectic actions of the two $(p_j,q_j)$-orbits by
        \begin{equation}\label{eq:CAj def}
            \CA_{j,\pm}
            := \frac{1}{q_j (\omega - p_j/q_j)^2}
               \sum_{k = 0}^{q_j - 1}
               S \circ F^k(\xi_{j,\pm}),
        \end{equation}
        where $\xi_{j,+}$ and $\xi_{j,-}$ are any points on the two orbits,
        respectively.
        Then
        \begin{equation}\label{eq:CAj est}
            \Biggl|
              \CA_{j,\pm}
              - \frac{1}{2(1 + f_1^*)}
                \Bigl(
                  1 \mp \frac{2\sigma_j \eu^{-2\pi q_j}}{(2\pi q_j)^2}
                \Bigr)
            \Biggr|
            \;\lesssim\;
            q_j^{\nu_0} \Bigl|\omega - \frac{p_j}{q_j}\Bigr|
            + q_j^{\nu_0 - \nu} \eu^{-2\pi q_j},
        \end{equation}
        where $f_1 : \T_1 \to \C$ is defined in terms of $\phi = \phi_0 + \phi_1$
        as in~\eqref{eq:f1}, $f_1^*$ denotes its average over $\T$, and $\nu_0 = 64$.

        \item 
        Define the rescaled Lyapunov exponents of these orbits by
        \begin{equation}\label{eq:Lamj def}
            \Lambda_{j,\pm}
            := -\frac{1}{q_j^2 (\omega - p_j/q_j)^2}
               \det\bigl(d_{\xi_{j,\pm}} F^{q_j} - \id\bigr),
        \end{equation}
        where $\xi_{j,\pm}$ are as in \textup{(2)}.
        Then
        \begin{equation}\label{eq:Lamj est}
            \bigl|
              \Lambda_{j,\pm} \mp \sigma_j \eu^{-2\pi q_j}
            \bigr|
            < \frac{1}{50} |\sigma_j| \eu^{-2\pi q_j}.
        \end{equation}
    \end{enumerate}
\end{lemma}

The main theorem will be deduced from Lemma~\ref{lem:2 orb} together with the following elementary lemma.

To state Lemma~\ref{lem:linfty picard}, we denote by $\ell^\infty$ the Banach space of bounded sequences $\beta = (\beta_j)_{j \geq 1}$ of complex numbers with the supremum norm 
\begin{equation*}
    |\beta|_{\ell^\infty} := \sup_{j \geq 1} |\beta_j|.
\end{equation*}

\begin{lemma}\label{lem:linfty picard}
    Let $B$ be the unit open ball of $\ell^\infty$ centered at $0$. Let $T: B \to \ell^\infty$ be an analytic map. Suppose
    \begin{equation}\label{eq:linfty picard unif bnd}
        |T(\beta) - \beta|_{\ell^\infty} < \epsilon \qquad \text{for all} \quad \beta \in B.
    \end{equation}
    Then, for any $\alpha \in \ell^\infty$ such that $|\alpha|_{\ell^\infty} < (1 - 4\epsilon)/3$, there exists a unique $\beta_* \in \ell^\infty$ depending analytically on $\alpha$ such that $|\beta_*|_{\ell^\infty} < (1 - \epsilon)/3$ and $T(\beta_*) = \alpha$.
    Furthermore, if $\CU$ is an open subset of a complex Banach space and the map $T = T_u$ is analytically parametrized by $u \in \CU$, i.e., the value $T_u(\beta)_j$ is analytic in $u$ and $\beta \in B$ for each $j$, then the corresponding $\beta_*$ such that $T_u(\beta_*) = \alpha$ also depends analytically on $u$.
\end{lemma}

\begin{remark}
    We say that a map $T: U \to \ell^\infty$ defined on an open subset $U \subset \ell^\infty$ is \textit{analytic} if the following property holds: for any family $(b^\tau)$ of elements in $U$ parametrized by a complex parameter $\tau$ such that, for each index $j \geq 1$, the function $\tau \mapsto b_j^\tau$ is analytic, the function $\tau \mapsto T(b^\tau)_j$ is analytic for every index $j$.
\end{remark}

\begin{proof}[Proof of Lemma \ref{lem:linfty picard}]
    We prove this lemma by reducing it to a fixed-point problem and solving it by the usual Picard iteration.
    Let $\alpha \in \ell^\infty$ be given. Define an auxiliary map 
    \begin{equation*}
        \wt T_\alpha: \beta \mapsto \alpha - T(\beta) + \beta.
    \end{equation*}
    Then finding $\beta$ solving the equation $T(\beta) = \alpha$ is equivalent to finding a fixed point of $\wt T_\alpha$.
    
    By the assumption \eqref{eq:linfty picard unif bnd}, we have $\linf{\wt T_\alpha(\beta)} < \linf{\alpha} + \epsilon$.
    Hence, for $\linf{\alpha} + \epsilon \leq r < 1$, the map $\wt T_\alpha$ preserves the closed ball of radius $r$ centered at $0$.
    Pick any $\beta, \beta' \in \ell^\infty$ such that $\linf{\beta} + \linf{\beta'} \leq 1$.
    Then by the mean value theorem applied to the function $\tau \mapsto \wt T_\alpha((1 - \tau) \beta + \tau \beta')_j$ for an arbitrary $j \geq 1$, 
    \begin{equation*}
        \wt T_\alpha(\beta')_j - \wt T_\alpha(\beta)_j = \int_0^1 \frac{\du}{\du \tau} \wt T_\alpha((1 - \tau) \beta + \tau \beta')_j d\tau.
    \end{equation*}
    By \eqref{eq:linfty picard unif bnd} and Cauchy's estimate over the domain $|\tau| \leq (1 - \linf{\beta})/\linf{\beta - \beta'}$,
    \begin{equation*}
        \verts{\frac{\du}{\du \tau} \wt T_\alpha((1 - \tau) \beta + \tau \beta')_j} \leq \frac{\epsilon \linf{\beta - \beta'}}{1 - \linf{\beta} - \linf{\beta' - \beta}} \qquad \text{for} \quad |\tau| \leq 1.
    \end{equation*}
    Since the above bound is independent of $j$, we deduce 
    \begin{equation*}
        \linf{\wt T_\alpha(\beta') - \wt T_\alpha(\beta)} \leq \frac{\epsilon \linf{\beta - \beta'}}{1 - \linf{\beta} - \linf{\beta' - \beta}}.
    \end{equation*}
    In particular, for $r < 1/3$, the Lipschitz constant of $\wt T_\alpha$ is at most $\epsilon / (1 - 3r)$.
    By the condition on $\alpha$, there exists $r$ such that $|\alpha| + \epsilon \leq r < (1 - \epsilon)/3$.
    It follows that the map $\wt T_\alpha$ preserves the closed ball of radius $r$ and is a contracting map there.
    By Banach fixed point theorem, we obtain a unique fixed point for $\wt T_\alpha$ therein.

    It remains to show the analytic dependence of $\beta$ on $\alpha$ and on $u \in \CU$, if $T = T_u$ is analytically parametrized by $u$ belonging to an open subset of a complex Banach space.
    Recall that, by the usual proof of Banach fixed point theorem, the fixed point $\beta_*$ can be given by the limit $\beta_* = \lim_{n \to \infty} \wt T_\alpha^n(0)$, where the contracting property of $\wt T_\alpha$ implies that $(\wt T_\alpha^n(0))_n$ is a Cauchy sequence in $\ell^\infty$.
    On the other hand, if  $\alpha^\tau$ depends analytically on a complex parameter $\tau$ in a domain of $\C$ and has small $\linf{\ \cdot \ }$-norm uniformly in $\tau$, and $T$ is analytically parametrized by $u$,
    then, by the assumed analyticity of $T$, for fixed $j$ the sequence $\{\wt T_{\alpha^\tau}^n(0)_j \ : \ n \geq 1\}$ is a Cauchy sequence (with respect to uniform norm in $\tau$) of analytic functions in $\tau$ and $u$.
    By standard results in complex analysis, the limit $\lim_{n \to \infty} \wt T_{\alpha^\tau}^n(0)_j$ is an analytic function in $\tau$ and $u$.
    It follows that $\beta_*$ is analytic in $\alpha$ and $u$.
    The proof of the Lemma is finished.
\end{proof}

\subsection{Proof of the main theorem.}
We now complete the proof of the main theorem using Lemma~\ref{lem:2 orb} and Lemma~\ref{lem:linfty picard}.
Let $\omega$ and $(p_j, q_j)$ be given as in the statement of Theorem~\ref{thm:main}.
In particular, we may assume $q_1 > n$ so that Lemma~\ref{lem:2 orb} applies to the sequence $(p_j, q_j)$.
We consider 
\begin{equation*}
    \sigma_j = \frac{2 + \wt \sigma_j}{q_j} \quad \text{with} \quad |\wt \sigma_j| \leq 1.
\end{equation*}
This choice of $\sigma_j$ automatically ensures that $q_j|\sigma_j| \geq 1$.
Moreover, for $q_1$ sufficiently large we have $\sum_{j} |\sigma_j| < 1/10$.
It follows that the condition $\sum_j|\sigma_j| + 2\|\phi_0\|' \leq 1/2$ is satisfied whenever $\|\phi_0\|' \leq 1/5$.
In this case, by Lemma~\ref{lem:2 orb} there exists $\phi_1 \in \CA_n'$ depending analytically on $(\wt \sigma_j)$ and $\phi_0$ such that conclusions of Lemma~\ref{lem:2 orb} hold.
In particular, for each $j \geq 1$ there is a pair of $(p_j, q_j)$-orbits depending analytically on $(\wt \sigma_j)$ and $\phi_0$.
Let $\xi_{j, \pm}$ be points picked from the two distinguished $(p_j,q_j)$-orbits furnished by Lemma~\ref{lem:2 orb}, respectively.
We shall verify that the conclusions of Theorem~\ref{thm:main} hold with respect to the sequence of orbits $\xi_j = \xi_{j,+}$ provided by Lemma~\ref{lem:2 orb}.
We first prove the simpler part~(2) before proceeding to part~(1).

\paragraph{Proof of part (2).} 
Let $\Lambda_j = \Lambda_{j,+}$ with $\Lambda_{j, +}$ defined as in \eqref{eq:Lamj def}.
We provisionally define 
\begin{equation*}
    \wt \Lambda_j = q_j\Lambda_{j} \eu^{2\pi q_j} - 2.
\end{equation*}
By analytic dependence of $\phi_1$ and $\xi_j$ on $(\sigma_j)$ and $\phi_0$, the map $T: (\wt \sigma_j)_j \mapsto (\wt \Lambda_j)_j$ is analytic and is analytically parametrized by $\phi_0$.
On the other hand, the estimate \eqref{eq:Lamj est} and the choice of $\sigma_j$ implies
\begin{equation*}
    \verts{\wt \Lambda_j -  \wt \sigma_j} < \frac{1}{16}.
\end{equation*}
We may in turn apply Lemma \ref{lem:linfty picard} with $\epsilon = 1/16$ and
$T: (\wt \sigma_j)_j \mapsto (\wt \Lambda_j)_j$ to conclude that, for all sequence $(\alpha_j)_j$ with $|\alpha_j| < 1/4$ and all $\phi_0$ such that $\|\phi_0\|' \leq 1/5$, there exists a unique sequence $(\wt \sigma_j)_j$ with $|\wt \sigma_j| < 5/16$ and depending analytically on $(\alpha_j)$ and $\phi_0$  such that $\wt \Lambda_j = \alpha_j$.
By Lemma~\ref{lem:2 orb}, the function $\phi_1$ realizing the prescribed values of $\wt \Lambda_j$ can be chosen to depend analytically on $(\alpha_j)$ and $\phi_0$.
It now follows from \eqref{eq:Lamj est} that the $q_j$-periodic orbits $(\xi_j)$ has a Lyapunov exponent given by
\begin{equation*}
    \det(\du_{\xi_j} F^{q_j} - \id) = - (2 + \alpha_j) q_j \Big( \frac{p_j}{q_j} - \omega \Big)^2 \eu^{-2\pi q_j}.
\end{equation*}
Since the Lyapunov exponent is obviously independent of coordinates, we deduce \eqref{eq:main thm lyap}, where the potential $V_\alpha$ is defined by \eqref{eq:RIC potential} with $\phi = \phi_0 + \phi_1$ and the choice of additive constant fixed by the normalization condition on $S$.

\paragraph{Proof of part (1).} 
We impose the following arithmetic condition as in \eqref{eq:main Liouv condition}
\begin{equation}\label{eq:liouv condition}
    \verts{\omega - \frac{p_j}{q_j}} \leq q_j^{-\nu} \eu^{-2\pi q_j} \quad \text{with} \quad \nu = 70.
\end{equation}
Let us write $\CA_j = \CA_{j, +}$ with $\CA_{j, +}$ as in \eqref{eq:CAj def} and provisionally define 
\begin{equation*}
    \wt \CA_j := 2 \pi^2 q_j^3\bigl[1 - 2(1 + f_1^*) \CA_j\bigr] \eu^{2\pi q_j}-2.
\end{equation*}
By analyticity we have similarly that $T: (\wt \sigma_j)_j \mapsto (\wt \CA_j)_j$ is an analytic map.
Combining the arithmetic condition \eqref{eq:liouv condition} with \eqref{eq:CAj est}, 
\begin{equation*}
    \verts{\wt \CA_j - \wt \sigma_j} \lesssim q_j^{\nu_0 + 3 - \nu} \lesssim \frac{1}{q_j^3}
\end{equation*}
In particular, if $q_1$ is large enough, then the right hand side above can be bounded uniformly in $j$ by $<1/16$.
We may in turn apply Lemma \ref{lem:linfty picard} as in the preceding proof of part (2) to conclude that, for all sequence $(\alpha_j)_j$ with $|\alpha_j| < 1/4$ and all $\phi_0$ such that $\|\phi_0\|' \leq 1/5$, there exists a unique sequence $(\wt \sigma_j)_j$ with $|\wt \sigma_j| < 5/16$ such that $\wt \CA_j = \alpha_j$.
By Lemma~\ref{lem:2 orb}, there exists $\phi_1$ depending analytically on $(\alpha_j)$ and $\phi_0$ such that the standard map $F$ associated with $(\omega, \phi = \phi_0 + \phi_1)$ in adapted coordinates has a sequence of $q_j$-periodic orbits $(\xi_j)$ whose symplectic action with respect to the normalized generating function $S$ in adapted coordinates is
\begin{equation*}
     \sum_{k = 0}^{q_j - 1} S \circ F^k (\xi_j) = \frac{q_j \big(\frac{p_j}{q_j} - \omega\big)^2}{2(1 + f_1^*)} \Big[ 1 - \frac{(2 + \alpha_j)\eu^{-2\pi q_j} }{2\pi^2 q_j^3} \Big].
\end{equation*}
By \eqref{eq:RIC action}, the symplectic action of $(\xi_j)$ with respect to the normalized generating function $\sS$ in the original coordinates is 
\begin{equation*}
     \sum_{k = 0}^{q_j-1} \sS \circ \sH  \circ F^k(\xi_j) =\sum_{k = 0}^{q_j-1} S \circ F^k(\xi_j) + \bigg(\omega - \int_\T\phi \dot\phi^+\bigg)p_j
\end{equation*}
where $\sH$ is the coordinate transform defined in \eqref{eq:RIC coord}.
Hence, we have deduced \eqref{eq:main thm action}.

If $(\sigma_j)$ and $\phi_0$ are real valued as assumed in Theorem~\ref{thm:main}, then by Lemma \ref{lem:2 orb} part (1), the standard map has only 2 geometrically distinct $(p_j, q_j)$-orbits, containing $\xi_{j, +}$ and $\xi_{j, -}$ respectively. 
Since $\sigma_j > 0$ by our choice, it follows easily from \eqref{eq:CAj est} and the estimate \eqref{eq:liouv condition} that the orbit of $\xi_{j, +}$ has smaller symplectic action between the two, and hence $\xi_j = \xi_{j, +}$ is a (globally) minimal orbit of rotation number $p_j/q_j$.

Finally, it remains to show that one can choose $V_\alpha$ in such a way that the quantities $f_1^*$ and $(\phi_\alpha \dot\phi_\alpha^+)^*$ are both constant in $\alpha$.
The main idea is that both of two quantities are dominated by the low-frequency components in the Fourier series of $\phi$.
Hence, if $q_1$ is sufficiently large compared to the order of nonzero Fourier coefficients of $\phi_0$, then the choice of $\phi_0$ will have a dominant effect on $f_1^*$ and $(\phi_\alpha \dot\phi_\alpha^+)^*$.
By adjusting $\phi_0$, we can keep the two quantities constant while varying $\alpha$.

Let us formalize this argument. 
We define a map 
\begin{equation*}
    \CG: (\phi_0, \phi_1) \mapsto \big( f_1^*,(\phi \dot\phi^+)^*\big)
\end{equation*}
where $\phi = \phi_0 + \phi_1$ and $f_1$ is defined as in \eqref{eq:f1}.
In the preceding steps, we have constructed an analytic map $(\alpha, \phi_0) \mapsto \phi_1$ where the standard map associated with $(\omega, \phi = \phi_0 + \phi_1)$ has a sequence of closed orbits with specified symplectic actions.
Since $\phi_1$ is determined by $(\alpha, \phi_0)$, we consider $\CG(\phi_0, \phi_1)$ as a function in terms $(\alpha, \phi_0)$ and we apply the implicit function theorem.
In fact, we pick a nonzero $\wt \phi_0 \in \CA_n'$ and consider the derivative
\begin{equation*}
    \frac{\du \CG}{\du \phi_0} = \frac{\partial \CG}{\partial \phi_0} + \frac{\partial \CG}{\partial \phi_1}\frac{\partial \phi_1}{\partial \phi_0}.
\end{equation*}
Pick $\wt \phi_0 \in \CA_n'$ such that the derivative $\frac{\partial \CG}{\partial \phi_0}$ is surjective onto $\C^2$
and take $q_1 \gg n$.
Since 
\begin{equation*}
    \sup_{\Im x = 0} |\phi_1(x)| \leq \eu^{-2\pi q_1},
\end{equation*}
it is easy to establish by Cauchy's estimate that 
$\frac{\partial \CG}{\partial \phi_1}\frac{\partial \phi_1}{\partial \phi_0}$ has arbitrarily small operator norm as $q_1 \to \infty$.
Hence, for $q_1$ sufficiently large, $\du \CG/\du \phi_0$ is surjective and we can apply the implicit function theorem to obtain $\epsilon_0 > 0$ and an analytic map  $\alpha \mapsto \phi_0$ such that $\CG(\phi_0, \phi_1)$ is constant in $\alpha$ for $|\alpha|_{\ell^\infty} < \epsilon_0$.

We have completed the proof of Theorem \ref{thm:main} using Lemma \ref{lem:2 orb}.
We proceed to prove Lemma \ref{lem:2 orb} now.

\subsection{Proof of Lemma~\ref{lem:2 orb}.}

From now on, we fix
\begin{equation*}
    \kappa = 2.
\end{equation*} 
Let $n > 0$ be given by Theorem \ref{thm:fx pt} with $\epsilon = 1/4$ and $\kappa = 2$.
Take $\omega \in \R \setminus \Q$ and a sequence $(p_j, q_j)$ such that \eqref{eq:equidistr arithm} holds for some $\nu \geq 70$.
Let $\phi_0 \in \CA_n'$ and $(\sigma_j)_j$ be given as in Lemma~\ref{lem:2 orb}. 
In particular, we assume 
\begin{equation*}
    |\sigma_j| q_j \geq 1
\end{equation*}
and 
\begin{equation}\label{eq:sigma smallness}
    \sum_j |\sigma_j| + 2 \|\phi_0\|' \leq \frac12.
\end{equation}

\paragraph{Application of Theorem~\ref{thm:fx pt}.}
We apply Theorem~\ref{thm:fx pt} to the function $h \in \CA$ defined by
\begin{equation}\label{eq:h choice}
    h(x):= \sum_{j \geq 1} \sigma_j \eu^{-2\pi q_j} \cos(2\pi q_j x).
\end{equation}
Note that, by the definition of the Fourier-weighted norm, 
\begin{equation*}
  \|h\| = \sum_{j=1}^\infty |\sigma_j|.
\end{equation*}
By \eqref{eq:sigma smallness}, the condition \eqref{eq:h+phi0 0} of Theorem~\ref{thm:fx pt} is satisfied with the choice $\epsilon = 1/4$.

Fix $N_j := \kappa q_j = 2 q_j$. 
Then, by Theorem~\ref{thm:fx pt}, there exists $\phi_1 \in \CA_n'$ depending analytically on $h$ and $\phi_0$ such that the standard map associated with $(\omega, \phi := \phi_0 + \phi_1)$ in the $q_j$–resonant coordinates can be written as
\begin{equation*}
  F_{q_j}(x,y) = \bigl(x + p_j/q_j + y,\; y + g_{q_j}(x,y)\bigr),
\end{equation*}
where the resonant normal form $g_{q_j}$ satisfies \eqref{eq:fx pt 0}.

For notational simplicity, put
\begin{equation*}
    \fg_{q_j} := \langle g_{q_j} \rangle_{q_j}^{N_j}.
\end{equation*}
By \eqref{eq:fx pt 0} and our choice of $h$,
\begin{equation}\label{eq:res}
    \frac{\partial \fg_{q_j}}{\partial x}(x, 0)
    =  \sigma_j \eu^{-2\pi q_j}
       \Big(\frac{p_j}{q_j} - \omega\Big)^2  \cos(2\pi q_j x).
\end{equation}
Integrating \eqref{eq:res} in $x$ gives
\begin{equation}\label{eq:fg(x,0)}
    \fg_{q_j}(x,0)
    =  \frac{\sigma_j\eu^{-2\pi q_j}}{2\pi q_j}
       \Big(\frac{p_j}{q_j} - \omega\Big)^2 \sin(2\pi q_j x)
       \;+\; g_{q_j}^*(0),
\end{equation}
where $g_{q_j}^*(0)$ denotes the average of $x \mapsto g_{q_j}(x,0)$ over $\T$.

By Theorem \ref{thm:rnf}, the functions $\fg_{q_j}$ and $g_{q_j}$ are defined on $\T_{a_j} \times \D_{b_j}$ with
\begin{equation*}
  a_j = 1 - \frac{3}{q_j},
  \qquad
  b_j = \frac{1}{5}\Big|\omega - \frac{p_j}{q_j}\Big|.
\end{equation*}
The conclusions \eqref{eq:rnf concl 2} and \eqref{eq:rnf res compare} respectively read as 
\begin{align}
    &|g_{q_j} - \fg_{q_j}|_{1/q_j, b_j}
      \;\lesssim\; \eu^{- 6\pi q_j}
      \Big|\omega - \frac{p_j}{q_j}\Big|^2, \label{eq:qj fg - g}\\
    &\Big\|
       \Big\langle
         g_{q_j}(\cdot , 0)
         + \frac{1}{2} \Big(\frac{p_j}{q_j} - \omega \Big)^2
           \frac{\langle \dot f_1 \rangle_{q_j}}{1 + \langle f_1 \rangle_{q_j}}
       \Big\rangle_{q_j}^{N_j}
     \Big\|
     \;\lesssim\;
     q_j^{\nu_0 - \nu}
     \Big|\omega - \frac{p_j}{q_j}\Big|^2,\label{eq:qj res compare}
\end{align}
where $f_1$ is as in \eqref{eq:f1} with $\phi = \phi_0 + \phi_1$, and we have absorbed the factor $\kappa \eu^{4\pi\kappa}$ into the implicit constant because $\kappa = 2$ is fixed.

Moreover, by \eqref{eq:gq star}, for $|y| < b_j$ we have
\begin{equation}\label{eq:gq star y}
    |g_{q_j}^*(y)|
    \;\lesssim\;
    q_j^{\nu_0 - \nu}
    \Big(
       \sup_{\Im x = 0} |(g_{q_j} - \fg_{q_j})(x,y)|
       + \sup_{\Im x = 0}|\fg_{q_j}^\bullet(x,y)|
    \Big),
\end{equation}
since
$g_{q_j}^\bullet
  = g_{q_j} - g_{q_j}^*
  = \bigl(g_{q_j} - \fg_{q_j}\bigr)
    + \bigl(\fg_{q_j} - \fg_{q_j}^*\bigr)
  = \bigl(g_{q_j} - \fg_{q_j}\bigr) + \fg_{q_j}^\bullet$
and $\fg_{q_j}^* = g_{q_j}^*$.

For $y = 0$, we have an explicit formula \eqref{eq:fg(x,0)} for $\fg_{q_j}(x,0)$.
Hence, using \eqref{eq:qj fg - g} and the fact that $|\sigma_j| \geq 1/q_j \geq \eu^{-4\pi q_j}$ for $q_j$ sufficiently large, we obtain
\begin{equation}\label{eq:qj g star}
    |g_{q_j}^*(0)|
    \;\lesssim\;
    |\sigma_j|\,q_j^{\nu_0 - \nu-1}
    \eu^{-2\pi q_j}
    \Big|\omega - \frac{p_j}{q_j}\Big|^2.
\end{equation}

For general $|y| < b_j$, 
we have, by $\kappa = 2$ and the estimate for Fourier coefficients,
\begin{equation*}
    \begin{aligned}
        |\fg_{q_j}^\bullet|_{1/q_j, b_j} 
    &\leq \sum_{k \in \pm \{q_j, 2q_j\}}
       \sup_{|y| < b_j}|\wh g_{q_j}(y)_k| \,\eu^{2\pi \frac{|k|}{q_j}} \\
    &\leq \sum_{k \in \pm \{q_j, 2q_j\}}
       \eu^{-2\pi (a_j - \frac{1}{q_j}) |k|}\, |g_{q_j}|_{a_j, b_j}\\
    &\lesssim \eu^{-2\pi q_j} |g_{q_j}|_{a_j,b_j}
      \;\lesssim\;
      \eu^{-2\pi q_j}
      \Big|\omega - \frac{p_j}{q_j}\Big|^2,
    \end{aligned}
\end{equation*}
where in the last inequality we used \eqref{eq:rnf concl 1}.
Combining this with \eqref{eq:gq star y} and \eqref{eq:qj fg - g}, and noting that
\begin{equation*}
  \sup_{\Im x = 0}|(g_{q_j} - \fg_{q_j})(x,y)|
  \;\le\;
  |g_{q_j} - \fg_{q_j}|_{1/q_j,b_j},
\end{equation*}
we obtain
\begin{equation*}
  \sup_{|y|<b_j}|g_{q_j}^*(y)|
  \;\lesssim\;
  \eu^{-2\pi q_j}
  \Big|\omega - \frac{p_j}{q_j}\Big|^2.
\end{equation*}
It follows from $\fg_{q_j} = g_{q_j}^* + \fg_{q_j}^\bullet$ that
\begin{equation}\label{eq:qj fg}
    \bigl|\fg_{q_j}\bigr|_{1/q_j, b_j}
    \;\lesssim\;
    \eu^{-2\pi q_j}
    \Big|\omega - \frac{p_j}{q_j}\Big|^2.
\end{equation}

\paragraph{Almost equidistributed orbits.}
We now show that $F_{q_j}$ has two geometrically distinct $(p_j,q_j)$-periodic orbits and they are \textit{almost equidistributed} in the sense that they closely shadow the sequence of points $(x^* + kp_j/q_j, 0)$, $k = 0,1,\ldots, q_j-1$ for some $x^* \in \T_{1/q_j}$.

To ease the notation, we will suppress the subscripts $j$ and $q_j$ in the following two paragraphs and write simply $(p, q) = (p_j, q_j)$, $N = N_j$, $\sigma = \sigma_j$, $g = g_{q_j}$, $\fg = \fg_{q_j}$, $a = a_j = 1-3/q$ and $b = b_j = |\omega - p/q|/5$.
We will apply the estimates \eqref{eq:qj fg - g}, \eqref{eq:qj res compare}, \eqref{eq:qj g star} and \eqref{eq:qj fg} obtained by the preceding paragraph.
Also recall that, by \eqref{eq:fg(x,0)},
\begin{equation*}
    \fg(x,0) =  \frac{\sigma \eu^{-2\pi q}}{2\pi q}  \Big(\frac{p}{q} - \omega\Big)^2 \sin(2\pi q x) + g^*(0).  
\end{equation*}
The first term on the right hand side has exactly $2q$ simple zeros at $x = k/(2q)$ for $k = 0, \ldots, 2q-1$.
By \eqref{eq:qj g star}, we may view the second term $ g^*(0)$ as a small perturbation of the first term and use Rouch\'e's theorem and periodicity of $\fg$ to conclude that $\fg(\cdot, 0)$ has exactly $2q$ simple zeros in $\T_{1/q}$ which are equidistributed and close to the points $k/(2q)$.
More precisely, for some $x^* \in \T_{a}$ with $|x^*| \lesssim q^{\nu_0 - \nu - 1}$, we have 
\begin{equation}\label{eq:equidistr zeros}
    \fg\Big(x^*_\pm + \frac{k}{q}, 0 \Big) = 0, \quad \text{where $k = 0, \ldots, q-1\ $, $x^*_+ = x^*$ and $x^*_- = \frac{1}{2q} -x^*$.}
\end{equation}
By $1/q$-periodicity of $\fg(\cdot, 0)$, it is readily verified that the points $(x^*_{\pm} + k/q, 0)$ are $(p,q)$-periodic under the `reduced' map 
\begin{equation}\label{eq:reduced F}
    \fF_q(x,y) = (x + p/q + y, y + \fg(x,y)).
\end{equation}
In view of \eqref{eq:qj fg - g}, we may consider the map 
\begin{equation*}
    F_q(x,y) = (x + p/q + y, y + g(x,y))
\end{equation*}
as a small perturbation of \eqref{eq:reduced F}.
We will show in the following lemma that the periodic orbits of $\fF_q$ starting from $(x^*_{\pm} + k/q, 0)$ persist under this perturbation.

\begin{lemma}\label{lem:equidistr orb}
    Let $x_\pm^* + k/q$ be the zeros of $\fg(\cdot,0)$ given by
    \eqref{eq:equidistr zeros}.
    Set
    \begin{equation*}
      a' := \frac{1}{2q},
      \qquad
      b' := \frac{1}{6}\biggl|\omega - \frac{p}{q}\biggr|,
      \qquad
      \alpha_\pm^k := \frac{\partial \fg}{\partial x}
        \Bigl(x_\pm^* + \frac{k}{q}, 0\Bigr).
    \end{equation*}
    Then, for $q$ sufficiently large (depending only on the universal implicit
    constants in the estimates below), the following holds.

    There exist exactly $2q$ distinct points
    \begin{equation*}
      (\wt x_+^k, \wt y_+^k),
      \qquad
      (\wt x_-^k, \wt y_-^k),
      \qquad k = 0,\dots,q-1,
    \end{equation*}
    lying in $\T_{a'} \times \D_{b'}$ which are $(p,q)$-periodic under the map $F_q$.
    More precisely, these $2q$ points form exactly two geometrically distinct $(p,q)$-periodic orbits of $F_q$:
    for each sign $\pm$,
    \begin{equation*}
      F_q^n(\wt x_\pm^k, \wt y_\pm^k)
      = \bigl(\wt x_\pm^{k+np}, \wt y_\pm^{k+np}\bigr),
      \qquad n \in \Z,
    \end{equation*}
    where the index $k+np$ is taken modulo $q$.

    Furthermore, the points $(\wt x_\pm^k, \wt y_\pm^k)$ depend analytically on
    the parameters $(\sigma_j)_j$ and on $\phi_0$, and the following estimates hold:
    \begin{gather}
        \bigl|\wt x_\pm^k - (x_\pm^* + k/q)\bigr|
        \;\lesssim\;
        q\, \eu^{-4\pi q},
        \qquad
        |\wt y_\pm^k|
        \;\lesssim\;
        q^3 \eu^{-6\pi q}
        \biggl|\omega - \frac{p}{q}\biggr|^2,
        \label{eq:equidistr orb est}
        \\
        \bigl|
          \det\bigl(\du_{(\wt x_\pm^k,\wt y_\pm^k)}F_q^q - \id\bigr)
          + q^2 \alpha_\pm^k
        \bigr|
        \;\ll\;
        q^2 |\alpha_\pm^k|.
        \label{eq:equidistr orb det est}
    \end{gather}
\end{lemma}

\begin{remark}
    In the above lemma and its proof, we use the notation $A \ll B$ to mean that for all $\varepsilon > 0$ there exists $Q > 0$ depending only on $\varepsilon$ such that $|A/B| < \varepsilon$ for all $q \geq Q$.
\end{remark}

\begin{proof}[Proof of Lemma~\ref{lem:equidistr orb}.]
    Consider the interpolation map 
    \begin{equation*}
        F^\tau_q(x,y) = (x + p/q + y, y + \fg(x,y) + \tau (g(x,y) - \fg(x,y))).
    \end{equation*}
    Since $b \ll q^{-2}$ and $(3q - 1)(|\fg|_{1/q,b} + |g - \fg|_{1/q,b}) \ll b$ by \eqref{eq:qj fg} and \eqref{eq:qj fg - g}, we may apply Lemma \ref{lem:gam} (with $f = 0$ and $a = 1/q$) to $F^\tau_q$ for $|\tau|\leq R$ with some $R > 1$ to obtain $\gamma^\tau, \Gamma^\tau: \T_{a'} \to \C$ depending analytically on $\tau$.
    Using Cauchy's estimate in $\tau$ and choosing $R$ such that, for example, $(3q-1) (|\fg|_{1/q,b} + R|\fg - g|_{1/q,b}) = b$, we have
    \begin{equation}\label{eq:Gam01 approx}
        \verts{\Gamma^1 - \Gamma^0}_{a'} \leq \frac{\sup_{|\tau| \leq R} |\Gamma^\tau|_{a'}}{R - 1} \lesssim q |\fg - g|_{1/q,b}.
    \end{equation}
    We first consider $k = 0$ and look for the periodic orbit near the point $(x^*_+, 0)$.
    The cases for $x^*_-$ and other $k$ can be treated similarly.
    Observe in particular that by the explicit formula \eqref{eq:fg(x,0)} for $\fg(\cdot, 0)$ and the equality $x^*_- = (2q)\inv - x^*_+$, the quantity $\alpha^k_\pm$ is independent of $k$ and $\alpha^0_+ = -\alpha^0_-$.

    From now on until the end of the proof, we write $x^* = x^*_+$ for simplicity of notation.
    Let us put $\alpha = \alpha^0_+ = \partial \fg/\partial x(x^*, 0)$.
    By the explicit formula \eqref{eq:fg(x,0)},
    \begin{equation}\label{eq:alpha explicit}
        \alpha = \sigma\eu^{-2\pi q}  \Big(\frac{p}{q} - \omega\Big)^2 \cos(2\pi q x^*).
    \end{equation}
    Since $(x^*, 0)$ is $(p,q)$-periodic under $\fF_q$, it follows from Lemma \ref{lem:gam} that $\Gamma^0(x^*) = 0$ and $\gamma^0(x^*) = 0$.
    By \eqref{eq:Gam approx} and periodicity of $\fg(\cdot, 0)$ as well as the estimates \eqref{eq:qj fg} and $b \ll q^{-2}$
    \begin{equation}\label{eq:Gam0 - q fg}
        \verts{\Gamma^0 - q \fg(\cdot, 0)}_{a'} \lesssim \frac{q^2 |\fg|_{1/q,b}^2}{b}.
    \end{equation}
    By Cauchy's estimate and $|x^*| \lesssim q^{\nu_0 - \nu - 1} \ll a'$, we have
    \begin{equation}\label{eq:dotGam0 - q alpha}
        \verts{\dot\Gamma^0(x^*) - q \alpha} \lesssim \frac{q^3 |\fg|_{1/q,b}^2}{b} \ll q |\alpha|.
    \end{equation}
    The last $\ll$ can be obtained from the explicit \eqref{eq:alpha explicit}.
    We apply Lemma \ref{lem:zero pert} with $R = a' - |x^*|$,  $\eta_0 = \Gamma^0$ and $\eta_1 = \Gamma^1 - \Gamma^0$.
    To verify the condition \eqref{eq:zeroPt cond}, observe that, by \eqref{eq:Gam01 approx}, \eqref{eq:qj fg - g} and \eqref{eq:qj fg},
    \begin{equation*}
        r_0 r_1 \leq \frac{|\Gamma^0|_{a'} |\Gamma^0 - \Gamma^1|_{a'}}{(a' - |x^*|)^2} \lesssim q^4 |\fg|_{1/q,b} |g - \fg|_{1/q,b} \ll q^2 |\alpha|^2.
    \end{equation*}
    Combining with \eqref{eq:dotGam0 - q alpha}, we readily verify the condition \eqref{eq:zeroPt cond} for $q$ sufficiently large.
    Hence, the Lemma applies and we obtain a zero point $\wt x$ of $\Gamma^1$ satisfying
    \begin{equation}\label{eq:wt x0 - x*}
        |\wt x - x^*| \lesssim \frac{|\Gamma^1 - \Gamma^0|_{a'}}{|\dot \Gamma^0(x^*)|} \lesssim \frac{|\fg - g|_{1/q,b}}{|\alpha|} \lesssim q \eu^{-4\pi q}.
    \end{equation}
    Let us put $\wt y := \gamma^1(\wt x)$.
    Then, since $\gamma^0(x^*) = 0$, by \eqref{eq:Gam01 approx}, \eqref{eq:wt x0 - x*} and the triangle inequality,
    \begin{equation*}
        \begin{aligned}
            |\wt y| \leq& |\gamma^1(\wt x) - \gamma^0(\wt x)| + |\gamma^0(\wt x) - \gamma^0(x^*)|\\
            \lesssim& q |\fg - g|_{1/q,b} + \frac{q^2 |\fg|_{1/q,b} |\fg - g|_{1/q,b}}{|\alpha|}\\
            \lesssim& \frac{q^2 |\fg|_{1/q,b} |\fg - g|_{1/q,b}}{|\alpha|} \\
            \lesssim& q^3 \eu^{-6\pi q} |\omega - p/q|^2.
        \end{aligned}
    \end{equation*}
    In the second last line, we have used $|\alpha| \lesssim q|\fg|_{1/q,b}$, which follows from Cauchy's estimate.
    We have established the estimates \eqref{eq:equidistr orb est} for $k = 0$.

    For \eqref{eq:equidistr orb det est}, we decompose
    \begin{equation}\label{eq:det est decomp}
        \begin{aligned}
            &|\det(\du_{(\wt x, \wt y)} F_q^q - \id) + q^2 \alpha|\\ 
            \leq& |\det(\du_{(\wt x, \wt y)} F_q^q - \id) + q \dot\Gamma^1(\wt x)| + |q \dot\Gamma^1(\wt x) - q \dot\Gamma^0(x^*)| + |q \dot\Gamma^0(x^*) - q^2 \alpha|.
        \end{aligned}
    \end{equation}
    By \eqref{eq:dotGam0 - q alpha}, the last term on the right hand side is $\ll q^2 |\alpha|$.
    By the second inequality of \eqref{eq:zero pert}, the second term on the right hand side of \eqref{eq:det est decomp} can be estimated as
    \begin{equation*}
        |q \dot\Gamma^1(\wt x) - q \dot\Gamma^0(x^*)| \lesssim  \frac{q^4 |\fg|_{1/q,b} |g - \fg|_{1/q,b}}{|\alpha|} \ll q^2 |\alpha|.
    \end{equation*}
    Finally, by \eqref{eq:lyap bnd}, the first term on the right hand side of \eqref{eq:det est decomp} can be estimated as
    \begin{equation*}
        |\det(\du_{(\wt x, \wt y)} F_q^q - \id) + q \dot\Gamma^1(\wt x)| \lesssim \frac{q^4 |g|_{1/q,b}^2}{b} \ll q^2 |\alpha|.
    \end{equation*}
    This completes the proof of \eqref{eq:equidistr orb det est} for $(\wt x, \wt y) = (\wt x^0_+, \wt y^0_+)$.
    The cases for other $k$ can be treated similarly by replacing $x^*$ with $x^*_\pm + k/q$ in the above argument.

    Finally, we show the uniqueness of the $(p,q)$-periodic points in $\T_{a'} \times \D_{b'}$, where $b' = |\omega - p/q|/6$ as in the statement.
    Since $q^{-4} \ll |\alpha|/(q^2 |\fg|_{1/q,b}) \lesssim |\dot\Gamma^0(x^*)|(a' - |x^*|)^2/|\Gamma^0|_{a'}$, Lemma \ref{lem:zero pert} implies that $\wt x^k_{\pm}$ is the unique zero of $\Gamma^1$ in the $q^{-4}$-neighborhood of $x^*_\pm + k/q$.
    Hence, any $(p,q)$-periodic point of $F_q$ lying in a $q^{-4}$-neighborhood of some $x^*_\pm + k/q$ in the $x$-coordinate must coincide with the point $(\wt x^k_\pm, \wt y^k_\pm)$ constructed above.
    We need to show that there is no $(p,q)$-periodic point in $\T_{a'} \times \D_{b'}$ which lies outside these neighborhoods.
    By the explicit expression for $\fg(\cdot, 0)$,
    \begin{equation}\label{eq:fg lower bound}
        q^{-6} \eu^{-2\pi q}|\omega - p/q|^2 \ll |\fg(x, 0)| 
    \end{equation}
    for $x \in \T_{a'}$ lying outside the union of the $q^{-4}$-neighborhoods of the points $x^*_\pm + k/q$.
    On the other hand, by \eqref{eq:Gam0 - q fg} and \eqref{eq:Gam01 approx}, we have
    \begin{equation}\label{eq:Gam1 - q fg}
        \verts{\Gamma^1 - q \fg(\cdot, 0)}_{a'} \lesssim \frac{q^2 |\fg|_{1/q,b}^2}{b} + q|\fg - g|_{1/q,b} \lesssim q^2 \eu^{-4\pi q} |\omega - p/q|^3 + q \eu^{-6\pi q} |\omega - p/q|^2.
    \end{equation}
    Combining \eqref{eq:fg lower bound} and \eqref{eq:Gam1 - q fg}, we deduce that $\Gamma^1$ does not vanish outside the $q^{-4}$-neighborhoods of the points $x^*_\pm + k/q$, which implies via Lemma \ref{lem:gam} that the points $(\wt x^k_\pm, \wt y^k_\pm)$ constructed above are the only $(p,q)$-periodic points of $F_q$ in $\T_{a'} \times \D_{b'}$.
    It is also simple to verify, for example using the formula for $F_q$ and $|\wt y^k_\pm| \ll q^{-4}$, that we must have $F_q^n(\wt x^k_\pm, \wt y^k_\pm) = (\wt x^{k + np}_\pm, \wt y^{k + np}_\pm)$.
    Finally, the analyticity of the points $(\wt x^k_\pm, \wt y^k_\pm)$ in terms of $g$ follows from the construction using Lemma \ref{lem:zero pert} and Lemma \ref{lem:gam}.
    The proof is complete.
\end{proof}

The following lemma was used in the proof of Lemma~\ref{lem:equidistr orb}. This is an elementary result in complex analysis, but we nevertheless give a proof for clarity.

\begin{lemma}\label{lem:zero pert}
    Let $R > 0$ and $\eta_0, \eta_1: \D_R \to \C$ be bounded analytic functions such that $\eta_0(0) = 0$ and $\eta_0'(0) \neq 0$.
    Let us denote $r_j = R\inv\sup\{|\eta_j(z)| : |z| < R\}$ for $j = 0, 1$.
    Suppose 
    \begin{equation}\label{eq:zeroPt cond}
        r_0 r_1 < c |\eta_0'(0)|^2
    \end{equation}
    for some universal constant $c$.
    Then there exists a unique zero point $x$ of $\eta_0 + \eta_1$ within the open disc $\D_{\bar R}$ with $$\bar R = 2c\frac{|\eta_0'(0)|}{r_0} R.$$
    Moreover, we have
    \begin{equation}\label{eq:zero pert}
        \frac{|x|}{R} \lesssim \frac{r_1}{|\eta_0'(0)|} \quad \text{and} \quad 
        |\eta_0'(x) + \eta_1'(x) - \eta_0'(0)| \lesssim \frac{r_0 r_1}{|\eta_0'(0)|}.
    \end{equation}
\end{lemma}
\begin{proof}
    After replacing $\eta_j(x)$ by $\eta_j(Rx)/(R \eta_0'(0))$ we may assume $R = 1$ and $\eta_0'(0) = 1$ so that $r_j$ is the uniform norm of $\eta_j$ over the unit disc.
    By the mean value theorem we can write 
    \begin{equation*}
        \eta_0(x) = x + x^2 \int_0^1 \eta_0''(tx) (1-t) dt.
    \end{equation*}
    By Cauchy's estimate applied to $|x| < 1/2$, we have $|\eta_0(x) - x| < |x|/2$ if $|x| < (16r_0)\inv$.
    Hence, if $r_1 < (32 r_0)\inv$, then Rouch\'e's theorem implies that there exists a unique simple zero $x^*$ of $\eta_0 + \eta_1$ in the disc $|x| < (16 r_0)\inv$.
    In fact, Rouch\'e's theorem applies to the circle of radius $2r_1$ and we deduce $|x^*| \lesssim r_1$.
    By the mean value theorem and Cauchy's estimate, we have $|\eta_0'(x^*) + \eta_1'(x^*) - \eta_0'(0)| \lesssim r_0 |x^*| + r_1 \lesssim r_0 r_1$.
    In the last $\lesssim$ we used the fact $r_0 \geq 1$, which follows from Cauchy's estimate of $|\eta_0'(0)|$.
    After scaling $\eta_0$ and $\eta_1$ back to the original functions, we obtain the desired conclusions with $c = 1/32$.
\end{proof}

We complement Lemma~\ref{lem:equidistr orb} with the following corollary showing the uniqueness of the orbits that it constructs.

\begin{corollary}\label{cor:only 2}
    Suppose that $\phi$ is real-valued, so that the standard map $F$ acts on
    the real locus $\T \times \R$ in the adapted coordinates.
    Then the $(p,q)$-periodic points $(\wt x_\pm^k, \wt y_\pm^k)$ of $F_q$
    obtained in Lemma~\ref{lem:equidistr orb} lie in $\T\times\R$, and these are the only
    $(p,q)$-periodic orbits of $F$ contained in the real locus.
\end{corollary}

\begin{proof}[Proof of Corollary~\ref{cor:only 2}]
    Let $\phi$ be real-valued so that $F$ acts on $\T \times \R$.
    By Lemma \ref{lem:equidistr orb}, the points $(\wt x^k_\pm, \wt y^k_\pm)_{k = 0}^{q-1}$ in Lemma \ref{lem:equidistr orb} exhaust all $(p,q)$-periodic points in the domain $\T_{a'} \times \D_{b'}$ of $F_q$. 
    To prove the required statement it is enough to show that all $(p,q)$-orbits of $F$ must intersect the subset $H_q(\T_{a'} \times \D_{b'})$.
    For this, it suffices to show that every $(p,q)$-periodic orbit of $F$ (in the real locus $\T \times \R$) must intersect the real annulus $A$ defined as 
    \begin{equation*}
        A := \T \times \bigg\{ y \in \R \ : \ \bigg|y - \Big(\frac{p}{q} - \omega \Big) \bigg| < \frac{1}{15} \verts{\omega - \frac{p}{q}} \bigg\}.
    \end{equation*}
    Indeed, the estimate \eqref{eq:h12 est} clearly implies $A \subset H_q(\T_{a'} \times \D_{b'})$.

    We also recall that the conditions $\|\dddot\phi\| \leq 1$ and $\phi^* = 0$ imposed by Theorem \ref{thm:rnf} automatically guarantee that $\|\dot\phi\| \leq (2\pi)^{-2}$ by definition of the $\|\cdot\|$-norm.
    Then, by the mean value theorem applied to \eqref{eq:fg} and the bound on $\|\dot\phi\|$, it is straightforward to show (for example using the formula \eqref{eq:bar fg int}) that $\max(2|f(x,y)|,|g(x,y)|) < |y|/8$. 

    Now, let $(x_k, y_k) = F^{k}(x_0, y_0)$ be a $(p,q)$-periodic orbit under $F$ lifted to the universal cover. In particular, $x_q = x_0 + p$ and $y_q = y_0$.
    Assume for a contradiction that $(x_k, y_k)$ does not intersect $A$ for any $k$.
    Then we observe that $y_k - (p/q - \omega)$ must have the same sign for all $k$.
    For, if not, then we have $|y_k - y_{k+1}| \geq \frac{2}{15}|\omega - p/q|$ and $|y_k| \leq \frac{14}{15} |\omega - p/q|$ for some $k$. 
    By \eqref{eq:F RIC} and \eqref{eq:fg}, we deduce $\frac{2}{15}|\omega - p/q| \leq |g(x_k,y_k)|$.
    However, this implies $\frac{2}{15}|\omega - p/q| < |y_k|/8$, which contradicts $|y_k| < |\omega - p/q|$.
    
    We have shown that all points $(x_k, y_k)$ of the orbit lie on the same side of $A$.
    Suppose for example $p/q - \omega > 0$ and $y_k \geq \frac{16}{15}(p/q - \omega)$ for all $k$.
    Since $x_q - x_0 = p$, we have $x_{k+1} - x_k \leq p/q$ for some $k$, which implies $y_k + f(x_k, y_k) \leq p/q - \omega$.
    On the other hand, since $|f(x_k, y_k)| < |y_k|/16$ we have $y_k < (p/q - \omega)/(1- 1/16)$, which contradicts $y_k > \frac{16}{15}(p/q - \omega)$.
    Hence, we have shown that any $(p,q)$-orbit of $F$ must intersect $A$, which by $A \subset H_q(\T_{a'} \times \D_{b'})$ implies that it intersects $H_q(\T_{a'} \times \D_{b'})$.
    The cases $y_k \leq \frac{14}{15}(p/q - \omega)$ for all $k$ and the cases with $p/q - \omega < 0$ can be treated similarly.
    It follows from this argument that this orbit must intersect $A$ and it must be one of the $(p,q)$-orbits obtained by Lemma~\ref{lem:equidistr orb}.
\end{proof}

\paragraph{End of proof of Lemma~\ref{lem:2 orb}.}

Finally, we show that for each $j$ the two periodic orbits constructed in Lemma~\ref{lem:equidistr orb} satisfy the conclusions of Lemma~\ref{lem:2 orb}, thereby completing its proof.

The uniqueness statement in part~(1) of Lemma~\ref{lem:2 orb} follows immediately from Corollary~\ref{cor:only 2}.
We now verify the estimates \eqref{eq:Lamj est} and \eqref{eq:CAj est} for these periodic orbits.

We begin with \eqref{eq:Lamj est}.
In the case of the zero $x_-^* + k/q$, the corresponding quantity $\alpha_-^k$ in \eqref{eq:equidistr orb det est} is given by
\begin{equation*}
    \alpha_-^k = -\sigma \eu^{-2\pi q} \Bigl(\frac{p}{q} - \omega\Bigr)^2
    \cos(2\pi q x^*).
\end{equation*}
Thus,
\begin{equation*}
    \verts{
      \frac{\det\bigl(\du_{(\wt x_\pm^k, \wt y_\pm^k)}F_q^q - \id\bigr)}
           {q^2 (p/q - \omega)^2}
      \pm \sigma \eu^{-2\pi q} \cos(2\pi q x^*)
    }
    \;\ll\; |\sigma| \eu^{-2\pi q}.
\end{equation*}
From the estimate $|x^*| \lesssim q^{\nu_0 - \nu - 1} \ll 1/q$ we obtain
$|\cos(2\pi q x^*) - 1| \ll 1$.
The desired estimate \eqref{eq:Lamj est} now follows from the definition of
$\Lambda_{j,\pm}$ by taking $q$ sufficiently large.

We next turn to \eqref{eq:CAj est}.
Recall that, by \eqref{eq:rnf S compare}, the function $S_q$ satisfies
\begin{equation}\label{eq:Sq approx recall}
    \biggl|
      S_q(x,y)
      - \Bigl(\frac{p}{q} - \omega + y\Bigr)^2
        \frac{1 + f_1(x)}{2\bigl(1 + \langle f_1 \rangle_q(x)\bigr)^2}
    \biggr|
    \;\lesssim\;
    q^{\nu_0} \biggl|\omega - \frac{p}{q}\biggr|^3,
\end{equation}
with $\nu_0 = 64$.
Substituting $x = \wt x_\pm^k$, $y = \wt y_\pm^k$ and using the mean value theorem
together with \eqref{eq:equidistr orb est}, we obtain
\begin{equation}\label{eq:equidistr bound S}
    \begin{aligned}
        &
        \biggl|
          \Bigl(\frac{p}{q} - \omega + \wt y_\pm^k\Bigr)^2
          \frac{1 + f_1(\wt x_\pm^k)}{2\bigl(1 + \langle f_1 \rangle_q(\wt x_\pm^k)\bigr)^2}
          -
          \Bigl(\frac{p}{q} - \omega\Bigr)^2
          \frac{1 + f_1(x_\pm^* + k/q)}{2\bigl(1 + \langle f_1 \rangle_q(x_\pm^* + k/q)\bigr)^2}
        \biggr| \\
        &\qquad\lesssim\;
        \biggl\|
           \frac{1 + f_1}{2(1 + \langle f_1 \rangle_q)^2}
        \biggr\|
        \biggl|\omega - \frac{p}{q}\biggr|\, |\wt y_\pm^k|
        +
        \biggl\|
           \frac{\du}{\du x}
           \Bigl(
             \frac{1 + f_1}{2(1 + \langle f_1 \rangle_q)^2}
           \Bigr)
        \biggr\|
        \biggl|\omega - \frac{p}{q}\biggr|^2
        \biggl|\wt x_\pm^k - \Bigl(x_\pm^* + \frac{k}{q}\Bigr)\biggr| \\
        &\qquad\lesssim\;
        \biggl|\omega - \frac{p}{q}\biggr|\, |\wt y_\pm^k|
        +
        \biggl|\omega - \frac{p}{q}\biggr|^2
        \biggl|\wt x_\pm^k - \Bigl(x_\pm^* + \frac{k}{q}\Bigr)\biggr|
        \;\lesssim\;
        q \eu^{-4\pi q} \biggl|\omega - \frac{p}{q}\biggr|^2.
    \end{aligned}
\end{equation}
In the last line we have used the basic estimates
$\|f_1\| < 1/2$ and $\|\dot f_1\| \lesssim \|\ddot\phi\|$, which follow from the definition~\eqref{eq:f1}.

Summing over $k = 0,\dots,q-1$ yields
\begin{equation}\label{eq:action q-proj}
    \sum_{k=0}^{q-1}
      \frac{1 + f_1(x_\pm^* + k/q)}{2(1 + \langle f_1 \rangle_q(x_\pm^* + k/q))^2}
    =
    q\biggl\langle
        \frac{1 + f_1}{2(1 + \langle f_1 \rangle_q)^2}
      \biggr\rangle_q(x_\pm^*)
    =
    \frac{q}{2\bigl(1 + f_1^* + \langle f_1^\bullet \rangle_q(x_\pm^*)\bigr)}.
\end{equation}
Recall that $x_+^* = x^*$ and $x_-^* = (2q)^{-1} - x^*$ with
$|x^*| \lesssim q^{\nu_0 - \nu - 1} \ll 1/q$.
In particular, we have $|\Im x^*_\pm| \ll 1$.
By Cauchy's estimate and the mean value theorem we obtain
\begin{equation}\label{eq:xq star vs 0}
    \bigl|
      \langle f_1^\bullet \rangle_q(x_+^*)
      - \langle f_1^\bullet \rangle_q(0)
    \bigr|
    \;\lesssim\;
    q^{\nu_0 - \nu} |\langle f_1^\bullet \rangle_q|_{1/q}
    \;\lesssim\;
    q^{\nu_0 - \nu}
    \biggl(\sum_{\substack{|k|\ge q \\ q\mid k}}
       |f_1| \eu^{-2\pi(1-1/q)|k|}
    \biggr)
    \;\lesssim\;
    q^{\nu_0 - \nu} \eu^{-2\pi q}.
\end{equation}
The second $\lesssim$ follows by estimating the Fourier coefficients of $f_1$ individually using analyticity over $\T_1$.
Similarly
\begin{equation}\label{eq:xq star- vs 0}
    \biggl|
      \langle f_1^\bullet \rangle_q(x_-^*)
      - \langle f_1^\bullet \rangle_q\Bigl(\frac{1}{2q}\Bigr)
    \biggr|
    \;\lesssim\;
    q^{\nu_0 - \nu} \eu^{-2\pi q}.
\end{equation}
Moreover, for $\Im x = 0$ we have
$|\langle f_1 \rangle_q^{>q}(x)|
 \leq \sum_{|k|>q,\;q\mid k} |f_1| \eu^{-2\pi|k|}
 \lesssim \eu^{-4\pi q}$.
Combining this with \eqref{eq:xq star vs 0} and \eqref{eq:xq star- vs 0} gives
\begin{equation*}
    \biggl|
      \pm\bigl((\hat f_1)_q + (\hat f_1)_{-q}\bigr)
      - \langle f_1^\bullet \rangle_q(x_\pm^*)
    \biggr|
    \;\lesssim\;
    q^{\nu_0 - \nu} \eu^{-2\pi q},
\end{equation*}
where $(\hat f_1)_{\pm q}$ are the Fourier coefficients of $f_1$ of order $\pm q$.
It follows that
\begin{equation}\label{eq:reduction to vee}
    \biggl|
      \frac{1}{1 + f_1^* \pm ((\hat f_1)_q + (\hat f_1)_{-q})}
      - \frac{1}{1 + f_1^* + \langle f_1^\bullet \rangle_q(x_\pm^*)}
    \biggr|
    \;\lesssim\;
    q^{\nu_0 - \nu} \eu^{-2\pi q}.
\end{equation}

Let $\CA_\pm$ denote the rescaled action defined by the right-hand side of
\eqref{eq:CAj def} with $\xi_j$ taken to be the $(p,q)$-orbit of
$(\wt x_\pm^k, \wt y_\pm^k)$ under $F_q$, i.e.
\begin{equation*}
    \CA_\pm
    :=
    \frac{1}{q(\omega - p/q)^2}
    \sum_{k=0}^{q-1} S_q(\wt x_\pm^k, \wt y_\pm^k).
\end{equation*}
Combining \eqref{eq:Sq approx recall}, \eqref{eq:equidistr bound S},
\eqref{eq:action q-proj}, and \eqref{eq:reduction to vee} yields
\begin{equation*}
    \biggl|
      \CA_\pm
      - \frac{1}{2\bigl(1 + f_1^* \pm ((\hat f_1)_q + (\hat f_1)_{-q})\bigr)}
    \biggr|
    \;\lesssim\;
    q^{\nu_0} \biggl|\omega - \frac{p}{q}\biggr|
    + q^{\nu_0 - \nu} \eu^{-2\pi q}.
\end{equation*}
Since $|(\hat f_1)_{\pm q}| \lesssim \eu^{-2\pi q}$, up to an additional error
$\lesssim q^{\nu_0 - \nu} \eu^{-2\pi q}$ we may write
\begin{equation}\label{eq:normalized action est}
    \biggl|
      \CA_\pm
      - \frac{1}{2(1 + f_1^*)}
        \biggl(
          1 \mp \frac{(\hat f_1)_q + (\hat f_1)_{-q}}{1 + f_1^*}
        \biggr)
    \biggr|
    \;\lesssim\;
    q^{\nu_0} \biggl|\omega - \frac{p}{q}\biggr|
    + q^{\nu_0 - \nu} \eu^{-2\pi q}.
\end{equation}

We now control $(\hat f_1)_q + (\hat f_1)_{-q}$ via the choice of the coefficients
$\sigma_j$ in \eqref{eq:h choice}.
Recall from \eqref{eq:qj res compare} and the definition of $\fg$ that
\begin{equation*}
    \Bigl\|
      \fg^\bullet(\cdot,0)
      + \frac{1}{2}\Bigl(\frac{p}{q} - \omega\Bigr)^2
        \biggl\langle
          \frac{\langle \dot f_1 \rangle_q}{1 + \langle f_1 \rangle_q}
        \biggr\rangle_q^N
    \Bigr\|
    \;\lesssim\;
    q^{\nu_0 - \nu} \biggl|\omega - \frac{p}{q}\biggr|^2.
\end{equation*}
Since $N = 2q$, the $k$th Fourier coefficient of the expression inside the
$\|\cdot\|$-norm is nonzero only for $|k| = q$ or $2q$, and analyticity in $\T_1$
gives a bound $q^{\nu_0 - \nu} \eu^{-2\pi q} |\omega - p/q|^2$ for such coefficients.
Hence the supremum norm over the real torus $\T$ is bounded by
$\lesssim q^{\nu_0 - \nu} \eu^{-2\pi q} |\omega - p/q|^2$.
A similar argument shows that the supremum norm over $\T$ of
$\langle\langle \dot f_1 \rangle_q (1 + \langle f_1 \rangle_q)^{-1}\rangle^{>N}$
is bounded by $\lesssim \eu^{-6\pi q}$.
For $\Im x = 0$, the explicit formula \eqref{eq:fg(x,0)} then yields
\begin{equation}\label{eq:sig-f1 est}
    \biggl|
      \frac{\sigma \eu^{-2\pi q}}{2\pi q} \sin(2\pi q x)
      + \frac{1}{2}\,
        \frac{\langle \dot f_1 \rangle_q(x)}
             {1 + \langle f_1 \rangle_q(x)}
    \biggr|
    \;\lesssim\;
    q^{\nu_0 - \nu} \eu^{-2\pi q}.
\end{equation}

We decompose
$\langle f_1 \rangle_q
 = f_1^* + \langle f_1 \rangle_q^q + \langle f_1 \rangle_q^{>q}$.
Analyticity and $\Im x = 0$ imply
$|\langle \dot f_1 \rangle_q^{>q}(x)| \lesssim q \eu^{-4\pi q}$,
$|\langle f_1 \rangle_q^q| \lesssim \eu^{-2\pi q}$,
and $|\langle \dot f_1 \rangle_q^q| \lesssim q \eu^{-2\pi q}$, hence
\begin{equation*}
    \biggl|
      \frac{\langle \dot f_1 \rangle_q(x)}
           {1 + \langle f_1 \rangle_q(x)}
      - \frac{\langle \dot f_1 \rangle_q^q(x)}{1 + f_1^*}
    \biggr|
    \;\lesssim\;
    q \eu^{-4\pi q}
    \;\lesssim\;
    q^{\nu_0 - \nu} \eu^{-2\pi q}.
\end{equation*}
Substituting this into \eqref{eq:sig-f1 est} and comparing the Fourier coefficients
of order $\pm q$, we obtain
\begin{equation*}
    \biggl|
      \frac{\sigma \eu^{-2\pi q}}{(2\pi q)^2}
      - \frac{1}{2}\,
        \frac{(\hat f_1)_q + (\hat f_1)_{-q}}{1 + f_1^*}
    \biggr|
    \;\lesssim\;
    q^{\nu_0 - \nu} \eu^{-2\pi q}.
\end{equation*}
Combining this with \eqref{eq:normalized action est}, we arrive at
\begin{equation*}
    \biggl|
      \CA_\pm
      - \frac{1}{2(1 + f_1^*)}
        \Bigl(
          1 \mp \frac{2\sigma \eu^{-2\pi q}}{(2\pi q)^2}
        \Bigr)
    \biggr|
    \;\lesssim\;
    q^{\nu_0} \biggl|\omega - \frac{p}{q}\biggr|
    + q^{\nu_0 - \nu} \eu^{-2\pi q}.
\end{equation*}
This is precisely \eqref{eq:CAj est}.
The proof of Lemma~\ref{lem:2 orb} is complete.

We have now completed all steps in the proof of Theorem~\ref{thm:main}, with the sole exception of the technical intermediate result Theorem~\ref{thm:rnf}.
The remainder of the article will be devoted to the proof of Theorem~\ref{thm:rnf}.

\section{Proof of Theorem \ref{thm:rnf}}\label{sec:rnf proof}

Before entering the proof of Theorem \ref{thm:rnf}, we broadly outline the main steps.

We start with the standard map \eqref{eq:F RIC} in adapted coordinates, where the RIC is given by $\{y = 0\}$. 
For a fixed pair $(p, q)$, we focus on a complex neighbourhood of the line $\{y = p/q - \omega\}$.
On this neighbourhood, we construct two successive changes of coordinates, denoted by $\Theta$ and $\id + \Sigma$, in Section~\ref{sec:rnf part 1} and Section~\ref{sec:rnf part 2}, respectively.

The common idea behind the construction of these two coordinate changes is to consider a near-identity change of variables 
\begin{equation}\label{eq:rnf Phi}
\Phi = \id + \Psi \quad \text{where} \quad \Psi(x,y) = (\psi(x,y),\varphi(x,y))
\end{equation}
for some functions $\psi$ and $\varphi$.
Put $A(x,y) = (x + \omega + y, y)$ and $G(x,y) = (f(x,y), g(x,y))$. 
Let 
\begin{equation*}
    \Phi \circ F \circ \Phi\inv = \wt F = A + \wt G
\end{equation*}
where $\wt G(x,y) = (\wt f(x,y), \wt g(x,y))$.
Expanding $\wt F \circ \Phi = \Phi \circ F$, we have 
\begin{equation*}
    \wt G \circ (\id + \Psi) = G + \Psi \circ F + A - A \circ (\id + \Psi).
\end{equation*}
Componentwise, this is
\begin{equation}\label{eq:tf tg}
    \begin{cases}
        \wt f \circ (\id + \Psi) = f + \psi \circ F - \psi - \varphi, \\
        \wt g \circ (\id + \Psi) = g + \varphi \circ F - \varphi.
    \end{cases}
\end{equation}
To simplify $\wt F$ as much as possible, we shall make $\wt f = 0$ by fixing
\begin{equation}\label{eq:varphi}
    \varphi = f + \psi \circ F - \psi.
\end{equation}
Then $\wt F(x,y) = (x + p/q + y, y + \wt g(x,y))$ with 
\begin{equation}\label{eq:tg}
    \wt g \circ (\id + \Psi) = g + (f \circ F - f) + (\psi \circ F^2 - 2 \psi \circ F + \psi).
\end{equation}
Our goal is to choose $\psi$ so that $\wt g(\cdot, 0)$ is closely approximated by its resonant part $\langle g_q \rangle_q^N$, as formalized by the estimate \eqref{eq:rnf concl 2}.

In the first step, we base the construction of $\psi$ on a formal Taylor-series computation, and the coordinate transform \eqref{eq:rnf Phi} is accordingly defined over a neighbourhood of $\{y = 0\}$ which contains the circle $\{y = p/q - \omega\}$.
More precisely, we start in Section~\ref{sec:rnf1 proof taylor} by analyzing the Taylor expansion
\begin{equation}\label{eq:tg taylor}
    \wt g(x,y) = \wt g_1(x) y + \wt g_2(x) y^2 + \CO(y^3)
\end{equation}
where $\wt g$ is as given in \eqref{eq:tg}, and we take 
\begin{equation*}
    \psi(x,y) = \psi_1(x) y + \psi_2(x) y^2
\end{equation*}
with suitable choices of $\psi_1$ and $\psi_2$ so that $\wt g_1$ and $\wt g_2$ are almost $1/q$-periodic.
The resulting coordinate transform $\Phi$ will not be symplectic in general.

In order to evaluate the symplectic action of any orbit found in these new coordinates, we also compute the leading term $\wt s_2$ in the Taylor expansion
\begin{equation}\label{eq:tS Taylor outline}
    \wt S(x,y) = \wt s_2(x) y^2 + \CO(y^3)
\end{equation}
of the pullback $\wt S = S \circ \Phi\inv$ of the generating function $S$.

Then, we restrict everything to a neighbourhood of $\{y = p/q - \omega\}$.
This requires evaluating the Taylor expansion around $y = 0$ at a finite $y \approx p/q - \omega$, and hence we also need a quantitative estimate of the $\CO(y^3)$ remainders in \eqref{eq:tg taylor} and \eqref{eq:tS Taylor outline}.
This estimate will be the main content of Section \ref{ssec:rnf1 q remainder}.
The coordinate transform resulting from Step~1 will be constructed as $\Theta = \Phi\inv \circ \tau$ where $\tau$ is a translation that re-centers the domain at $y = p/q - \omega$.

Roughly speaking, at $y \approx p/q - \omega$ the Taylor remainders in \eqref{eq:tg taylor} and \eqref{eq:tS Taylor outline} have size $\sim |\omega - p/q|^3$, whereas the $q$-resonant part of $\wt g$ restricted to the real torus $\T$ has size $\sim |\omega - p/q|^2 \eu^{-2\pi q}$.
Since $\wt g$ has a rather explicit dependence on $\phi$, this informally indicates that the eigendata of the $(p, q)$-orbit can be controlled by the $q\Z$-Fourier modes of $\phi$ over a range of size at most $\sim |\omega - p/q|^2\eu^{-2\pi q}$.
It can be verified, although we do not do so here, that if we impose the stronger arithmetic condition $|\omega - p/q| \leq \eu^{-2\pi C q}$ for some $C > 1$, then the size $|\omega - p/q|^3$ of the error term will be small enough and the first coordinate change already suffices to deduce the main results. 

If, however, the size of $|\omega - p/q|$ is only polynomially small in $q$, as assumed in part~(1) of the main theorem, then the error contributed by the Taylor remainder outweighs the size of the $q$-resonant part of $\wt g$.
Therefore, we need an iterative scheme in Step~2 to reduce the non-resonant errors further. 

In Section~\ref{sec:rnf part 2 pf}, we use a similar type of coordinate transform given by \eqref{eq:rnf Phi}–\eqref{eq:tg} but with $\omega = p/q$, since we have re-centered the coordinates at $y = p/q - \omega$ in Step~1.
The coordinate transform obtained in Step~2 is near-identity and is denoted $\id + \Sigma$.
The iteration scheme is similar in spirit to the works \cite{BounemouraNiederman} and \cite{Martín_2016}. 
Unlike \cite{BounemouraNiederman}, the coordinates we obtain are not symplectic.

We remark that, in contrast to Step~1, the construction in Step~2 (Section \ref{sec:rnf part 2 pf}) is nonexplicit because it involves $\sim q$ iterations of coordinate transforms.
Hence, while we can make the non-resonant part \eqref{eq:rnf2 iter concl 2} of $\wt g$ exponentially small in $q$, we do not have a better approximation of the resonant part in \eqref{eq:rnf2 iter res compare} — we do not have sufficient information on the contribution to the resonant part from this iteration scheme.

For technical reasons, we cannot skip Step~1 and directly begin with the iteration scheme in Step~2.
This is because, without Step~1, the first iteration of the scheme in Step~2 will introduce an `error' of comparable size to the resonant part of the original map, and it seems difficult to analyze this `error' term explicitly.
On the other hand, Step~1 ensures that the non-resonant part of the map is small enough so that errors introduced by each iteration in Step~2 are always `much smaller' than the resonant part.
This situation appears similar to Lazutkin's proof of existence of invariant curves near the boundary of a convex billiard table: in his proof, one has to transform the original coordinates into higher-order `Lazutkin coordinates' so that the non-integrable part of the billiard map vanishes to sufficiently high order near the boundary before applying the KAM iteration.

We also remark that, while the method in Step~2 yields the weaker arithmetic condition~\eqref{eq:main dioph condition} for the results on Lyapunov exponents, it does not improve the condition~\eqref{eq:main Liouv condition} for the results on symplectic actions.
The key issue is that, whereas Lyapunov exponents can be computed in \emph{any} coordinates, the symplectic action is invariant only under exact symplectic coordinate transforms.
Although the coordinate transform in Step~1 is non-symplectic, it is fairly explicit, and we can approximate the symplectic action using just the leading term \eqref{eq:tS Taylor outline}, \emph{provided} the Taylor remainder is much smaller than the leading term, as ensured by the strong condition~\eqref{eq:main Liouv condition}.
If only weaker conditions of the type~\eqref{eq:main dioph condition} are assumed, then the relevant information on the symplectic actions of $q$-periodic orbits is contained in the first $\sim q$ terms in the Taylor expansion \eqref{eq:tS Taylor outline}.
In other words, it is unclear how to analyze the dependence of the $q$-resonant part of the restricted generating function $S(\cdot, y)$ on $\phi$ when $|y - \omega|$ is only \emph{polynomially} small in $q$.

We now begin the proof of Theorem~\ref{thm:rnf} proper.

\subsection{Step 1: coordinate changes via Taylor expansion}\label{sec:rnf part 1}

\begin{proposition}\label{prop:rnf1}
    Let $\omega \in \R \setminus \Q$ and $\phi: \T_1 \to \C$ be analytic with zero average and $\|\dddot\phi\| \leq 1$.
    Let $F$
    be the standard map associated with $(\omega, \phi)$ in the adapted coordinates.
    Then the following statements hold.
    For $q \in \Z^+$ sufficiently large, if
    \begin{equation}\label{eq:rnf1 arithm}
         \verts{\omega - \frac{p}{q}} \leq \frac{q^{-6}}{6} \qquad \text{for some} \quad p \in \Z,
    \end{equation}
    then there exists an analytic change of coordinates $\Theta$ on the domain $\T_{1 - q\inv} \times \D_{|\omega - p/q|/2}$ which is close to a translation by $p/q - \omega$ in $y$ in the sense that 
    \begin{equation}\label{eq:th12}
        \Theta(x,y) = (x + \theta_1(x,y),\; y + p/q - \omega + \theta_2(x,y))
    \end{equation}
    with 
    \begin{equation}\label{eq:th12 est}
        |\theta_1|_{1-1/q, |\omega - p/q|/2} \lesssim q \verts{\omega - \frac{p}{q}}; \quad |\theta_2|_{1-1/q, |\omega - p/q|/2} < \frac{1}{11} \verts{\omega - \frac{p}{q}},
    \end{equation}
    and the map $\wt F_q := \Theta\inv \circ F \circ \Theta$ in these new coordinates writes
    \begin{equation*}
        \wt F_q: \begin{pmatrix}
            x \\ y 
        \end{pmatrix} \mapsto \begin{pmatrix}
            x + \frac{p}{q} + y \\ y + \wt g_q(x,y)
        \end{pmatrix} .
    \end{equation*}
    Define $f_1: \T_1 \to \C$ as in \eqref{eq:f1} and $\wt S_q := S \circ \Theta$, where  $S$ is the normalized generating function associated with $(\omega, \phi)$ in adapted coordinates.
    Then for $\nu_0 = 64$ and for all $(x,y) \in \T_{1- 1/q} \times \D_{|\omega - p/q|/2}$,
    \begin{equation}\label{eq:rnf1 residue}
        \begin{aligned}
            \Bigg|\wt g_q(x,y) + \frac{1}{2}\Big(\frac{p}{q} - \omega + y\Big)\Big(\frac{p}{q} - \omega - y\Big)\frac{\langle \dot f_1 \rangle_q(x)}{1 + \langle f_1 \rangle_q(x)}\Bigg| \lesssim& q^{\nu_0} \verts{\omega - \frac{p}{q}}^3\\
            \verts{\wt S_q(x,y) - \Big(\frac{p}{q} - \omega + y\Big)^2\frac{1+f_1}{2(1 + \crochet{f_1}_q)^2}(x) } \lesssim& q^{\nu_0} \verts{\omega - \frac{p}{q}}^3.
        \end{aligned}
    \end{equation}
    Moreover, the coordinate transform $\Theta$ depends analytically on $\phi$ and is real-analytic when $\phi$ is real-valued.
    In this case, the map $\wt F_q$ preserves the area form
    \begin{equation}\label{eq:rnf1 area form}
        \Theta^* \du x \du y = \Big(\frac{1}{1 + \langle f_1 \rangle_q} + \wt \eta\Big) \du x \du y
    \end{equation}
    for some real-analytic $\wt \eta$ such that
    \begin{equation}\label{eq:rnf1 area form estimate}
        |\wt \eta|_{1-\frac{1}{q}, \frac12|\omega - \frac{p}{q}|} \lesssim q^{\nu_0} \verts{\omega - \frac{p}{q}}.
    \end{equation}
\end{proposition}

The rest of Section \ref{sec:rnf part 1} aims to prove Proposition~\ref{prop:rnf1}.

\subsubsection{Preliminary Taylor series calculations}\label{sec:rnf1 proof taylor}
Recall that the standard map associated with $(\omega, \phi)$ in adapted coordinates $F$ is given by \eqref{eq:F RIC} - \eqref{eq:fg}.
Consider a change of coordinates $\Phi = \id + \Psi$ as described in \eqref{eq:rnf Phi} - \eqref{eq:tg}.
We shall choose $\psi$ to be of the form $\psi(x,y) = \psi_1(x)y + \psi_2(x) y^2$ for some $\psi_1$ and $\psi_2$ such that in a neighbourhood of $\{y = p/q - \omega\}$ the function $\wt g$ in \eqref{eq:tg} is almost $1/q$-periodic in $x$.

\paragraph{Taylor series of $F$ and the generating function.}
In this paragraph we obtain the Taylor expansion at $y = 0$ of $f$ and $g$ as well as the generating function $S$ up to order $\CO(y^2)$. 
Let us begin by fixing some notations.
Denote 
\begin{equation*}
    \begin{aligned}
        f(x,y) = f_1(x) y + f_2(x) y^2 + \CO(y^3) \quad \text{and} \quad 
        g(x,y) = g_1(x) y + g_2(x) y^2 + \CO(y^3).
    \end{aligned}
\end{equation*}
We adopt similar notations for the Taylor expansions of $\psi$, $\varphi$ and $\wt g$.
Then, a straightforward computation using the Taylor series of \eqref{eq:fg} yields 
\begin{equation}\label{eq:fg12}
    1 + f_1 = \frac{1}{(1 + \dot \phi)(1 + \dot \phi^+)}; \quad f_2 = - \frac{\ddot\phi^+}{2(1 + \dot \phi)^2(1 + \dot\phi^+)^3}; \quad g_1 = 0; \quad g_2 = - \frac{1}{2} \dot f_1.
\end{equation}
We remark that the last equality $g_2 = -\dot f_1 /2$ can be computed from the area-preservation property $F^*(\du x \du y) = \du x \du y$ and $g_1 = 0$.
Next, we consider the normalized generating function $S$ associated with $(\omega, \phi)$ in adapted coordinates (see Section~\ref{sec:RIC coord}).
Integrating the Taylor expansion of the second equation in \eqref{eq:RIC gen func-fg} along $y$ and using the Taylor remainder formula and the normalization condition $S(x, 0) = 0$, we get
\begin{equation}\label{eq:S Taylor}
    S(x, y) = \frac{1 + f_1(x)}{2} y^2  + \CO(y^3).
\end{equation}

\paragraph{Taylor series after coordinate changes.}
To express the Taylor expansion of \eqref{eq:varphi} and \eqref{eq:tg} in $y$, let us define a finte-difference operator $\CD_{\omega}$ by 
\begin{equation}\label{eq:Domega}
    \CD_{\omega} u := u \circ R_{\omega} - u \quad \text{for any} \quad u: \T_a \to \C.
\end{equation}
Also recall the notation $u^\pm: = u \circ R_{\pm \omega}$.
Then, Taylor expanding the right hand sides of \eqref{eq:varphi} and of the second line of \eqref{eq:tf tg} 
and substituting $g_1 = 0$, $g_2 = -\dot f_1 /2$,
\begin{equation}\label{eq:RHS Taylor}
    \begin{aligned}
        \varphi =&  (f_1 + \CD_\omega \psi_1) y + \Big(f_2 + \CD_\omega \psi_2 - \frac{1}{2} \dot f_1 \psi_1^+ + (1 + f_1) \dot \psi_1^+\Big)y^2 + \CO(y^3).\\
        g + \varphi \circ F - \varphi =& (\CD_\omega \varphi_1) y + \Big(\CD_\omega \varphi_2 -\frac{1}{2} (1 + \varphi_1^+)\dot f_1 + (1 + f_1) \dot \varphi_1^+\Big)y^2 + \CO(y^3),
    \end{aligned}
\end{equation}
On the other hand, Taylor expanding $\wt g \circ (\id + \Psi)$ at $y = 0$ yields 
\begin{equation}\label{eq:LHS Taylor}
    \wt g \circ (\id + \Psi) = (1 + \varphi_1) \wt g_1 y + \big((1 + \varphi_1)^2 \wt g_2 + \varphi_2 \wt g_1 + (1 + \varphi_1) \psi_1 \dot{\wt g}_1\big) y^2 + \CO(y^3).
\end{equation}
From the first line of \eqref{eq:RHS Taylor},
\begin{equation}\label{eq:varphi 12}
    \varphi_1 = f_1 + \CD_\omega \psi_1, \quad \varphi_2 = f_2 + \CD_\omega \psi_2 -\frac{1}{2} \dot f_1\psi_1^+ + (1 + f_1) \dot \psi_1^+.
\end{equation}
Equating the second line of \eqref{eq:LHS Taylor} with \eqref{eq:RHS Taylor},
\begin{equation*}
    \begin{aligned}
        (1 + \varphi_1) \wt g_1 =& \CD_\omega \varphi_1, \\
        (1 + \varphi_1)^2 \wt g_2 + \varphi_2 \wt g_1 + (1 + \varphi_1) \psi_1 \dot{\wt g}_1 =&  \CD_\omega \varphi_2 -\frac{1}{2} (1+\varphi_1^+) \dot f_1  + (1 + f_1) \dot \varphi_1^+.
    \end{aligned}
\end{equation*}
We may differentiate the first line in $x$ and multiply throughout by $\psi_1$ to obtain $(1 + \varphi_1) \psi_1 \dot {\wt g}_1 + \psi_1 \dot\varphi_1 \wt g_1 = \psi_1 \CD_\omega \dot\varphi_1$, which can be used to eliminate the $(1 + \varphi_1) \psi_1 \dot{\wt g}_1$ term on the second line above and obtain 
\begin{equation}\label{eq:tg 12}
    \begin{aligned}
    (1 + \varphi_1) \wt g_1 =&  \CD_\omega \varphi_1, \\
    (1 + \varphi_1)^2 \wt g_2 + (\varphi_2 - \psi_1 \dot\varphi_1) \wt g_1 + \psi_1 \CD_\omega \dot\varphi_1 =&  \CD_\omega \varphi_2 - \frac{1}{2} (1+\varphi_1^+) \dot f_1  + (1 + f_1) \dot \varphi_1^+.
    \end{aligned}
\end{equation}
The formulae \eqref{eq:varphi 12} and \eqref{eq:tg 12} together give the dependence of the leading Taylor coefficients $\wt g_1$ and $\wt g_2$ of $\wt g$ in terms of $\psi_1$ and $\psi_2$.
Our next task is to choose $\psi_1$ and $\psi_2$ appropriately so that $\wt g_1$ and $\wt g_2$ are almost $1/q$-periodic on $\T$.

\paragraph{The choice of $\psi_1$, $\psi_2$.}
From now on, let $q > 1$ and $p \in \Z$ be fixed.
Analogous to \eqref{eq:Domega}, we define a finte-difference operator $\CD_{p/q}$ by 
\begin{equation*}
    \CD_{p/q} u := u \circ R_{p/q} - u \quad \text{for any} \quad u: \T_a \to \C.
\end{equation*}
Recall that we denote by $\langle u \rangle_q$ and $\curly{u}_q$ to be the $1/q$-periodization of $u$ and the complement part respectively (see Section \ref{sec:notation}).
Given $v: \T_1 \to \C$, there is a \textit{unique} $u: \T_1 \to \C$, denoted $\CD_{p/q}\inv v$, such that $\CD_{p/q} u = v$ and $\langle u \rangle_q = 0$. The function $u$ can be defined via its Fourier series as 
\begin{equation*}
    \hat u_k = \begin{cases}
        \frac{\hat v_k}{\eu^{2\pi \iu k \frac{p}{q}} - 1} & \text{if $q \nmid k$}\\
        0 & \text{otherwise.}
    \end{cases} 
\end{equation*}
Recall that $\wt g_1$, $\wt g_2$ and $\wt s_2$ denote the lead terms of the Taylor expansions of $\wt g$ and $\wt S$ as in \eqref{eq:tg taylor} and \eqref{eq:tS Taylor outline}.
In the sequel, we shall choose $\psi_1$ and $\psi_2$ so as to make $\{\wt g_1\}$ and $\{\wt g_2\}_q$ as small as possible assuming $p/q$ is sufficiently close to $\omega$.
Informally, we will obtain 
\begin{equation}\label{eq:rnf1 tg12 informal}
    \wt g_1 \approx \Big(\omega - \frac{p} {q}\Big) \frac{\langle \dot f_1 \rangle_q}{1 + \crochet{f_1}_q}, \quad \wt g_2 \approx \frac{1}{2} \frac{\langle \dot f_1 \rangle_q}{1 + \langle f_1 \rangle_q}
\end{equation}
and 
\begin{equation}\label{eq:rnf1 tS informal}
    \wt s_2 \approx \frac{1+f_1}{2(1 + \langle f_1 \rangle_q)^2}.
\end{equation}
In this section, we obtain the terms on the right hand sides of `$\approx$' above.  
The rigorous estimates of the differences between the two sides will be left to the next section.
The guiding heuristics of the choice of $\psi$ that follows is that $|\omega - p/q|$ shall be regarded as a `small' quantity so that $\CD_\omega \crochet{u}_q$ and $(\CD_{\omega} - \CD_{p/q})u$ generally have `smaller' orders of magnitude than $\crochet{u}_q$ and $u$ respectively.
In fact, to achieve \eqref{eq:rnf1 tg12 informal} we take $\psi(x,y) = \psi_1(x) y + \psi_2(x) y^2$ with
\begin{equation}\label{eq:psi12}
\begin{aligned}
    &\psi_1 = (\CD_{\omega} - \CD_{p/q})\CD_{p/q}^{-2} \curly{f_1}_q - \CD_{p/q}\inv \{f_1\}_q; \\
    &\psi_2 = - \CD_{p/q}^{-2} \curly{\CD_\omega \Big( f_2  -\frac{1}{2} \dot f_1 \psi_1^+ + (1 + f_1) \dot \psi_1^+ \Big) -\frac{1}{2} (1+\varphi_1^+)\dot f_1  + (1 + f_1) \dot \varphi_1^+}_q.
\end{aligned}
\end{equation}
By the choice of $\psi_1$ and the first line of \eqref{eq:RHS Taylor},
\begin{equation}\label{eq:varphi1 explicit}
\begin{aligned}
        \varphi_1 =& f_1 + (\CD_\omega - \CD_{p/q})\psi_1 + \CD_{p/q}\psi_1  \\
    =& f_1 + (\CD_\omega - \CD_{p/q})\psi_1 + (\CD_{\omega} - \CD_{p/q})\CD_{p/q}^{-1} \curly{f_1}_q - \curly{f_1}_q\\
    =& \crochet{f_1}_q + (\CD_\omega - \CD_{p/q})(\psi_1 + \CD_{p/q}^{-1} \curly{f_1}_q)\\
    =& \crochet{f_1}_q + (\CD_\omega - \CD_{p/q})^2\CD_{p/q}^{-2} \curly{f_1}_q.
\end{aligned}
\end{equation}
Thus $\varphi_1 = \crochet{ f_1}_q + \{\varphi_1\}_q$ and we may consider $\{\varphi_1\}_q$ and $\CD_\omega \varphi_1$ as small errors. 
It follows from the first line of \eqref{eq:tg 12} that 
\begin{equation}\label{eq:tg1 explicit}
\begin{aligned}
        \wt g_1 =& \frac{\CD_\omega \crochet{f_1}_q}{1 + \crochet{f_1}_q} + \frac{\CD_\omega \curly{\varphi_1}_q - \curly{\varphi_1}_q \wt g_1}{1 + \crochet{f_1}_q}\\
        =& \Big(\omega - \frac{p} {q}\Big) \frac{\langle \dot f_1 \rangle_q}{1 + \crochet{f_1}_q} 
        + \frac{\CD_\omega \crochet{f_1}_q - (\omega - \frac{p} {q}) \langle \dot f_1 \rangle_q}{1 + \crochet{f_1}_q} 
        + \frac{\{\varphi_1^+\}_q - (1 + \wt g_1) \curly{\varphi_1}_q }{1 + \crochet{f_1}_q}.
\end{aligned}
\end{equation}
On the other hand, substituting \eqref{eq:varphi 12} into the right hand side of the second line of \eqref{eq:tg 12} and using the choice of $\psi_2$,
\begin{equation}\label{eq:tg2 rhs}
    \begin{aligned}
        &\CD_\omega^2 \psi_2 + \CD_\omega \Big(f_2 -\frac{1}{2} \dot f_1 \psi_1^+ + (1 + f_1) \dot \psi_1^+\Big) -\frac{1}{2} (1+\varphi_1^+)\dot f_1  + (1 + f_1) \dot \varphi_1^+ \\ 
        =&(\CD_\omega^2 - \CD_{p/q}^2) \psi_2 + \CD_\omega \crochet{f_2 -\frac{1}{2} \dot f_1 \psi_1^+ + (1 + f_1) \dot \psi_1^+}_q + \crochet{(1 + f_1) \dot \varphi_1^+ -\frac{1}{2} (1+\varphi_1^+)\dot f_1}_q \\ 
        =&(\CD_\omega^2 - \CD_{p/q}^2) \psi_2 + \CD_\omega \crochet{f_2 -\frac{1}{2} \dot f_1 \psi_1^+ + (1 + f_1) \dot \psi_1^+}_q + \crochet{(1 + f_1) \dot \varphi_1 -\frac{1}{2} (1+\varphi_1)\dot f_1}_q \\
        &  + \crochet{ (1 + f_1) \CD_\omega \dot \varphi_1 - \frac{1}{2} \dot f_1 \CD_\omega \varphi_1}_q\\ 
        =&(\CD_\omega^2 - \CD_{p/q}^2) \psi_2 + \CD_\omega \crochet{f_2 -\frac{1}{2} \dot f_1 \psi_1^+ + (1 + f_1) \dot \psi_1^+}_q + \crochet{(1 + f_1) \langle \dot f_1 \rangle_q -\frac{1}{2} (1+\crochet{f_1}_q)\dot f_1}_q \\
        & + \crochet{(1 + f_1) \{\dot \varphi_1\}_q -\frac{1}{2} \{\varphi_1\}_q \dot f_1 }_q
        + \crochet{ (1 + f_1) \CD_\omega \dot \varphi_1 - \frac{1}{2} \dot f_1 \CD_\omega \varphi_1}_q\\
        =& \frac{1}{2}(1 + \crochet{f_1}_q) \langle \dot f_1 \rangle_q + (\CD_\omega^2 - \CD_{p/q}^2) \psi_2 + \CD_\omega \crochet{f_2 -\frac{1}{2} \dot f_1 \psi_1^+ + (1 + f_1) \dot \psi_1^+}_q  \\
        & + \crochet{(1 + f_1) \{\dot \varphi_1\}_q -\frac{1}{2} \{\varphi_1\}_q \dot f_1 }_q
        + \crochet{ (1 + f_1) \CD_\omega \dot \varphi_1 - \frac{1}{2} \dot f_1 \CD_\omega \varphi_1}_q.
    \end{aligned}
\end{equation}
In the last line, we used the simple observation that $\langle (1 + f_1) \langle \dot f_1\rangle_q \rangle_q = (1 + \langle f_1 \rangle_q) \langle \dot f_1 \rangle_q$.
Using $(1 + \varphi_1)^2 = (1 + \crochet{f_1}_q)^2 + (2 + \crochet{f_1}_q + \varphi_1)\{\varphi_1\} $,  the left hand side of the second line of \eqref{eq:tg 12} writes 
\begin{equation}\label{eq:tg2 lhs}
\begin{aligned}
     (1 + \crochet{f_1}_q)^2  \wt g_2 + (2 + \crochet{f_1}_q + \varphi_1)\{\varphi_1\}  \wt g_2  + (\varphi_2 - \psi_1 \dot\varphi_1) \wt g_1 + \psi_1 \CD_\omega \dot\varphi_1.
\end{aligned}
\end{equation}
Equating \eqref{eq:tg2 lhs} with \eqref{eq:tg2 rhs}, we obtain 
\begin{equation}\label{eq:tg2 explicit}
    \begin{aligned}
        \wt g_2 = \frac{1}{2} \frac{\langle \dot f_1 \rangle_q}{1 + \langle f_1 \rangle_q} 
        &+ \frac{1}{(1 + \crochet{f_1}_q)^2}\bigg[(\CD_\omega^2 - \CD_{p/q}^2) \psi_2 + \CD_\omega \crochet{f_2 -\frac{1}{2} \dot f_1 \psi_1^+ + (1 + f_1) \dot \psi_1^+}_q  \\
        & \qquad  + \crochet{(1 + f_1) \{\dot \varphi_1\}_q -\frac{1}{2} \{\varphi_1\}_q \dot f_1 }_q
        + \crochet{ (1 + f_1) \CD_\omega \dot \varphi_1 - \frac{1}{2} \dot f_1 \CD_\omega \varphi_1}_q \\
        & \qquad -(2 + \crochet{f_1}_q + \varphi_1)\{\varphi_1\}  \wt g_2  - (\varphi_2 - \psi_1 \dot\varphi_1) \wt g_1 - \psi_1 \CD_\omega \dot\varphi_1\bigg].
    \end{aligned}
\end{equation}

We now compute an approximation of the pullback $\wt S = S \circ \Phi\inv$ of the generating function to the new coordinates using \eqref{eq:S Taylor}:
\begin{equation}\label{eq:tS Taylor}
        \wt S(x,y) 
        = \frac{1+f_1}{2(1 + \varphi_1)^2} y^2 + \CO(y^3).
\end{equation}
Using $(1 + \varphi_1)^2 = (1 + \crochet{f_1}_q)^2 + (2 + \crochet{f_1}_q + \varphi_1)\{\varphi_1\} $, we may write 
\begin{equation}\label{eq:tS2 explicit}
    \wt s_2 = \frac{1+f_1}{2(1 + \langle f_1 \rangle_q)^2} - \frac{(1+f_1)(2  +\langle f_1 \rangle_q + \varphi_1) \{\varphi_1\}_q}{2(1 + \langle f_1 \rangle_q)^2 (1 + \varphi_1)^2}.
\end{equation}

\subsubsection{Quantitative estimates: the leading terms of $\wt g$.}\label{ssec:rnf1 q leading}

The goal of this section is to give a precise justification of the approximations  \eqref{eq:rnf1 tg12 informal} and \eqref{eq:rnf1 tS informal} by showing that all but the first terms on the right hand sides of \eqref{eq:tg1 explicit}, \eqref{eq:tg2 explicit} and \eqref{eq:tS2 explicit} are relatively small.
More precisely, we prove 
\begin{lemma}\label{lem:tg12 ts2 est}
    Suppose 
    \begin{equation}\label{eq:rnf1 cons standing assump}
        \|\dddot\phi\|  \leq 1 \quad \text{and} \quad 
        \verts{\omega - \frac{p}{q}} \leq \frac{1}{q^3}.
    \end{equation}
    Then
    \begin{equation}\label{eq:tg12 est}
        \norm{\wt g_1 - \Big(\omega - \frac{p} {q}\Big) \frac{\langle \dot f_1 \rangle_q}{1 + \crochet{f_1}_q}} \lesssim q^2 \verts{\omega - \frac{p}{q}}^2, \qquad \norm{\wt g_2 - \frac{1}{2} \frac{\langle \dot f_1 \rangle_q}{1 + \langle f_1 \rangle_q} } \lesssim q^4 \verts{\omega - \frac{p}{q}}
    \end{equation}
    and 
    \begin{equation}\label{eq:ts2 est}
        \norm{\wt s_2 - \frac{1+f_1}{2(1 + \langle f_1 \rangle_q)^2}} \lesssim \verts{\omega - \frac{p}{q}}.
    \end{equation}
\end{lemma}
Before proceeding to the proof of Lemma \ref{lem:tg12 ts2 est}, 
let us note some elementary facts.
Since 
\begin{equation*}
    k \nmid q \quad \implies \quad |\eu^{2\pi \iu k p/q} - 1| \geq |\eu^{2\pi \iu /q} - 1| = 2\sin\frac{\pi}{q} \geq \frac{4}{q},
\end{equation*} 
we have, for $u: \T_1 \to \C$ analytic, that
\begin{equation}\label{eq:CD inv}
    \|\CD_{p/q}\inv \{u\}_q\| = \sum_{q \nmid k} \verts{\frac{\hat u_k}{\eu^{2\pi \iu k p/q} - 1}} \eu^{2\pi |k|} 
    \leq \frac{q}{4} \|u\|
\end{equation}
where $\CD_{p/q}\inv \{u\}_q$ is uniquely defined as the function $v$ such that $\langle v \rangle_q = 0$ and $\CD_{p/q} v = u$.
On the other hand, since $(\CD_{\omega} - \CD_{p/q}) u = u(\cdot + \omega) - u(\cdot + p/q) = \CD_{\omega - p/q} (u \circ R_{p/q})$, where $R_{p/q}: x \mapsto x + p/q$ is a rotation on $\T_1$, by Lemma \ref{lem:fourier MVT} with $\delta = \omega - p/q$,
\begin{equation}\label{eq:Fourier MVT}
    \|(\CD_\omega - \CD_{p/q}) u\| =  \|\CD_{\omega - p/q} (u \circ R_{p/q})\| \leq \verts{\omega - \frac{p}{q}} \|\dot u\|.
\end{equation}
Finally, we also remark that, after integrating the Fourier series of $\dddot\phi$ and using the definition of $\|\; \cdot\; \|$-norm, the condition $\|\dddot\phi\| \leq 1$  implies the following convenient estimates.
\begin{equation}\label{eq:dphi bnd}
    \|\ddot\phi\| \leq \frac{1}{2\pi} \quad \text{and} \quad \|\dot\phi\| \leq \frac{1}{(2 \pi)^2}.
\end{equation}
This observation has already been used previously in the proof of Corollary~\ref{cor:only 2}.

Let us begin to prove Lemma \ref{lem:tg12 ts2 est}.

By \eqref{eq:dphi bnd} and the explicit expression \eqref{eq:fg12}, we first have 
\begin{equation}\label{eq:f1 norm}
    \|f_1\| \leq \frac{(2 + \|\dot\phi\|) \|\dot\phi\|}{(1 - \|\dot\phi\|)^2}  < \frac{1}{18}.
\end{equation}
In the above, we used the inequality $\|(1 + f_1)\inv\| \leq (1 - \|f_1\|)\inv$, which follows from the fact that the space of analytic functions $\T_1 \to \C$ with finite $\|\; \cdot \; \|$-norm is a Banach algebra (Lemma \ref{lem:banach-alg}).

Furthermore, it is straightforward to deduce from $\|\dot\phi\| \leq (2\pi)^{-2}$ and the explicit expressions \eqref{eq:fg12} that
\begin{equation}\label{eq:fg 12 dot norm}
    \|\dot f_1\|, \,  \|\ddot f_1\|,\,  \|f_2\|,\,  \|\dot f_2\|,\,  \|g_2\|,\,  \|\dot g_2\| \,  \lesssim \,  1.
\end{equation}

Recall that by \eqref{eq:varphi1 explicit} we have $\varphi_1 = \crochet{f_1}_q + (\CD_\omega - \CD_{p/q})^2\CD_{p/q}^{-2} \curly{f_1}_q$.
Hence, by \eqref{eq:CD inv} - \eqref{eq:fg 12 dot norm} and assumption \ref{eq:rnf1 cons standing assump},
\begin{equation}\label{eq:varphi1 norm}
    \begin{aligned}
        \|\varphi_1\| 
        \leq& \|\crochet{f_1}_q\| + \|(\CD_\omega - \CD_{p/q})^2\CD_{p/q}^{-2} \curly{f_1}_q\|\\
        \leq& \|f_1\| + \Big(\frac{q}{4}\Big)^2 \verts{\omega - \frac{p}{q}}\|(\CD_\omega - \CD_{p/q})\dot f_1\|
        <\frac{1}{18}
    \end{aligned}
\end{equation}
for $q$ sufficiently large.
Similarly, using \eqref{eq:fg 12 dot norm}
\begin{equation}\label{eq:dvarphi1 norms}
    \begin{aligned}
        \|\dot \varphi_1\| =& \|\langle \dot f_1\rangle_q + (\CD_\omega - \CD_{p/q})^2\CD_{p/q}^{-2} \{ \dot f_1\}_q\|  \lesssim \|\dot f_1\| + q^2 \verts{\omega - \frac{p}{q}} \|\ddot f_1\| \lesssim 1;\\
        \|\ddot \varphi_1\| =& \|\langle \ddot f_1\rangle_q + (\CD_\omega - \CD_{p/q})^2\CD_{p/q}^{-2} \{ \ddot f_1\}_q\|  \lesssim \|\ddot f_1\| + q^2 \|\ddot f_1\| \lesssim q^2.
    \end{aligned}
    \end{equation}

\begin{remark}
    We give more precise bounds in \eqref{eq:f1 norm} and \eqref{eq:varphi1 norm} in stead of just an implicit universal constant $\lesssim 1$ as in \eqref{eq:fg 12 dot norm} and \eqref{eq:dvarphi1 norms}.
    The reason is that we will need to provide lower bounds on the denominators is \eqref{eq:tg1 explicit}, \eqref{eq:tg2 explicit} and \eqref{eq:tS2 explicit}.
\end{remark}

Since $\{\varphi_1\}_q = (\CD_\omega - \CD_{p/q})^2 \CD_{p/q}^{-2} \{f_1\}_q$ and $\crochet{\varphi_1}_q = \crochet{f_1}_q$, by periodicity of $\crochet{f_1}$ and \eqref{eq:CD inv}, \eqref{eq:Fourier MVT}, \eqref{eq:fg 12 dot norm},
\begin{equation}\label{eq:varphi1 norm aux}
    \begin{aligned}
        \|\{\varphi_1\}_q\| \lesssim& q^2 \verts{\omega - \frac{p}{q}}^2 \|\ddot f_1\| \lesssim q^2 \verts{\omega - \frac{p}{q}}^2 \qquad \text{(applying \eqref{eq:Fourier MVT} twice)};\\
        \|\{\dot \varphi_1\}_q\| \lesssim& q^2 \verts{\omega - \frac{p}{q}} \|\ddot f_1\| \lesssim q^2 \verts{\omega - \frac{p}{q}} \qquad \text{(applying \eqref{eq:Fourier MVT} once)};\\
        \|\CD_{\omega} \varphi_1\| =& \|(\CD_{\omega} - \CD_{p/q}) \crochet{f_1}_q\| + \|\CD_{\omega} \{\varphi_1\}_q\|
        \lesssim \verts{\omega - \frac{p}{q}}\|\dot f_1\| + \|\{\varphi_1\}_q\| \lesssim \verts{\omega - \frac{p}{q}};\\
        \|\CD_{\omega} \dot\varphi_1\| =& \|(\CD_{\omega} - \CD_{p/q}) \langle \dot f_1 \rangle_q\| + \|\CD_{\omega} \{\dot \varphi_1\}_q\|
        \lesssim \verts{\omega - \frac{p}{q}}\|\ddot f_1\| + \|\{\dot\varphi_1\}_q\| \lesssim q^2 \verts{\omega - \frac{p}{q}};\\
    \end{aligned}
\end{equation}
Using the first line of \eqref{eq:tg 12}, we have $\wt g_1 = (1 + \varphi_1)\inv \CD_\omega \varphi_1$.
By \eqref{eq:varphi1 norm} and \eqref{eq:varphi1 norm aux},
\begin{equation*}
    \|\wt g_1\| \lesssim \|\CD_\omega \varphi_1\| \lesssim \verts{\omega - \frac{p}{q}}.
\end{equation*}
We may now estimate the last two terms of \eqref{eq:tg1 explicit}.
Using Lemma~\ref{lem:fourier MVT} with $\delta = \omega - p/q$ and the preceding bounds \eqref{eq:f1 norm}, \eqref{eq:fg 12 dot norm} on $f_1$ as well as the first line of \eqref{eq:varphi1 norm aux}, we have
\begin{equation*}
\begin{aligned}
    \norm{\wt g_1 - \Big(\omega - \frac{p} {q}\Big) \frac{\langle \dot f_1 \rangle_q}{1 + \crochet{f_1}_q}} 
    \leq& \norm{ \frac{\CD_\omega \crochet{f_1}_q - (\omega - \frac{p} {q}) \langle \dot f_1 \rangle_q}{1 + \crochet{f_1}_q} } 
    + \norm{\frac{\{\varphi_1^+\}_q - (1 + \wt g_1) \curly{\varphi_1}_q }{1 + \crochet{f_1}_q}}\\
    \lesssim& \verts{\omega - \frac{p}{q}}^2 \|\ddot f_1\| + q^2 \verts{\omega - \frac{p}{q}}^2 
    \lesssim q^2 \verts{\omega - \frac{p}{q}}^2.
\end{aligned}
\end{equation*}
We have established the first inequality of \eqref{eq:tg12 est}.
To approximate $\wt g_2$, it is necessary to bound the second term of \eqref{eq:tg2 explicit}, which in turn requires bounds on $\psi_1$ and $\psi_2$ as well as $\varphi_2$.
To this end, let us first use definition \eqref{eq:psi12} of $\psi_1$, inequalities \eqref{eq:CD inv}, \eqref{eq:Fourier MVT}, \eqref{eq:fg 12 dot norm} and assumption \eqref{eq:rnf1 cons standing assump} to obtain
\begin{equation}\label{eq:psi1 norms}
        \|\psi_1\| \lesssim q^2 \verts{\omega - \frac{p}{q}} \|\dot f_1\| + q \|f_1\| \lesssim q;\qquad 
        \|\dot \psi_1\| \lesssim q^2 \verts{\omega - \frac{p}{q}} \|\ddot f_1\| + q \|\dot f_1\| \lesssim q
\end{equation}
and 
\begin{equation}\label{eq:ddpsi1 norm}
    \|\ddot\psi_1\| \lesssim q^2 \|\ddot f_1\| + q \|\ddot f_1\| \lesssim q^2.
\end{equation}
Similarly, by definition \eqref{eq:psi12} and \eqref{eq:dvarphi1 norms}, \eqref{eq:psi1 norms},
\begin{equation*}
    \|\psi_2\| 
    \lesssim q^2 \Big( \|f_2\|  + \|\dot f_1 \psi_1^+\| + \|(1 + f_1) \dot \psi_1^+\| + \|(1+\varphi_1^+)\dot f_1\|  + \|(1 + f_1) \dot \varphi_1^+\|\Big)
    \lesssim q^3.
\end{equation*}
By a direct computation, we find that $\dot \psi_2$ is given by
\begin{equation*}
    - \CD_{p/q}^{-2} \curly{\CD_\omega \Big( \dot f_2  -\frac{1}{2} \ddot f_1 \psi_1^+ + (1 + f_1) \ddot \psi_1^+ + \frac{1}{2} \dot f_1 \dot \psi_1^+\Big) -\frac{1}{2} (1+\varphi_1^+)\ddot f_1  + (1 + f_1) \ddot \varphi_1^+ + \frac{1}{2} \dot\varphi_1^+\dot f_1}_q.
\end{equation*}
Thus, by \eqref{eq:dvarphi1 norms} and \eqref{eq:ddpsi1 norm},
\begin{equation*}
    \begin{aligned}
        \|\dot \psi_2\| 
        \lesssim& q^2 \Big( \| \dot f_2\| + \|\ddot f_1 \psi_1^+\| + \|(1 + f_1) \ddot \psi_1^+\|+ \|\dot f_1 \dot \psi_1^+\|+ \|(1 + \varphi_1^+)\ddot f_1\|+\|(1 + f_1)\ddot \varphi_1^+\|+ \|\dot\varphi_1^+\dot f_1\|\Big)\\
        \lesssim& q^4.
    \end{aligned}
\end{equation*}
Then, from the definition \eqref{eq:varphi 12} of $\varphi_2$, a similar estimate gives
\begin{equation*}
    \|\varphi_2\| = \bigg\| f_2 + \CD_\omega \psi_2 -\frac{1}{2} \dot f_1\psi_1^+ + (1 + f_1) \dot \psi_1^+\bigg\| \lesssim q^3.
\end{equation*}
It now follows from \eqref{eq:varphi1 norm} the second line of \eqref{eq:tg 12} that
\begin{equation*}
    \|\wt g_2\| \lesssim \bigg\|\CD_\omega \varphi_2 - \frac{1}{2} (1+\varphi_1^+) \dot f_1  + (1 + f_1) \dot \varphi_1^+ - \big((\varphi_2 - \psi_1 \dot\varphi_1) \wt g_1 + \psi_1 \CD_\omega \dot\varphi_1\big)\bigg\|  \lesssim q^3.
\end{equation*}
To recapitulate, we have obtained 
\begin{equation}\label{eq:leading norms recap}
    \begin{gathered}
        \|\psi_1\|, \|\dot \psi_1\| \lesssim q; 
        \quad \|\ddot\psi_1\| \lesssim q^2;
        \quad \|\psi_2\|\lesssim q^3;  
        \quad \|\dot\psi_2\|\lesssim q^4; 
        \quad \|\varphi_1\| < \frac{1}{18}; \quad  \|\dot\varphi_1\| \lesssim 1; 
        \\
        \|\ddot\varphi_1\| \lesssim q^2; \quad \|\varphi_2\| \lesssim q^3; \quad \|\{\varphi_1\}_q\| \lesssim q^2 \verts{\omega - \frac{p}{q}}^2; 
        \quad \|\{\dot \varphi_1\}_q\| \lesssim q^2 \verts{\omega - \frac{p}{q}};\\
         \|\CD_\omega \varphi_1\| \lesssim \verts{\omega - \frac{p}{q}}; 
        \quad \|\CD_\omega \dot \varphi_1\| \lesssim q^2 \verts{\omega - \frac{p}{q}}; 
        \quad \|\wt g_1\| \lesssim \verts{\omega - \frac{p}{q}}; 
        \quad \|\wt g_2\| \lesssim q^3.
    \end{gathered}
\end{equation}
Now we are ready to approximate $\wt g_2$ by estimating the last term of \eqref{eq:tg2 explicit}.
Since $\|(1 + \langle f_1 \rangle_q)^{-2}\| \lesssim 1$, we may focus on the terms within the square bracket of \eqref{eq:tg2 explicit}.
Since $(\CD_\omega^2 - \CD_{p/q}^2) \psi_2 = (\CD_\omega - \CD_{p/q})(\CD_\omega + \CD_{p/q}) \psi_2$, we may use \eqref{eq:Fourier MVT} together with the inequalities summarized in \eqref{eq:leading norms recap} to bound the first two terms as follows.
\begin{equation*}
\begin{aligned}
        &\norm{(\CD_\omega^2 - \CD_{p/q}^2) \psi_2 + \CD_\omega \crochet{f_2 -\frac{1}{2} \dot f_1 \psi_1^+ + (1 + f_1) \dot \psi_1^+}_q} \\
    \lesssim& \verts{\omega - \frac{p}{q}}\bigg( \|\dot \psi_2\| + \norm{\crochet{\dot f_2  -\frac{1}{2} \ddot f_1 \psi_1^+ + (1 + f_1) \ddot \psi_1^+ + \frac{1}{2} \dot f_1 \dot \psi_1^+}_q}\bigg)
    \lesssim q^4 \verts{\omega - \frac{p}{q}}.
\end{aligned}
\end{equation*}
The rest of the terms can be controlled by a routine application of \eqref{eq:leading norms recap}:
\begin{equation*}
    \begin{aligned}
        &\bigg\|\crochet{(1 + f_1) \{\dot \varphi_1\}_q -\frac{1}{2} \{\varphi_1\}_q \dot f_1 }_q\bigg\| \lesssim  q^2 \verts{\omega - \frac{p}{q}}\\
        &\bigg\|\crochet{ (1 + f_1) \CD_\omega \dot \varphi_1 - \frac{1}{2} \dot f_1 \CD_\omega \varphi_1}_q \bigg\| \lesssim q^2 \verts{\omega - \frac{p}{q}} \\ 
        &\|(2 + \crochet{f_1}_q + \varphi_1)\{\varphi_1\}  \wt g_2  + (\varphi_2 - \psi_1 \dot\varphi_1) \wt g_1 + \psi_1 \CD_\omega \dot\varphi_1\| \lesssim q^3 \verts{\omega - \frac{p}{q}}.
    \end{aligned}
\end{equation*}
We have shown
\begin{equation*}
    \begin{aligned}
        \norm{\wt g_2 - \frac{1}{2} \frac{\langle \dot f_1 \rangle_q}{1 + \langle f_1 \rangle_q} } \lesssim q^4 \verts{\omega - \frac{p}{q}}
    \end{aligned}
\end{equation*}
and completed the proof of \eqref{eq:tg12 est}.
Finally, we turn to the estimate \eqref{eq:ts2 est} of  $\wt s_2$. 
By the second line of \eqref{eq:tS Taylor} and the bounds on $\|f_1\|$, $\|\varphi_1\|$ and $\|\{\varphi_1\}_q\|$ collected in \eqref{eq:leading norms recap} as well as the assumption \eqref{eq:rnf1 cons standing assump},
\begin{equation*}
    \begin{aligned}
        \norm{\wt s_2 - \frac{1+f_1}{2(1 + \langle f_1 \rangle_q)^2}} =& \norm{\frac{(1+f_1)(2  +\langle f_1 \rangle_q + \varphi_1) \{\varphi_1\}_q}{2(1 + \langle f_1 \rangle_q)^2 (1 + \varphi_1)^2}}\\
        \lesssim& \|\{\varphi_1\}_q\| \lesssim q^2 \verts{\omega - \frac{p}{q}}^2 \lesssim \verts{\omega - \frac{p}{q}}.
    \end{aligned}
\end{equation*}
The proof of Lemma \ref{lem:tg12 ts2 est} is finished.

\subsubsection{Quantitative estimates: the Taylor remainder}\label{ssec:rnf1 q remainder}

Having obtained approximations of the leading Taylor coefficients $\wt g_1$ and $\wt g_2$ of $\wt g$ and $\wt s_2$ of the pullback generating function $\wt S$, we now need to estimate the corresponding Taylor remainders of $\wt g$.
We need to introduce some additional definitions in the following.

For $\lambda > 0$, we define the complex domain $\CD_{\lambda} \subset \C/\Z \times \C$ to be
\begin{equation*}
    \CD_\lambda = \{(x,y) \in \T_1 \times \C \ : \ |\Im x| + \lambda\inv |y| < 1\}.
\end{equation*}
If $u$ is a function defined on $\CD_{\lambda}$, we denote by $|u|^{(\lambda)}$ the supremum norm of $u$ over $\CD_{\lambda}$.
If $u$ is of class $\CC^r$ on $\CD_{\lambda}$, we set
\begin{equation}\label{eq:tri norm}
    |u|^{(\lambda, r)} := \sup\bigg\{ \biggl\vert\frac{\partial^{j+k} u}{\partial x^j \partial y^k}(x,y)\biggr\vert \ : \ (x,y) \in \CD_\lambda,\ 0 \le j+k \le r\bigg\}.
\end{equation}
The elementary properties of the $|\cdot|^{(\lambda, r)}$-norms that we will need are collected in Appendix \ref{sec:tri dom}.

\begin{figure}[h]
\centering
\begin{tikzpicture}[scale=2]
  \draw[->] (-1.4,0) -- (1.4,0) node[right] {$\Im x$};
  \draw[->] (0,0) -- (0,1.2) node[above] {$|y|$};

  \def\lam{0.7}
  \path[fill=gray, opacity=0.5]
    (-1,0) -- (0,\lam) -- (1,0) -- (1,0) -- cycle;

  \draw (1,0) node[below] {$1$} -- (1, -0.03);
  \draw (-1,0) node[below] {$-1$} -- (-1, -0.03);
  \draw (0,\lam) node[right] {$\lambda$} -- (0.03, \lam);
\end{tikzpicture}
\caption{An illustration of $\CD_\lambda$.}
\end{figure}

The following Lemma is the main result of the current section.
\begin{lemma}\label{lem:taylor rem est}
    Suppose 
    \begin{equation*}
        \|\dddot\phi\|  \leq 1 \quad \text{and} \quad 
        \verts{\omega - \frac{p}{q}} \leq \frac{q^{-6}}{6}.
    \end{equation*}
    Then $\wt g$ and $\wt S$ are analytic on $\CD_\lambda$  with $\lambda = q^{-5}/4$.
    Furthermore, for $(x,y) \in \CD_{\lambda}$, we have 
    \begin{equation}\label{eq:tg rem est}
        \verts{\wt g(x,y) - \big(\wt g_1(x) y + \wt g_2(x) y^2\big)} \lesssim q^{\nu_0}|y|^3
    \end{equation}
    and 
    \begin{equation}\label{eq:ts rem est}
        \verts{\wt S(x,y) - \wt s_2(x) y^2} \lesssim q^{\nu_0}|y|^3
    \end{equation}
    where $\nu_0 = 64$.
\end{lemma}

Before proceeding to the proof, we explain the technical issue that necessitates the introduction of the domains of the form $\CD_{\lambda}$.
The key issue is that the function $\wt g(\cdot, y)$ for a fixed $y \neq 0$ is in general less regular than $\phi: \T_1 \to C$, in the sense that it extends analytically only to a complex domain $\T_a$ for some $a < 1$. 
\footnote{This is analogous to the fact in billiard dynamics that the billiard map on the phase space is generally less smooth than the boundary of the convex billiard domain.}
The loss of regularity implies that the Taylor remainder in \eqref{eq:tg taylor} has smaller analyticity strip in $x$ whereas the leading terms $\wt g_j, j = 1,2$ are analytic over the larger domain $\T_1$.
By analyticity, the $k$-th Fourier coefficients of $\wt g_j$ decay exponentially faster in $k$ than that of the Taylor remainder. 
This is potentially a problem because, with $y \approx p/q - \omega$, the $\pm q$-th Fourier coefficient of $\wt g(\cdot, y)$, which are presumably the largest Fourier coefficients of the resonant part of $\wt g$, has the size $\sim |\omega - p/q|^2 \eu^{2\pi (a - 1) q}$ whereas that of the Taylor remainder has the size $\sim |\omega - p/q|^3$.
If $\eu^{2\pi (a - 1) q} \ll |\omega - p/q|$ then the uncontrolled error from the Taylor remainder will be much larger than the approximating functions in the estimates \eqref{eq:tg12 est} and \eqref{eq:ts2 est}.
Therefore, it is key to keep track of how fast $a$ decreases as $|y|$ increases from $0$.
More precisely, we aim to establish the analyticity of $\wt g$ over a domain $\CD_{\lambda}$, which by definition has a `triangular shape' in the sense that its $y$-slice is a complex neighborhood $\T_a$ such that $a$ decreases linearly from $1$ with a speed $\lambda\inv$ as $|y|$ increases from $0$.
Then, we restrict $\wt g$ and $\wt S$ on the slab $\CD_\lambda \cap \{|y - (p/q - \omega)| < |\omega - p/q|/2\}$, where the loss of analyticity width in the $x$-variable is uniformly bounded in $q$ by $\lesssim q^5 |\omega - p/q| \ll 1$.
This allows us to show that the first $N$ Fourier coefficients (with $N$ not too big relative to $q$) of the Taylor remainders of \eqref{eq:tg taylor} and \eqref{eq:tS Taylor} are much smaller than the counterparts of the leading terms, and hence they can be regarded truly as `error terms'.

Another technical issue arises when we consider the Taylor remainders over $\CD_\lambda$: it seems not straightforward to show that the remainders vanish to the order $y^3$.
If a function is analytic and bounded over a `rectangular' domain like $\T_1 \times \C$, then it is standard to use Cauchy's estimate to obtain that the $N$-th order Taylor remainder of the expansion at $y = 0$ vanishes to the order $y^N$.
For functions defined on $\CD_\lambda$, however, the Cauchy's estimate is not applicable.
Therefore, we resort to a more rudimentary approach by directly bounding the $\CC^r$ norms \eqref{eq:tri norm} of $\wt g$ and $\wt S$ viewed as a $\CC^\infty$ function on $\CD_\lambda$.
The lengthy but elementary estimates about functions on domains of the form $\CD_\lambda$ are collected in Appendix \ref{sec:tri dom}.
This approach is probably far from optimal and it is partially responsible for the large size of the exponent $\nu$ in the condition \eqref{eq:om arith cond} of Theorem \ref{thm:rnf}.

We now begin the proof of Lemma \ref{lem:taylor rem est}.

\paragraph{Chain of domains.}
Recall that, by construction,
\begin{equation*}
    \wt g = \big(g + \varphi \circ F - \varphi \big) \circ (\id + \wt \Sigma), \qquad \varphi = f + \psi \circ F - \psi
\end{equation*}
where $\id + \wt \Sigma$ is the right inverse of $\id + \Psi$.
We shall consider the chain of inclusion of domains schematised as follows
\begin{equation}\label{eq:dom chain}
    \begin{tikzcd}
    & \wt g  & g + \varphi\circ F - \varphi  & \varphi, f, g & \psi \\
    {\CD_{ \lambda_4}} \arrow[r, hook] & {\CD_{ \lambda_3}} \arrow[r, "\id + \wt\Sigma", hook] \arrow[u, dotted] & {\CD_{ \lambda_2}} \arrow[r, "F"] \arrow[u, dotted] & {\CD_{ \lambda_1}} \arrow[r, "F"] \arrow[u, dotted] & {\CD_{ \lambda_0}} \arrow[u, dotted]
    \end{tikzcd}
\end{equation}
where we shall fix
\begin{equation*}
    \lambda_0 = 1, \quad \lambda_1 = \frac{1}{q^5}, \quad \lambda_2 = \frac{1}{2q^5}, \quad \lambda_3 = \frac{1}{3 q^5} \quad \text{and} \quad \lambda_4 = \frac{1}{4 q^5}.
\end{equation*}
In the diagram \eqref{eq:dom chain}, the vertical arrows with dotted lines denote domains of definitions.
We will show $\wt g$ and $\wt S$ are well-defined analytic functions on $\CD_{\lambda_3}$ and establish a uniform $\CC^3$ bound
\begin{equation}\label{eq:tg est on Dlambda}
    |\wt g|^{(\lambda_4, 3)} \lesssim q^{\nu_0} \quad \text{and} \quad |\wt S|^{(\lambda_4, 3)} \lesssim q^{\nu_0}.
\end{equation}
with a universal constant $\nu_0 = 64$.

In fact, the estimate \eqref{eq:tg est on Dlambda} is sufficient for us to deduce \eqref{eq:tg rem est} and \eqref{eq:ts rem est} of Lemma \ref{lem:taylor rem est}.
Indeed, by the usual Taylor remainder formula,
\begin{equation*}
\begin{aligned}
    &\wt g(x,y) = \wt g_1(x) y + \wt g_2(x) y^2 + \CQ_1(x,y); \qquad& \CQ_1(x,y) = \frac{y^3}{2} \int_0^1 \frac{\partial^3 \wt g}{\partial y^3}(x, ty) (1 - t)^2 dt;\\
    &\tilde S(x,y) =  \tilde s_2(x) y^2 + \CQ_2(x,y), \qquad &\CQ_2(x,y) = \frac{y^3}{2} \int_0^1 \frac{\partial^3 \tilde S}{\partial y^3}(x, ty) (1 - t)^2 dt.
\end{aligned}
\end{equation*}
For $(x,y) \in \CD_{\lambda_4}$, by the explicit expressions above and \eqref{eq:tg est on Dlambda},
\begin{equation*}
    |\CQ_j(x,y)| \lesssim  q^{\nu_0} |y|^3, \qquad (j = 1,2).
\end{equation*}
We readily obtain \eqref{eq:tg rem est} and \eqref{eq:ts rem est}.

In the rest of this section, we focus on proving that all the compositions in \eqref{eq:dom chain} are well-defined and the estimate \eqref{eq:tg est on Dlambda} holds.

\paragraph{Estimates of $\ol f$ and $\ol g$.}

Recall that the functions $f$ and $g$ in the formula $F(x,y) = (x + \omega +y +f(x,y), y + g(x,y))$ are determined by $(\omega, \phi)$ via \eqref{eq:fg} and that $f(x,0) =g(x,0)  = 0$ for all $x$ so that $F$ preserves the real torus $\T \times \{0\}$. 
By our choice of $\psi$ in the preceding section and by \eqref{eq:varphi}, we clearly have $\varphi(x,0) = 0$ as well.
We may therefore put 
\begin{equation*}
    f(x,y) = y \ol f(x,y), \quad g(x,y) = y \ol g(x,y), \quad
    \psi(x,y) = y \ol \psi(x,y), \quad \text{and} \quad \varphi(x,y) = y \ol \varphi(x,y)
\end{equation*}
for some $\ol f$, $\ol g$, $\ol \psi$ and $\ol \varphi$ analytic in a complex neighbourhood of $\T \times \{0\}$.

Let us first focus on $\ol f$ and $\ol g$.
Substituting the definitions of $\ol f$ and $\ol g$ into \eqref{eq:fg} and rearranging,
\begin{equation}\label{eq:bar fg implicit}
    \begin{gathered}
        1 + \ol f(x,y) + \frac{\phi(x + \omega + y(1 + \ol f(x,y)))  -\phi(x + \omega)}{y} = \frac{1}{1+\dot\phi(x)}\\
        1 + \ol g(x,y) =  \big[1 + \dot\phi(x + \omega + y (1+ \ol f(x,y)))\big] \bigg[ 1 + \ol f(x,y) + \frac{\phi(x + y (1 + \ol f(x,y))) - \phi(x)}{y}\bigg].
    \end{gathered}
\end{equation}
By the Mean Value Theorem,
\begin{equation*}
\begin{aligned}
    \frac{\phi(x + \omega + y(1 + \ol f(x,y)))  -\phi(x + \omega)}{y} =& (1 + \ol f(x,y)) \int_0^1 \dot \phi(x + \omega + t (y +  f(x,y))) dt\\
    \frac{\phi(x  + y(1 + \ol f(x,y)))  -\phi(x )}{y} =& (1 + \ol f(x,y))\int_0^1 \dot \phi(x  + t (y +  f(x,y))) dt.
\end{aligned}
\end{equation*}
Hence, Equation \eqref{eq:bar fg implicit} can be written as 
\begin{equation}\label{eq:bar fg int}
    \begin{aligned}
        \ol f(x,y) = & \bigg[(1+\dot\phi(x)) \Big(1 + \int_0^1  \dot \phi(x + \omega + t (y +  f(x,y))) dt \Big)\bigg]\inv - 1\\
        \ol g(x,y) =& \frac{(1 + \dot\phi(x + \omega +  y+ f(x,y))) \big(1 + \int_0^1 \dot \phi(x  + t (y +  f(x,y))) dt \big)}{(1+\dot\phi(x))\big(1 + \int_0^1 \dot \phi(x + \omega + t (y +  f(x,y))) dt \big)} - 1.
    \end{aligned}
\end{equation}

Using \eqref{eq:bar fg implicit} and \eqref{eq:bar fg int}, we control the norms of $\bar f$ and $\bar g$ over domains $\CD_\lambda$ for a sufficiently small $\lambda$.

\begin{lemma}\label{lem:bar fg sup}
    Under the condition $\|\dddot\phi\|\leq 1$, the functions $\bar f$ and $\bar g$ are well-defined on the domain $\CD_{\lambda}$ with $\lambda = (1 - |\dot\phi|)^2$, and
    \begin{equation}\label{eq:bar fg sup}
        |\bar f|^{(\lambda)} \leq \frac{ |\dot\phi|(2 - |\dot\phi|)}{(1 - |\dot\phi|)^2}, \quad |1+\bar f|^{(\lambda)} \leq \frac{1}{(1 - |\dot\phi|)^2} \quad \text{and} \quad  
        |\bar g|^{(\lambda)} \leq  \frac{2|\dot\phi|(2 + |\dot\phi|)}{(1 - |\dot\phi|)^2}.
    \end{equation}
\end{lemma}
\begin{proof}
    We use \eqref{eq:bar fg implicit} and the contraction mapping principle to solve for $\bar f$.
    Let $\lambda = (1 - |\dot\phi|)^2$ and define $r = \lambda\inv - 1$. 
    Fix $(x,y) \in \CD_\lambda$ and consider the map
    \begin{equation*}
    \begin{aligned}
        \CP_{x,y}: \overline{\D}_r &\to \C \\
        z &\mapsto -\frac{\dot\phi(x)}{1 + \dot\phi(x)} - \frac{\phi(x + \omega + y(1 + z))  -\phi(x + \omega)}{y}.
    \end{aligned}
    \end{equation*}
    By definitions of $\CD_\lambda$ and $r$, when $|z|\leq r$ we have $|\Im(x + \omega + y(1 + z))| \leq |\Im x| + \lambda\inv |y| < 1$.
    Hence, the map  $\CP_{x,y}$ is well-defined on $\overline{\D}_r$, and  $\CP_{x,y}'(z) = -\dot\phi(x + \omega + y(1+z))$. 
    Since the assumption $\|\dddot\phi\| \leq 1$ implies \textit{a fortiori} that $|\dot\phi|<1$, the map $\CP_{x,y}$ is contracting over $\overline{\D}_r $.
    On the other hand, by the mean value theorem and the definitions of $\lambda$ and $r$, for $|z| \leq r$,
    \begin{equation*}
        |\CP_{x,y}(z)| \leq \frac{|\dot\phi|}{1 - |\dot\phi|} + (1 + r) |\dot\phi| = \Big(  \frac{1}{1 - |\dot\phi|} + \frac{1}{\lambda}\Big) |\dot\phi| = \frac{2 - |\dot\phi|}{(1 - |\dot\phi|)^2} |\dot\phi| = \frac{1}{(1 - |\dot\phi|)^2} - 1 = r.
    \end{equation*}
    It follows that $\CP_{x,y}(\overline{\D}_r) \subset \overline{\D}_r$.
    Hence, by the contracting mapping principle there exists  a unique fixed point $\CP_{x,y}(z) = z$ in $\overline{\D}_r$ depending analytically on $(x,y)$. 
    We shall define $\bar f(x,y)$ to be this fixed point.
    In particular, we have 
    \begin{equation*}
         |\bar f(x,y)| \leq r = \frac{|\dot\phi|(2 - |\dot\phi|)}{(1 - |\dot\phi|)^2} \quad \text{and} \quad |1+\bar f(x,y)| \leq 1 + r = \frac{1}{\lambda} = \frac{1}{(1 - |\dot\phi|)^2}.
    \end{equation*}
    We have obtained the first two inequalities of \eqref{eq:bar fg sup}.
    Finally, from the second line of \eqref{eq:bar fg int}, we deduce that $g$ is also well-defined on $\CD_\lambda$ and satisfies
    \begin{equation*}
        |\bar g(x,y)| \leq \frac{2 |\dot\phi|(2 + |\dot\phi|)}{(1 - |\dot\phi|)^2}.
    \end{equation*}
    The proof of the lemma is finished.
\end{proof}

Next, we apply Proposition~\ref{prop:tri comp} to estimate the norms $|\bar f|^{(\lambda', 2)}$ and $|\bar g|^{(\lambda', 2)}$ by shrinking $\CD_\lambda$ to a smaller domain $\CD_{\lambda'}$.

\begin{corollary}\label{cor:bar fg Cr}
    Suppose $\|\dddot\phi\| \leq 1$. 
    Then 
    \begin{equation*}
        |\bar f|^{(\lambda', 2)}+|\bar g|^{(\lambda', 2)} \lesssim 1 \quad \text{for} \quad \lambda' = (1 - |\dot\phi|)^2/2.
    \end{equation*}
\end{corollary}
\begin{proof}
To estimate $|\bar f|^{(\lambda', 2)}$ and $|\bar g|^{(\lambda', 2)}$, let us define $\Sigma(x,y) = (0, f(x,y))$ and 
\begin{equation*}
    \begin{aligned}
        w_1(x,y) =& \bigg[(1+\dot\phi(x)) \Big(1 + \int_0^1  \dot \phi(x + \omega + t y) dt \Big)\bigg]\inv - 1, \\
        w_2(x,y) =& \frac{(1 + \dot\phi(x + \omega +  y)) \big(1 + \int_0^1 \dot \phi(x  + t y) dt \big)}{(1+\dot\phi(x))\big(1 + \int_0^1 \dot \phi(x + \omega + t y) dt \big)} - 1.
    \end{aligned}
\end{equation*}
By \eqref{eq:bar fg int},
\begin{equation*}
    \bar f = w_1 \circ (\id + \Sigma) \quad \text{and} \quad \bar g = w_2 \circ (\id + \Sigma). 
\end{equation*}
Observe that for $0 \leq t \leq 1$, the expression $\dot \phi(x + \omega + ty)$ is well-defined for $|\Im x| + |y| <1$. 
Hence $w_1$ and $w_2$ are well-defined on $\CD_1$.
On the other hand, for $\lambda = (1 - |\dot\phi|)^2$ and $(x,y) \in \CD_{\lambda}$,
the results of Lemma~\ref{lem:bar fg sup} imply that $|\Im x| +|y+f(x,y)| \leq |\Im x| + (|1 + \bar f|^{(\lambda)}) |y| \leq |\Im x| + \lambda\inv |y| < 1$ with  $0 \leq t \leq 1$, and thus $(\id + \Sigma) (\CD_\lambda) \subset \CD_1$ and the map $w_j \circ (\id + \Sigma)$ is well-defined on $\CD_{\lambda}$.
Moreover, since the assumption $\|\dddot\phi\|\leq 1$ implies $|\dot\phi| \leq (2\pi)^{-2}$, computing the derivatives of $w_j,\ j = 1,2$ shows that the following crude bounds hold
\begin{equation}\label{eq:wij}
    |w_j|^{(1,1)} \lesssim 1 + \|\ddot\phi\| \lesssim 1 \quad \text{and} \quad  |w_j|^{(1,2)} \lesssim 1 + \|\dddot\phi\| \lesssim 1.
\end{equation}
Let us now apply Proposition \ref{prop:tri comp} part 3 with $v_1 = 0$, $v_2 = w_1$ , $\bar \sigma_1 = 0$, $\bar \sigma_2 = \bar f$ and $r = 2$.
We have accordingly that for all $\lambda' < \lambda = (1 - |\dot\phi|)^2$,
\footnote{The proposition is applied with $(\lambda, \lambda', \lambda'') \leftarrow (1, \lambda, \lambda')$.}
\begin{equation*}
    |\bar f|^{(\lambda', 2)} \leq C_2 \Big( \frac{2+ \lambda}{1 - (\lambda'(1 - |\dot\phi|)^{-2})^{\frac{1}{2}}}\Big)^{A_2} (1 + |w_1|^{(1, 1)})^{B_2} |w_1|^{(1, 2)}.
\end{equation*}
Fixing $\lambda' = \frac{1}{2}(1 - |\dot\phi|)^2$ and using \eqref{eq:wij}, we have
\begin{equation}\label{eq:bar f vs dddphi}
    |\bar f|^{(\lambda', 2)} \lesssim 1.
\end{equation}
At last, applying Proposition \ref{prop:tri comp} part 2 to $\bar g = w_2 \circ (\id + \Sigma)$ and using \eqref{eq:wij}, \eqref{eq:bar f vs dddphi} and the assumption $\|\dddot\phi\| \leq 1$, we have
\begin{equation*}
    |\bar g|^{(\lambda', 2)} \lesssim |w_2|^{(1, 2)} (1 + |\bar f|^{(\lambda', 2)})^2 \lesssim 1.
\end{equation*}
The proof is finished.
\end{proof}

We now consider the last two arrows of \eqref{eq:dom chain}.

Since $\lambda_1 < (1 - |\dot\phi|)^2$, it follows from Lemma \ref{lem:bar fg sup} that $F$ is well-defined on $\CD_{\lambda_1}$.

We first derive an extension to Proposition~\ref{prop:tri comp} part (1) and part (2) in order to verify that the last two arrows of \eqref{eq:dom chain} make sense and thus the compositions $\psi \circ F$ and $\varphi \circ F$ are well-defined.
Let decompose $F = \tau_\omega \circ F_0$ where $F_0(x,y) = (x + y +f(x,y), y + g(x,y))$ and $\tau_\omega(x,y) = (x + \omega, y)$.
Since $\tau_\omega(\CD_{\lambda}) = \CD_{\lambda}$ for any $\lambda > 0$, we just need to verify $F_0(\CD_{\lambda_2}) \subset \CD_{\lambda_1}$ and $F_0(\CD_{\lambda_1}) \subset \CD_{\lambda_0}$.

Fix $q \geq 2$. 
By the estimates in Lemma \ref{lem:bar fg sup} for $|\ol f|^{\lambda}$ and $|\ol g|^{\lambda}$ and the smallness of $|\dot\phi|$, one easily verifies that the condition \eqref{eq:tri o cond} of Proposition~\ref{prop:tri comp} is satisfied with $\bar \sigma_1 =1+ \bar f$, $\bar \sigma_2 = \bar g$ and $(\lambda', \lambda) \in \{(q^{-5}, 1),\;  (q^{-5}/2,q^{-5})\}$.
Thus, we may apply Proposition~\ref{prop:tri comp} part (1) with $\id + \Sigma = F_0$, which implies the desired inclusion.

Having verified that compositions $\psi \circ F$ and $\varphi \circ F$ are well-defined on $\CD_{\lambda_1}$ and on $\CD_{\lambda_2}$ respectively, we apply Proposition~\ref{prop:tri comp} part (2) with  $\sigma_1(x,y) = y + f(x,y)$,  $\sigma_2(x,y) = g(x,y)$, $v = \psi \circ \tau_\omega$ and $v = \varphi \circ \tau_\omega$ to get 
\begin{equation*}
\begin{cases}
    |\psi \circ F|^{(\lambda_1, 3)} = |\psi \circ \tau_\omega \circ F_0|^{(\lambda_1, 3)} \lesssim |\psi \circ \tau_\omega|^{(\lambda_0, 3)} (1 + |y+f|^{(\lambda_1, 3)} + |g|^{(\lambda_1, 3)})^3\\
    |\varphi \circ F|^{(\lambda_2, 3)} = |\varphi \circ \tau_\omega \circ F_0|^{(\lambda_2, 3)} \lesssim |\varphi \circ \tau_\omega|^{(\lambda_1, 3)} (1 + |y+f|^{(\lambda_2, 3)} + |g|^{(\lambda_2, 3)})^3.
\end{cases}
\end{equation*} 
Using Lemma~\ref{lem:tri cauchy} with $(\lambda', \lambda) = (\lambda_1, 2\lambda_1)$, we have 
\begin{equation*}
    1 + |y + f|^{(\lambda_1, 3)} + |g|^{(\lambda_1, 3)} \leq \frac{3 + \lambda_1}{1 - 1/2}   (1 + |1 + \bar f|^{(2\lambda_1, 2)} + |\bar g|^{(2\lambda_1, 2)}),
\end{equation*}
where, since $2\lambda_1 < (1 - |\dot\phi|)^2/2$, Corollary~\ref{cor:bar fg Cr} implies that the terms inside the parentheses in the preceding expression is $\lesssim 1$.
We also have the obvious equalities $|\psi \circ \tau_\omega|^{(\lambda_0, 3)} = ||\psi|^{(\lambda_0, 3)}$ and $|\varphi \circ \tau_\omega|^{(\lambda_0, 3)} = ||\varphi|^{(\lambda_0, 3)}$.
It follows that
\begin{equation}\label{eq:psiphi o F}
     |\psi \circ F|^{(\lambda_1, 3)} \lesssim |\psi|^{(\lambda_0, 3)} \quad \text{and} \quad |\varphi \circ F|^{(\lambda_2, 3)} \lesssim |\varphi|^{(\lambda_1, 3)}.
\end{equation}
By the same argument using Lemma \ref{lem:tri cauchy},
\begin{equation}\label{eq:fg lam}
    |f|^{(\lambda_1,3)} + |g|^{(\lambda_1, 3)} \lesssim 1.
\end{equation}

\paragraph{Estimates of $\psi$ and $\varphi$.}
Recall that $\bar \psi(x,y) = \psi_1(x) + \psi_2(x)y$ by construction. 
Applying Lemma~\ref{lem:tri cauchy} twice and using $\|\ddot\psi_1\|\lesssim q^2$ and $\|\dot\psi_2\|\lesssim q^4$ from \eqref{eq:leading norms recap}, we have
\begin{equation}\label{eq:psi lam}
    |\psi|^{(\lambda_0, 3)} \lesssim |\bar\psi|^{(2\lambda_0, 2)} \lesssim \|\ddot\psi_1\| + \|\dot\psi_2\| \lesssim q^4.
\end{equation}
Using the definition $\varphi = f + \psi \circ F - \psi$ and applying \eqref{eq:psiphi o F} - \eqref{eq:psi lam}, we obtain 
\begin{equation}\label{eq:phi lam}
    |\varphi|^{(\lambda_1, 3)} \lesssim |f|^{(\lambda_1,3)} + |\psi|^{(\lambda_0, 3)} \lesssim q^4 .
\end{equation}
Since $\varphi(x,y) = y \bar \varphi(x,y)$,
by  the mean value theorem, 
\begin{equation}\label{eq:barphi MVT}
    \bar \varphi(x,y) = \varphi_1(x) + y \int_0^1 \frac{\partial^2 \varphi}{\partial y^2}(x, ty)(1-t) dt.
\end{equation}
When $(x,y) \in \CD_{\lambda_1}$, we have $|y| \leq q^{-5}$. 
Therefore, the estimate \eqref{eq:phi lam} together with $\|\varphi_1\| < 1/18$ and $\|\ddot\varphi_1\| \lesssim q^2$ from \eqref{eq:leading norms recap} implies that, for some universal constant $c > 0$,
\begin{equation}\label{eq:barphi lam}
\begin{aligned}
    |\bar\varphi|^{(\lambda_1)} \leq&  \|\varphi_1\| + \frac{1}{2q^5} |\varphi|^{(\lambda_1, 2)} < \frac{1}{18} + \frac{c}{q};\\
    |\bar\varphi|^{(\lambda_2, 2)} \lesssim& \|\ddot\varphi_1\| + \verts{\frac{\partial^2 \varphi}{\partial y^2}}^{(\lambda_1, 1)} \lesssim q^2 + |\varphi|^{(\lambda_1, 3)} \lesssim q^4 . 
\end{aligned}
\end{equation}
In the second line, we have applied Lemma~\ref{lem:tri cauchy} to the integral term of \eqref{eq:barphi MVT}.

\paragraph{Estimates of $\id + \wt \Sigma$.}
Next, we consider the map $\id + \wt\Sigma$ in \eqref{eq:dom chain}.
To apply Proposition~\ref{prop:tri inv} with $\bar \xi_1 \leftarrow \bar \psi$, $\bar \xi_2 \leftarrow \bar \varphi$, $\Xi \leftarrow \Psi$ and $\Sigma \leftarrow \wt \Sigma$, we need to verify the condition \eqref{eq:tri inv} holds with $\lambda' = \lambda_3$, $\lambda = \lambda_2$ and $q$ sufficiently large, but this follows easily from \eqref{eq:psi lam} and \eqref{eq:barphi lam} as well as the definitions of $\lambda_3$ and $\lambda_2$. 

Accordingly, Proposition~\ref{prop:tri inv} applies and the right inverse $\id + \wt\Sigma$ of $\id + \Psi$ exists and is defined on $\CD_{\lambda_3}$.
By Proposition \ref{prop:tri comp} part (1) one also verifies that $(\id + \Psi)(\CD_{\lambda_4}) \subset \CD_{\lambda_3}$ for $q$ sufficiently large.
Hence, Proposition \ref{prop:tri inv} conclusion \ref{eq:tri left inv} implies that $(\id + \wt \Sigma) \circ (\id + \Psi) = \id$ on $\CD_{\lambda_4}$.
If we put $\wt F = (\id + \Psi) \circ F \circ (\id + \wt \Sigma)$, then in view of \eqref{eq:dom chain} the map $\wt F$ is well-defined on $\CD_{\lambda_3}$ and we have $ (\id + \Psi) \circ F  = \wt F \circ  (\id + \Psi)$ on $\CD_{\lambda_4}$.
Write $\wt F(x,y) = (x + \omega + y + \wt f(x,y), y + \wt g(x,y))$ for analytic $\wt f, \wt g: \CD_{\lambda_3} \to \C$. 
Then by the discussion in the beginning of Section \ref{sec:rnf proof} we have $\wt f \circ (\id + \Psi) = 0$ and $\wt g = (g + \varphi \circ F - \varphi) \circ (\id + \wt\Sigma)$.
By analyticity $\wt f$ vanishes identically.

Write $\wt \Sigma(x,y) = (y \bar \sigma_1(x,y), y \bar \sigma_2(x,y))$.
By the conclusion \eqref{eq:bar xi vs sig} of Proposition \ref{prop:tri inv} with $r = 2$ and $\lambda'' = 7\lambda_4/6$,
\begin{equation*}
    1 + |\bar \sigma_1|^{(\frac{7}{6}\lambda_4, 2)} + |\bar \sigma_2|^{(\frac{7}{6}\lambda_4, 2)}
    \leq C_2 \Big( \frac{2 + \lambda_3}{1 - \beta}\Big)^{A_2} \Big(\frac{1 + |\bar \psi|^{(\lambda_2, 1)} + |\bar \varphi|^{(\lambda_2, 1)}}{1 - |\bar \varphi|^{(\lambda_2)}}\Big)^{B_2} \big(1 + |\bar \psi|^{(\lambda_2, 2)} + |\bar \varphi|^{(\lambda_2, 2)}\big)
\end{equation*}
where $\beta = (\frac{7}{6}\frac{\lambda_4}{\lambda_3})^{1/2} = (7/8)^{1/2} < 1$.
By \eqref{eq:psi lam} and \eqref{eq:barphi lam}, we arrive at a crude bound:
\begin{equation}\label{eq:barsig}
    1 + |\bar \sigma_1|^{(\frac{7}{6}\lambda_4, 2)} + |\bar \sigma_2|^{(\frac{7}{6}\lambda_4, 2)}
    \lesssim  q^{4(B_2 + 1)}.
\end{equation}

\paragraph{Estimate of $\wt g$.}
Finally, we may bound $|\wt g|^{(\lambda_4, 3)}$ as follows. 
\begin{equation*}
    \begin{aligned}
        |\wt g|^{( \lambda_4, 3)} 
        &=|(g + \varphi \circ F - \varphi) \circ (\id + \wt\Sigma)|^{( \lambda_4, 3)} & \\
        &\lesssim |g + \varphi \circ F - \varphi|^{( \lambda_2, 3)} (1 + |\sigma_1|^{( \lambda_4, 3)} + |\sigma_2|^{( \lambda_4, 3)})^3 &\text{by Proposition~\ref{prop:tri comp} part (2)}\\
        & \lesssim |g + \varphi \circ F - \varphi|^{( \lambda_2, 3)} (1 + |\bar \sigma_1|^{( \frac{7}{6}\lambda_4, 2)} + |\bar \sigma_2|^{( \frac{7}{6}\lambda_4, 2)})^3 & \text{by Lemma~\ref{lem:tri cauchy}}\\
        & \lesssim q^{12(B_2 + 1)} |g + \varphi \circ F - \varphi|^{( \lambda_2, 3)}  & \text{by \eqref{eq:barsig}}\\
        & \lesssim q^{12(B_2 + 1)} ( |g|^{( \lambda_2, 3)} + |\varphi|^{( \lambda_1, 3)} ) & \text{by \eqref{eq:psiphi o F}}\\
        & \lesssim q^{12(B_2 + 1)} q^4  & \text{by \eqref{eq:fg lam} and \eqref{eq:phi lam}}.
    \end{aligned}    
\end{equation*}
By the explicit recursive formula for $B_r$ at the end of the proof of \ref{prop:tri comp}, one can readily compute that $B_2 = 4$.
Therefore, we have obtained $|\wt g|^{(\lambda_4, 3)} \lesssim q^{64}$, which is the first inequality of \eqref{eq:tg est on Dlambda}.

\paragraph{Estimate of $\wt S$.}
We start by estimating the $|\cdot|^{(\lambda_2, 3)}$-norm of the normalized generating function $S$ associated with $(\omega, \phi)$  in adapted coordinates.
By \eqref{eq:RIC gen func-fg},
\begin{equation*}
    \frac{\partial S}{\partial x} = y \Big[ (1 + \ol g)\Big(1 + \frac{\partial f}{\partial x}\Big) -1 \Big] \quad \text{and} \quad \frac{\partial S}{\partial y} = y (1 + \ol g)\Big(1 + \frac{\partial f}{\partial y}\Big).
\end{equation*}
It follows from Lemma~\ref{lem:tri cauchy}, Corollary~\ref{cor:bar fg Cr} and \eqref{eq:fg lam} that 
\begin{equation*}
    \begin{aligned}
        \verts{\frac{\partial S}{\partial x}}^{(\lambda_2, 2)} \lesssim&  \verts{(1 + \ol g)\Big(1 + \frac{\partial f}{\partial x}\Big) - 1}^{(\lambda_1, 1)} \lesssim (1 + |\ol g|^{(\lambda_1, 1)}) \Big(1 + \verts{\frac{\partial f}{\partial x}}^{(\lambda_1, 1)}\Big) \lesssim 1\\
        \verts{\frac{\partial S}{\partial y}}^{(\lambda_2, 2)} \lesssim&  \verts{(1 + \ol g)\Big(1 + \frac{\partial f}{\partial y}\Big) }^{(\lambda_1, 1)} \lesssim (1 + |\ol g|^{(\lambda_1, 1)}) \Big(1 + \verts{\frac{\partial f}{\partial y}}^{(\lambda_1, 1)}\Big) \lesssim 1.
    \end{aligned}
\end{equation*}
Applying the second estimate above to the Taylor remainder of
\begin{equation*}
    S(x,y) = \frac{1 + f_1(x)}{2} y^2 + y^3 \int_0^1 \frac{\partial^3 S}{\partial y^3}(x, ty) (1 - t)^2 \du t
\end{equation*}
and using $\|f_1\| \lesssim 1$, we get $|S|^{(\lambda_2)} \lesssim 1$ and thus,

\begin{equation}\label{eq:|S| lam}
    |S|^{(\lambda_2, 3)} \leq \max \bigg\{ |S|^{(\lambda_2)}, \verts{\frac{\partial S}{\partial x}}^{(\lambda_2, 2)}, \verts{\frac{\partial S}{\partial y}}^{(\lambda_2, 2)}\bigg\} \lesssim 1.
\end{equation}
It follows that
\begin{equation*}
    \begin{aligned}
        |\wt S|^{(\lambda_4, 3)} 
        \lesssim& |S|^{(\lambda_2, 3)} (1 + |\sigma_1|^{(\lambda_4, 3)} + |\sigma_2|^{(\lambda_4, 3)})^3 & \text{by Proposition \ref{prop:tri comp} part (2)}\\
        \lesssim& |S|^{(\lambda_2, 3)} (1 + |\bar \sigma_1|^{(7\lambda_4/6,\ 2)} + |\bar \sigma_2|^{(7\lambda_4/6,\ 2)})^3 & \text{by Lemma \ref{lem:tri cauchy}}\\
        \lesssim& q^{12(B_2 + 1)} & \text{by \eqref{eq:barsig} and \eqref{eq:|S| lam}.}
    \end{aligned}
\end{equation*}
Using $B_2 = 4$ we obtain $|\wt S|^{(\lambda_4, 3)} \lesssim q^{60}$, which obviously implies the second inequality of \eqref{eq:tg est on Dlambda}.
In view of the discussion immediately following \eqref{eq:tg est on Dlambda}, we have finished the proof of Lemma \ref{lem:tg12 ts2 est}.

\subsubsection{End of the proof of Proposition~\ref{prop:rnf1}}\label{sec:end prf rnf1}

We now complete the proof of  Proposition~\ref{prop:rnf1}.
Let us fix from now onwards that $\lambda = q^{-5}/4$ and let $\CD_\lambda$ and the $|\cdot|^{(\lambda, r)}$-norm be defined as in the beginning of Section \ref{sec:rnf1 proof taylor}.
Take $(x,y) \in \CD_{\lambda}$.
Combining Lemma~\ref{lem:tg12 ts2 est} and Lemma \ref{lem:taylor rem est},
\begin{equation}\label{eq:tri residue}
    \begin{gathered}
        \bigg|\wt g(x,y) -  y\Big(\omega - \frac{p} {q} + \frac{y}{2} \Big) \frac{\langle \dot f_1 \rangle_q(x)}{1 + \langle f_1 \rangle_q(x)}\bigg| \lesssim q^2 \verts{\omega - \frac{p}{q}} |y| + q^5 \verts{\omega - \frac{p}{q}} |y|^2 +  q^{\nu_0}|y|^3\\
        \bigg|\wt S(x,y) -  \frac{1+f_1(x)}{2(1 + \langle f_1 \rangle_q(x))^2 }y^2 \bigg| \lesssim \verts{\omega - \frac{p}{q}} |y|^2 + q^{\nu_0} |y|^3.
    \end{gathered}
\end{equation}
Recall that in the preceding paragraph we have shown $(\id + \Psi) \circ (\id + \wt\Sigma) = \id$ over $\CD_{\lambda}$ and 
\begin{equation*}
    (\id + \Psi) \circ F \circ (\id + \wt \Sigma): \begin{pmatrix}
        x \\ y 
    \end{pmatrix} \mapsto \begin{pmatrix}
        x + \omega + y \\ y + \wt g(x,y)
    \end{pmatrix}
\end{equation*}
is analytic on $\CD_{\lambda}$.
Define a translation $\tau(x,y) = (x, p/q - \omega + y)$.
Then we have 
\begin{equation}\label{eq:tau-translated id}
    \tau\inv \circ (\id + \Psi) \circ (\id + \wt\Sigma) \circ \tau = \id \quad \text{over} \quad \tau\inv \CD_{\lambda}.
\end{equation}
Using assumption \eqref{eq:rnf1 arithm}, one readily checks that $\T_{1 - q\inv} \times \D_{|\omega - p/q|/2} \subset \tau\inv \CD_{\lambda}$. 
We prove the conclusion of Proposition \ref{prop:rnf1} with the following choices.
\begin{equation}\label{eq:Thetaq}
    \Theta := (\id + \wt\Sigma) \circ \tau; \quad \wt F_q := \tau\inv \circ (\id + \Psi) \circ F \circ \Theta; \quad \wt S_q := S \circ \Theta; \quad  \wt g_q := \wt g \circ \tau.
\end{equation}
Then $\wt F_q(x,y) = (x + p/q + y, y + \wt g_q(x, y)) $ for $\wt g_q: \T_{1 - q\inv} \times \D_{|\omega - p/q|/2} \to \C$ analytic.

Let us fix $(x,y) \in \T_{1 - q\inv} \times \D_{|\omega - p/q|/2}$.

Define $\theta_1, \theta_2:\T_{1 - q\inv} \times \D_{|\omega - p/q|/2} \to \C$ as in \eqref{eq:th12}.
By \eqref{eq:tau-translated id} and the definition of $\Theta$, we have $(\id + \Psi) \circ \Theta = \tau$. 
Componentwise, this is means that
\begin{equation}\label{eq:th12-psiphi}
    \begin{cases}
        \theta_1(x,y) + \psi(x + \theta_1(x,y),\; y + p/q - \omega + \theta_2(x,y)) = 0\\
        \theta_2(x,y) + \varphi(x + \theta_1(x,y),\; y + p/q - \omega + \theta_2(x,y)) = 0.
    \end{cases}
\end{equation}
Since $\Theta(x,y) \subset \CD_{\lambda_2}$, by the first line of \eqref{eq:barphi lam} and $\|\phi_1\| < 1/18$, we have 
\begin{equation*}
    |\varphi(x + \theta_1(x,y),\; y + p/q - \omega + \theta_2(x,y))| < \frac{1}{18} |y + p/q - \omega + \theta_2(x,y)|
\end{equation*}
for $q$ sufficiently large.
Combining with the second line of \eqref{eq:th12-psiphi} and using $|y| < |\omega - p/q|/2$,
\begin{equation*}
    |\theta_2(x,y)| < \frac{1}{18}\Big(\frac{3}{2} \verts{\omega - \frac{p}{q}} + |\theta_2(x,y)| \Big).
\end{equation*}
The second estimate of \eqref{eq:th12 est} follows from above.
Recall that $\psi(x,y) = y \psi_1(x) + y^2 \psi_2(x)$ and from \eqref{eq:leading norms recap} we have $\|\psi_1\| \lesssim q$ and $\|\psi_2\| \lesssim q^3$. 
Therefore the first line of \eqref{eq:th12-psiphi} and $|\theta_2(x,y)| \lesssim |\omega - p/q|$ imply
\begin{equation*}
    |\theta_1(x,y)| \lesssim q \verts{\omega - \frac{p}{q}}.
\end{equation*}
We have used $|\omega - p/q| \leq q^{-6}$. The proof of \eqref{eq:th12 est} is complete.

Using $|p/q - \omega + y| \lesssim |\omega - p/q|$ and \eqref{eq:tri residue},
\begin{equation*}
    \begin{gathered}
        \bigg|\wt g_q(x,y) + \frac{1}{2} \Big(\frac{p} {q} - \omega + y \Big) \Big(\frac{p} {q} - \omega -y \Big) \frac{\langle \dot f_1 \rangle_q(x)}{1 + \langle f_1 \rangle_q(x)}\bigg| \lesssim  q^{\nu_0}\verts{\omega - \frac{p}{q}}^3\\
        \bigg|\wt S_q(x,y) - \Big(\frac{p}{q} - \omega + y\Big)^2 \frac{1+f_1(x)}{2(1 + \langle f_1 \rangle_q(x))^2 } \bigg| \lesssim q^{\nu_0} \verts{\omega - \frac{p}{q}}^3,
    \end{gathered}
\end{equation*}
which give \eqref{eq:rnf1 residue}.
It is clear from the construction that the change of coordinates $\id + \wt \Sigma$ depends analytically on $\phi$ and it is real-analytic when $\phi$ is so.

Recall that $\Sigma(x,y) = (\sigma_1(x,y), \sigma_2(x,y)) = (y\bar \sigma_1(x,y), y \bar \sigma_2(x,y)) $.
By \eqref{eq:barsig}, Lemma \ref{lem:tri cauchy} and $B_2 = 4$, $\lambda = \lambda_4$,
\begin{equation}\label{eq:sig Cr}
    |\bar \sigma_j|^{(\lambda, 2)} \lesssim q^{20} \qquad (j = 1,2).
\end{equation}
Since the original map $F$ preserves the area form $\du x \du y$, the map $\wt F_q$ defined by \eqref{eq:Thetaq} preserves the area form $\Theta^* \du x \du y = \tau^* (\id + \wt\Sigma)^* \du x \du y$, where 
\begin{equation*}
    (\id + \wt\Sigma)^* \du x \du y = \Big(1 + \frac{\partial \sigma_1}{\partial x} + \frac{\partial \sigma_2}{\partial y} + \frac{\partial \sigma_1}{\partial x}  \frac{\partial \sigma_2}{\partial y} - \frac{\partial \sigma_1}{\partial y} \frac{\partial \sigma_2}{\partial x}   \Big) \du x \du y.
\end{equation*}
Hence we may write 
\begin{equation*}
    \Theta^* \du x \du y = \Big(  \frac{1}{1 + \langle f_1 \rangle_q} + \wt \eta\Big) \du x \du y  
\end{equation*}
where 
\begin{equation*}
    \wt \eta = \frac{1}{1 + \varphi_1} - \frac{1}{1 + \langle f_1 \rangle_q} + \tau^*\bigg[\Big(\frac{\partial \sigma_2}{\partial y}(x,y) + 1 - \frac{1}{1 + \varphi_1(x)}\Big) + \Big(\frac{\partial \sigma_1}{\partial x}\Big(1 + \frac{\partial \sigma_2}{\partial y} \Big) - \frac{\partial \sigma_1}{\partial y} \frac{\partial \sigma_2}{\partial x}    \Big)\bigg].
\end{equation*}
Recall that the construction in Section \ref{sec:rnf1 proof taylor} gives $\varphi_1 = \langle f_1 \rangle_q + \{\varphi_1 \}_q$ and we have an upper bound $\|\{\varphi_1\}_q\| \lesssim q^2 |\omega - p/q|^2$ from \eqref{eq:leading norms recap}.
Therefore, it can be verified with mean value theorem that
\begin{equation*}
    \norm{\frac{1}{1 + \varphi_1} - \frac{1}{1 + \langle f_1 \rangle_q}} \lesssim q^2 \verts{\omega - \frac{p}{q}}^2.
\end{equation*}
On the other hand, Taylor expanding the $y$-component of the equality $\wt\Sigma + (\id + \Psi) \circ (\id + \wt\Sigma) = 0$ we have $\partial \sigma_2/\partial y (x, 0) = (1 + \varphi_1(x))\inv - 1$.
By \eqref{eq:sig Cr} and $\sigma_j(x, 0) = 0$,  we have that, for $(x,y) \in \CD_{\lambda}$,
\begin{equation*}
    \begin{aligned}
        &\verts{\frac{\partial \sigma_2}{\partial y}(x,y) + 1 - \frac{1}{1 + \varphi_1(x)}} \lesssim |y| |\sigma_2|^{(\lambda, 2)} \lesssim q^{20} |y|, \\
        &\verts{\frac{\partial \sigma_j}{\partial x}(x,y)} \lesssim |y||\sigma_j|^{(\lambda, 2)} \lesssim q^{20} |y| \quad  \text{and} \quad
        \verts{\frac{\partial \sigma_j}{\partial y}(x,y)} \lesssim |\sigma_j|^{(\lambda, 1)} \lesssim q^{20}.
    \end{aligned}
\end{equation*}
Since the translation $\tau$ maps  $(\T_{1 - q\inv} \times \D_{|\omega - p/q|/2})$ into $\CD_{\lambda}$, applying the above estimates to $\wt \eta$ over the domain $\T_{1 - q\inv} \times \D_{|\omega - p/q|/2}$, we get 
\begin{equation*}
    |\wt \eta|_{1 - \frac{1}{q}, \frac{1}{2}|\omega - \frac{p}{q}} \lesssim q^{40} \verts{\omega - \frac{p}{q}}.
\end{equation*}
Since $40 < \nu_0$, we have proved \eqref{eq:rnf1 area form}.
The proof of Proposition~\ref{prop:rnf1} is finished.

\subsection{Step 2: coordinate changes by an iterative scheme}\label{sec:rnf part 2}

\begin{proposition}\label{prop:rnf2}
    We consider the same setup as Proposition~\ref{prop:rnf1}. 
    Fix $\nu > \nu_0$, $M \geq 1$ and $\kappa \in \Z^+$.
    Then for all $q \in \Z^+$ sufficiently large and satisfying 
    \begin{equation}\label{eq:rnf2 arithm cond}
        \verts{\omega  - \frac{p}{q}} \leq q^{-\nu},
    \end{equation}
    and with $\wt F_q = \Theta\inv \circ F \circ \Theta: (x,y) \mapsto (x + p/q + y, y + \wt g_q(x,y))$ denoting the standard map in the coordinates constructed in Proposition~\ref{prop:rnf1}, the following conclusion holds true. 

    There exists a further change of coordinates
    \begin{equation*}
        \id + \Sigma:\T_{1 - \frac{2}{q}} \times \D_{\frac{1}{3}|\omega - \frac{p}{q}|} \to \T_{1 - \frac{1}{q}} \times \D_{\frac{1}{2} |\omega - \frac{p}{q}|}
    \end{equation*}
    such that the map $F_q := (\id + \Sigma)\inv \circ \wt F_q \circ (\id + \Sigma)$ takes the form 
    $(x, y) \mapsto (x + p/q + y, y +  g_q(x,y))$
    for some analytic $ g_q$
    satisfying 
    \begin{align}
        \label{eq:rnf2 iter concl 1}
        | g_q|_{1 - \frac{2}{q}, \frac{1}{3}|\omega - \frac{p}{q}|} \lesssim& \verts{\omega - \frac{p}{q}}^2 \\
        \label{eq:rnf2 iter concl 2}
        |\{ g_q\}_q|_{1 - \frac{2}{q}, \frac{1}{3}|\omega - \frac{p}{q}|} \lesssim& \eu^{-2\pi M q} \verts{\omega - \frac{p}{q}}^2 \\
        \label{eq:rnf2 iter res compare}
        \| \langle g_q (\cdot ,0) \rangle_q^N - \langle \wt  g_q (\cdot ,0) \rangle_q^N\| \lesssim& \kappa \eu^{4\pi \kappa} q^{\nu_0 - \nu} \verts{\omega - \frac{p}{q}}^2 \quad \text{where} \quad N = \kappa q.
    \end{align}
    Moreover, writing $\Sigma(x,y) = (\sigma_1(x,y), \sigma_2(x,y))$ for analytic $\sigma_1, \sigma_2: \T_{1 - 2/q} \times \D_{|\omega - p/q|/3} \to \C $, we have 
    \begin{equation}\label{eq:rnf2 iter psiphi}
        |\sigma_j|_{1 - \frac{2}{q}, \frac{1}{3}|\omega - \frac{p}{q}|}\lesssim q^{5 + \nu_0 - \nu} \verts{\omega - \frac{p}{q}}^2 , \qquad (j = 1,2).
    \end{equation}
\end{proposition}

\begin{remark}
    The `sufficiently large' condition on $q$ depends only on $\nu$, $M$ and $\kappa$. 
    In particular, it is independent of the map $F$ as long as $\|\dddot\phi\| \leq 1$.
    If the irrationality exponent of $\omega$ is $\geq \nu$, then there exist infinitely many $q$ satisfying \eqref{eq:rnf2 arithm cond}.
\end{remark}

\subsubsection{An iterative lemma for Proposition~\ref{prop:rnf2}}

Let us consider generally a map
\begin{equation}\label{eq:rnf2 orig map}
    F: (x,y) \mapsto (x + p/q + y, y + g(x,y))
\end{equation}
for analytic $g: \T_{a} \times \D_{b} \to \C$.
We consider a change of coordinates $\Phi = \id + \Psi$ where $\Psi(x,y) = (\psi(x,y), \varphi(x,y))$ and $\varphi = \psi \circ F - \psi$.
A straightforward calculation analogous to that in Section~\ref{sec:rnf1 proof taylor} shows that, in the new coordinates, the transformed map $\Phi \circ F \circ \Phi\inv$ writes 
\begin{equation*}
    \Phi \circ F \circ \Phi\inv(x,y) = (x + p/q + y, y + \wt g(x,y))
\end{equation*}
where $\wt g$ is an analytic function satisfying
\begin{equation}\label{eq:rnf2 hom eq}
    \wt g \circ \Phi(x,y) = g + \psi \circ F^2 - 2 \psi \circ F + \psi.
\end{equation}
In the following lemma, we show that a judicious choice of $\psi$ can reduce the nonresonant part of the map in the new coordinates in the sense that $\wt g$ has a `smaller' resonant part compared with $g$.

To state this result, we need to introduce a new norm on the space of analytic functions: for $u: \T_a \times \D_b \to \C$ analytic, let us define 
\begin{equation*}
    \normm{u}_{a,b} = \sup_{|y| < b,\ k \in \Z} |\hat u(y)_k| \eu^{2\pi a |k|}.
\end{equation*}
Recall that $\hat u(y)_k$ is the $k$-th Fourier coefficient of $u(\cdot,y)$.
Since $|\hat u(y)_k| \leq |u|_{a,b}\eu^{-2\pi a|k|}$,
\begin{equation*}
    \normm{u}_{a,b} \leq |u|_{a,b}.
\end{equation*}
Also observe that for $N \in \Z_{\geq 1}$ and $a' < a$,
\begin{equation}\label{eq:triple norm truncation}
    \begin{aligned}
        |\crochet{u}^{>N}|_{a',b} 
    \leq& \sum_{|k| \geq N+1} \sup_{|y| < b} |\hat u(y)_k| \eu^{2\pi a' |k|} 
    \leq \sum_{|k| \geq N+1} \normm{u}_{a,b} \eu^{-2\pi (a-a') |k|}\\
    \leq&  \frac{\eu^{-2\pi (a-a') (N+1)}}{1- \eu^{-2\pi (a-a')}}\normm{u}_{a,b}
    \leq \frac{\eu^{-2\pi (a-a')N}}{2\pi (a - a')}\normm{u}_{a,b}.
\end{aligned}
\end{equation}
In the last line, we used the elementary inequality $\eu^{-x}(1 - \eu^{-x})\inv < x\inv$ for $x \in \R$.

\begin{lemma}\label{lem:rnf2 one step}
    Let $\epsilon > 0$, $q \in \Z_{> 1}$, $N \in \Z_{\geq q}$, $5b < \alpha < a$ and $5\epsilon < \beta < b$ such that
    \begin{equation}\label{eq:rnf2 one step assump}
        qNb \leq \frac{1}{2} \quad \text{and} \quad \frac{q^2 N \epsilon}{\beta} \leq \frac{1}{2}.
    \end{equation}
    Let $F$ be a map of the form \eqref{eq:rnf2 orig map} with $|g|_{a,b} \leq \epsilon$.
    Then there exists a near-identity analytic change of coordinates
    \begin{equation*}
        \id + \Sigma: \T_{a - \alpha} \times \D_{b - \beta} \to \T_{a - \alpha/2} \times \D_{b - \beta/2},
    \end{equation*} 
    where $\Sigma(x,y) = (\sigma_1(x,y), \sigma_2(x,y))$ for some $\sigma_1, \sigma_2: \T_{a - \alpha} \times \D_{b - \beta} \to \C$ analytic and
    \begin{equation}\label{eq:rnf2 one step psiphi norm}
        |\sigma_1|_{a - \alpha, b - \beta} + |\sigma_2|_{a - \alpha, b - \beta}< q^2 N \normm{g - \langle g\rangle_q^N}_{a,b},
    \end{equation}
    and the map $(\id + \Sigma) \circ F \circ (\id + \Sigma)\inv$ in the new coordinates writes 
    \begin{equation*}
        (x,y) \mapsto (x + p/q + y, y + \wt g(x,y))
    \end{equation*}
    where $\wt g: \T_{a - \alpha} \times \D_{b - \beta} \to \C$ is real-analytic and
    \begin{equation}\label{eq:rnf2 one step conclusion}
        |\wt g - \langle g\rangle_q^N|_{a - \alpha, b - \beta} < \lambda \normm{g - \langle g\rangle_q^N}_{a,b} \quad \text{and} \quad 
        |\wt g|_{a - \alpha, b - \beta} < |g|_{a,b} + \lambda \normm{g - \langle g\rangle_q^N}_{a,b}, 
    \end{equation}
    where 
    \begin{equation*}
        \lambda = 12 \Big( \frac{q^2 N \epsilon}{\beta} + \frac{\eu^{- \pi \alpha N}}{\alpha} \Big).
    \end{equation*}
\end{lemma}
\begin{proof}
    We shall construct $\id + \Sigma$ as the inverse of the map $\Phi = \id + \Psi$ where $\Psi(x,y) = (\psi(x,y), \varphi(x,y))$ and $\varphi := \psi \circ F - \psi$.
    Let $p, q, a,b, \epsilon, N, \alpha, \beta$ be given as in the statement.
    Set the $k$-th Fourier coefficient of  $\psi(\cdot , y)$ to be 
    \begin{equation*}
        \hat \psi(y)_k = \frac{\hat g(y)_k}{2 \eu^{2\pi \iu k (p/q + y)}[1 - \cos(2\pi k (p/q + y))]} \qquad \text{($q \nmid k$ and $|k| \leq N$),}
    \end{equation*}
    and $\hat \psi(y)_k = 0$ for all other $k$.
    By Fourier expansion, this choice of $\psi$ is equivalent to 
    \begin{equation*}
        \psi \circ A^2 - 2 \psi \circ A + \psi = - \{g\}_q^N.
    \end{equation*}
    Hence, by \eqref{eq:rnf2 hom eq},
    \begin{equation*}
        \wt g \circ \Phi = \langle g \rangle_q^N + \{g\}_q^{> N} + (\psi \circ F^2 - \psi \circ A^2) + 2 (\psi \circ A - \psi \circ F).
    \end{equation*}
    Let us denote 
    \begin{equation*}
        \CR' = (\psi \circ F^2 - \psi \circ A^2) + 2 (\psi \circ A - \psi \circ F)
    \end{equation*}
    and 
    \begin{equation*}
        \CR = \langle g \rangle_q^{>N} + \Big( \langle g \rangle_q \circ \Phi\inv - \langle g \rangle_q \Big) + \{g\}_q^{>N} \circ \Phi\inv + \CR' \circ \Phi\inv
    \end{equation*}
    so that 
    \begin{equation*}
        \wt g = \langle g \rangle_q^N + \CR.
    \end{equation*}
    \noindent \textit{Step 1: estimate of $\psi$ and $\varphi$.} 
    By assumptions $Nb \leq 1/(2q)$ and $q \geq 2$, for $|y| < b$, $|k| \leq N$ and $q \nmid k$,
    \begin{equation*}
        \bigg|1 - \cos 2\pi k \Big(\frac{p}{q} + y\Big)\bigg| \geq 1 - \cos \frac{\pi}{q} \geq \frac{4}{q^2}.
    \end{equation*} 
    We have used the elementary inequality $1 - \cos(2\pi x) \geq 16 x^2$ for $|x| \leq 1/4$.
    Thus, by definition of $\psi$, we have 
    \begin{equation}\label{eq:rnf2 psi norm}
        \begin{aligned}
            |\psi|_{a, b} 
            \leq& \sum_{\substack{q \nmid k,\\ |k| \leq N}} \frac{q^2}{8} \sup_{|y| < b}|\hat g(y)_k| 
            \leq \sum_{\substack{q \nmid k,\\ |k| \leq N}}  \frac{q^2}{8} \normm{\{g\}_q^N}_{a,b} \eu^{- 2\pi |k| a} \\
            <& 2N \frac{q^2}{8} \normm{\{g\}_q^N}_{a,b} \eu^{- 2\pi a} < \frac{q^2 N}{4} \normm{\{g\}_q^N}_{a,b} 
            \leq \frac{q^2 N}{4} \normm{g - \langle g\rangle_q^N}_{a,b} .
        \end{aligned}
    \end{equation}
    Since $\normm{g - \langle g\rangle_q^N}_{a,b} \leq |g|_{a,b}$, we have in particular that 
    \begin{equation}\label{eq:rnf2 psi-eps}
        |\psi|_{a, b} < \frac{q^2 N \epsilon}{4}.
    \end{equation}
    Since $|g|_{a,b} \leq \epsilon$, we also have $F(\T_{a - b + \epsilon} \times \D_{b - \epsilon}) \subset \T_a \times \D_b$.
    It follows from $\varphi = \psi \circ F - \psi$ that 
    \begin{equation}\label{eq:rnf2 phi norm}
        |\varphi|_{a - b + \epsilon, b - \epsilon} \leq 2 |\psi|_{a, b}.
    \end{equation}
    
    \noindent \textit{Step 2: estimate of $\CR'$ over $\T_{a'} \times \D_{b'}$.}
    Let $a' = a - \alpha/2$ and $b' = b - \beta /2$.
Using the explicit expression of $F$ and the bound $|g|_{a,b} \leq \epsilon$ as well as the assumptions $5 \epsilon < \beta$ and $5 b < \alpha$, we deduce that the images of $\T_{a'} \times \D_{b'}$ under the maps $F^2$, $A^2$, $F$ and $A$ are all contained in $\T_{a - \alpha/10} \times \D_{b - \beta/10}$. \footnote{These images are  actually contained in an even smaller domain, but we will not need this stronger fact.}
    By mean value theorem and Cauchy's estimate, 
    \begin{equation*}
        \begin{aligned}
            &|\psi \circ F^2 - \psi \circ A^2|_{a',b'} \\
            = & \sup_{|\Im x| < a',\ |y| < b'} \verts{\psi\Big(x + \frac{2p}{q} + 2y + g(x,y), y + g(x,y) + g\circ F(x,y)\Big) - \psi\Big(x + \frac{2p}{q} + 2y, y\Big)} \\
            = & |g|_{a,b}\verts{\frac{\partial \psi}{\partial x}}_{a - \alpha/10, b - \beta/10}  + 2 |g|_{a,b}\verts{\frac{\partial \psi}{\partial y}}_{a - \alpha/10, b - \beta/10} \\
            \leq&  \Big(\frac{10}{\alpha} + \frac{20}{\beta}\Big) |\psi|_{a,b} |g|_{a,b}
             < \frac{22}{\beta}|\psi|_{a,b} |g|_{a,b}.
        \end{aligned}
    \end{equation*}
    In the last inequality, we used $\alpha > 5 b > 5\beta$.
    Similarly,
    \begin{equation*}
        |\psi \circ F - \psi \circ A|_{a',b'} 
        = |g|_{a,b} \verts{\frac{\partial \psi}{\partial y}}_{a - \alpha/10, b - \beta/10} 
        = \frac{10}{\beta} |\psi|_{a,b} |g|_{a,b}.
    \end{equation*}
    We have proved 
    \begin{equation*}
        |\CR'|_{a',b'} \leq \frac{42}{\beta} |\psi|_{a,b} |g|_{a,b}.
    \end{equation*}

    \noindent \textit{Step 3: inverting $\Phi$.}
    Observe that since $|g|_{a,b} \leq \epsilon$, the function $\varphi = \psi \circ F - \psi$ is well-defined on $\T_{a - b + \epsilon} \times \D_{b - \epsilon}$ and $|\varphi|_{a - b + \epsilon, b - \epsilon} \leq 2 |\psi|_{a,b}$.
    We shall apply Lemma \ref{lem:rect inv} with $a_0 = a - b + \epsilon$, $b_0 = b - \epsilon$, $\delta_1 = a_0 - a'$ and $\delta_2 = b_0 - b'$.
    Since $a' = a - \alpha/ 2$ and $b' = b - \beta/2$, we have $\delta_1 > 3\alpha/10$ and $\delta_2 > 3 \beta /10$ using the assumptions $5 \epsilon < \beta$ and $5 b < \alpha$.
    By \eqref{eq:rnf2 psi-eps} and $\alpha > 5 \beta$, one verifies that the second assumption in \eqref{eq:rnf2 one step assump} ensures the condition \eqref{eq:rect inv cond}. 
    It now follows from Lemma~\ref{lem:rect inv} that $\Phi$ is a diffeomorphism from $\T_{a'} \times \D_{b'}$ onto the image, which contains $\T_{a' - \delta_1} \times \D_{b' - \delta_2}$. 
    In particular, we may define 
    $\Phi\inv = \id + \Sigma: \T_{a - \alpha} \times \D_{b - \beta} \to \T_{a'} \times \D_{b'}$
    where $\Sigma(x,y) = (\sigma_1(x,y), \sigma_2(x,y))$ with 
    \begin{equation}\label{eq:rnf2 sig12 norm}
        |\sigma_1|_{a - \alpha, b - \beta} \leq |\psi|_{a',b'} ; \qquad |\sigma_2|_{a - \alpha, b - \beta} \leq |\varphi|_{a', b'}.
    \end{equation}
    Combining with \eqref{eq:rnf2 psi norm} and \eqref{eq:rnf2 phi norm}, we obtain \eqref{eq:rnf2 one step psiphi norm}.
    Since we have shown that $\wt g \circ \Phi$ is well-defined on $\T_{a'} \times \D_{b'}$, by precomposing with $(\id + \Sigma)$, we deduce that $\wt g$ is well-defined on $\T_{a - \alpha} \times \D_{b - \beta}$ as desired.

    \noindent \textit{Step 4: estimate of $\CR'$ over $\T_{a - \alpha} \times \D_{b - \beta}$.}
    Similar to step 2, using mean value theorem and Cauchy's estimate as well as \eqref{eq:rnf2 sig12 norm}, we obtain 
    \begin{equation*}
        \begin{aligned}
            \verts{\crochet{g}_q \circ \Phi\inv - \crochet{g}_q}_{a - \alpha, b - \beta} 
            \leq & \verts{\frac{\partial \crochet{g}_q}{\partial x}}_{a', b'} |\sigma_1|_{a - \alpha, b - \beta} + \verts{\frac{\partial \crochet{g}_q}{\partial y}}_{a', b'} |\sigma_2|_{a - \alpha, b - \beta}\\
            \leq & \Big(\frac{2|\psi|_{a, b}}{\alpha} + \frac{4|\psi|_{a, b}}{\beta}\Big)|\crochet{g}_q|_{a,b} \\
            \leq & \frac{22}{5 \beta}|\psi|_{a, b}|g|_{a,b}.
        \end{aligned}
    \end{equation*}
    On the other hand, applying \eqref{eq:triple norm truncation} to $ u=\crochet{g}_q$ and $u=\{g\}_q$, 
    \begin{equation*}
        \begin{aligned}
            |\crochet{g}_q^{>N}|_{a - \alpha, b - \beta} \leq& \frac{\eu^{-2\pi \alpha N}}{2\pi \alpha} \normm{\crochet{g}_q^{>N}}_{a,b} \\
            |\{g\}_q^{>N} \circ \Phi\inv|_{a - \alpha, b - \beta} \leq& |\{g\}_q^{>N}|_{a', b'} \leq  \frac{\eu^{-\pi \alpha N}}{\pi \alpha} \normm{\{g\}_q^{>N}}_{a,b}.
        \end{aligned}
    \end{equation*}
    Lastly, we have the obvious bound $|\CR' \circ \Phi\inv|_{a - \alpha, b - \beta} \leq |\CR'|_{a', b'}$
    Thus,
    \begin{equation*}
        \begin{aligned}
            |\CR |_{a - \alpha, b - \beta} 
            \leq& \frac{\eu^{-2\pi \alpha N}}{2\pi \alpha}\normm{\crochet{g}_q^{>N}}_{a,b}   
            + \frac{\eu^{-\pi \alpha N}}{\pi \alpha} \normm{\{g\}_q^{>N}}_{a,b}
            +\frac{22}{5 \beta}|\psi|_{a, b}|g|_{a,b}
            +\frac{42}{\beta} |\psi|_{a,b} |g|_{a,b}\\
            \leq& \bigg[\Big(\frac{1}{\pi} + \frac{\eu^{-\pi \alpha N}}{2\pi}\Big)  \frac{\eu^{-\pi \alpha N}}{\alpha}+ \Big(\frac{22}{5} + 42\Big) \frac{q^2 N \epsilon}{4\beta} \bigg]  \normm{g - \crochet{g}_q^N}_{a,b}\\
            <& 12 \Big( \frac{q^2 N \epsilon}{\beta} + \frac{\eu^{-\pi \alpha N}}{\alpha}\Big)\normm{g - \crochet{g}_q^N}_{a,b}.
        \end{aligned}
    \end{equation*}
    We have proven the first line of \eqref{eq:rnf2 one step conclusion}.
    To prove the second line of \eqref{eq:rnf2 one step conclusion}, it is enough to apply $|\crochet{g}_q|_{a,b} \leq |g|_{a,b}$ and the preceding estimates to  
    \begin{equation*}
        \wt g = \crochet{g}_q + \Big( \langle g \rangle_q \circ \Phi\inv - \langle g \rangle_q \Big) + \{g\}_q^{>N} \circ \Phi\inv + \CR' \circ \Phi\inv.
    \end{equation*}
\end{proof}

\begin{lemma}\label{lem:rect inv}
    Suppose $\psi, \varphi: \T_{a_0} \times \D_{b_0} \to \C$ are analytic and, for some $\delta_1, \delta_2 > 0$,
    \begin{equation}\label{eq:rect inv cond}
        \frac{|\psi|_{a_0,b_0}}{\delta_1} + \frac{|\varphi|_{a_0,b_0}}{\delta_2} < 1.
    \end{equation}
    Then the map $\Phi = \id + \Psi$ with $\Psi(x,y) = (\psi(x,y), \varphi(x,y))$ is an analytic diffeomorphism from $\T_{a_0 - \delta_1} \times \D_{b_0 - \delta_2}$ onto its image and surjects onto $\T_{a_0 - 2\delta_1} \times \D_{b_0 - 2\delta_2}$.
    The inverse $\id + \Sigma: \T_{a_0 - 2\delta_1} \times \D_{b_0 - 2\delta_2} \to \T_{a_0 - \delta_1} \times \D_{b_0 - \delta_2}$ of $\Phi$ satisfies $\Sigma(x,y) = (\sigma_1(x,y), \sigma_2(x,y))$ and 
    \begin{equation*}
        |\sigma_1|_{a_0 - 2\delta_1} \leq |\psi|_{a_0 - \delta_1}; \quad 
        |\sigma_2|_{a_0 - 2\delta_1} \leq |\varphi|_{a_0 - \delta_1}
    \end{equation*}
\end{lemma}

\subsubsection{Proof of Proposition~\ref{prop:rnf2}}\label{sec:rnf part 2 pf}

\begin{proof}
    By Proposition~\ref{prop:rnf1}, there exists a universal constant $c$ such that, for sufficiently large $q$, we have 
    \begin{equation*}
        |\wt g_q|_{1-1/q, |\omega - p/q|/2} \leq c (1 - \eu^{-2\pi M}) \verts{\omega - \frac{p}{q}}^2.
    \end{equation*}
    Let us pick an integer $n \in (2, \nu - 3)$
    and define
    \begin{equation*}
        \epsilon = c \verts{\omega - \frac{p}{q}}^2; \quad \fN = q^n; \quad \alpha = \frac{1}{q^2}; \quad \beta = \frac{1}{6q} \verts{\omega - \frac{p}{q}}; \quad b = \frac{1}{2}\verts{\omega - \frac{p}{q}}.
    \end{equation*}
    By the choice of $n$, there exists $q_0 \gg 1$ such that for all $q \geq q_0$ satisfying \eqref{eq:rnf2 arithm cond}, 
    \begin{equation}\label{eq:rnf2 assump unpacked}
        \begin{aligned}
            &\frac{b}{\alpha} = \frac{q^2}{2}|\omega - p/q| \leq \frac{q^{2 - \nu}}{2} < \frac{1}{5};\\
            &\frac{\epsilon}{\beta} = 6cq \verts{\omega - \frac{p}{q}} \leq 6 c q^{1 - \nu} < \frac15;\\
            &q N b = \frac{1}{2}q^{n + 1} \verts{\omega - \frac{p}{q}} \leq \frac{1}{2} q^{n + 1 - \nu} \leq \frac{1}{2};\\
            &\frac{q^2 N \epsilon}{\beta} = \frac{6c q^{3 + n}}{|\omega - p/q|} \leq 6 c q^{3 + n - \nu} < \frac{\eu^{-2\pi M}}{24};\\
            &\frac{1}{\alpha}\eu^{-\pi \alpha \fN} = q^2 \eu^{-\pi q^{n - 2}} < \frac{\eu^{-2\pi M}}{24};\\
            &\eu^{-2\pi q^{n-1}} \leq q^{\nu_0 - \nu} \quad \text{and} \quad q \geq \kappa.
        \end{aligned}
    \end{equation}
    Fix such a $q$ and define 
    \begin{equation*}
        a= 1 - \frac{1}{q}; \qquad a' = 1 - \frac{2}{q}; \qquad b' = \frac{1}{3} \verts{\omega - \frac{p}{q}}.
    \end{equation*}
    The first five lines of \eqref{eq:rnf2 assump unpacked} ensure that for any $\tilde a \in (a' + \alpha, a)$ and $\tilde b \in (b' + \beta, b)$, Lemma \ref{lem:rnf2 one step} applies to any map $F$ of the form \eqref{eq:rnf2 orig map} with analytic $g: \T_{\tilde a} \times \D_{\tilde b} \to \C$ such that $|g|_{\tilde a, \tilde b} \leq \epsilon$.
    The lemma produces a coordinate transform $\id + \Sigma$ such that the function $\wt g$ in the new coordinates satisfies 
    \begin{equation*}
        \begin{aligned}
            |\wt g - \langle g\rangle_q^\fN|_{\tilde a - \alpha, \tilde b - \beta} <& \eu^{-2\pi M} \normm{g - \langle g\rangle_q^\fN}_{\tilde a,\tilde b} \\ 
            |\wt g|_{\tilde a - \alpha, \tilde b - \beta} <& |g|_{\tilde a,\tilde b} + \eu^{-2\pi M} \normm{g - \langle g\rangle_q^\fN}_{\tilde a,\tilde b}.
        \end{aligned}
    \end{equation*}
    For notational convenience, we define 
    \begin{equation*}
        \CT g:= \wt g; \qquad \CR g:= \wt g - \langle g\rangle_q^\fN
    \end{equation*}
    so that the preceding inequalities write
    \begin{equation}\label{eq:rnf2 one step application}
        \begin{aligned}
            |\CR g|_{\tilde a - \alpha, \tilde b - \beta} <& \eu^{-2\pi M} \normm{g - \langle g\rangle_q^\fN}_{\tilde a,\tilde b}\\
            |\CT g|_{\tilde a - \alpha, \tilde b - \beta} <& |g|_{\tilde a,\tilde b} + \eu^{-2\pi M} \normm{g - \langle g\rangle_q^\fN}_{\tilde a,\tilde b}.
        \end{aligned}
    \end{equation}
    Let us also put $a_j = a - j \alpha$ and $b_j = b - j \beta$. 
    Then $a_q = a'$ and $b_q = b'$.
    In order to iterate Lemma \ref{lem:rnf2 one step} multiple times starting from an analytic $g: \T_a \times \D_b \to \C$, we suppose
    \begin{equation}\label{eq:rnf2 iter |g|}
        |g|_{a,b} \leq (1 - \eu^{-2\pi M}) \epsilon.
    \end{equation}
    Then we show by induction on $1 \leq j \leq q$ that, after repeatedly applying Lemma \ref{lem:rnf2 one step},
    \begin{gather}
        \label{eq:rnf2 iter ind hyp 1}
        \CT^j g = \crochet{g}_q^\fN + \sum_{k = 1}^{j-1} \crochet{\CR \CT^{k-1} g}_q^\fN + \CR \CT^{j-1} g,\\
        \label{eq:rnf2 iter ind hyp 2}
        |\CR \CT^{j-1}g|_{a_j, b_j} < \eu^{-2\pi M j} \normm{g - \crochet{g}_q^\fN}_{a,b} \quad \text{and} \quad |\CT^j g|_{a_j, b_j} < \sum_{k=0}^{j} \eu^{-2\pi M k} |g|_{a,b}.
    \end{gather}
    The case of $j = 1$ follows from definitions of $\CT g$ and $\CR g$ as well as \eqref{eq:rnf2 one step application}.
    Let us assume \eqref{eq:rnf2 iter ind hyp 1} and \eqref{eq:rnf2 iter ind hyp 2} hold for some $1 \leq j < q$.
    Then, the assumption \eqref{eq:rnf2 iter |g|} ensures that $|\CT^j g|_{a_j, b_j} < (1 - \eu^{-2\pi M})\inv |g|_{a,b} \leq \epsilon$ and hence Lemma \ref{lem:rnf2 one step} applies to $\CT^j g$. 
    Notice that $\CT^j g - \langle \CT^j g \rangle_q^\fN = \CR \CT^{j-1} g - \langle \CR \CT^{j-1} g \rangle_q^\fN$.
    Therefore, by \eqref{eq:rnf2 one step application} with $g$ replaced by $\CT^j g$ and the induction hypothesis,
    \begin{equation}\label{eq:rnf2 iter ind step 1}
        \begin{aligned}
            |\CR\CT^j g|_{a_{j+1}, b_{j+1}} <& \eu^{-2\pi M}\normm{\CT^j g - \langle \CT^j g \rangle_q^\fN}_{a_j,  b_j} \\
            =& \eu^{-2\pi M} \normm{\CR \CT^{j-1} g - \langle \CR \CT^{j-1} g \rangle_q^\fN}_{a_j,  b_j} \\
            \leq& \eu^{-2\pi M}|\CR \CT^{j-1} g|_{a_j, b_j} \\
            <& \eu^{-2\pi M (j+1)} \normm{g - \crochet{g}_q^\fN}_{a,b}
        \end{aligned}
    \end{equation}
    Thus, we obtain the first inequality of \eqref{eq:rnf2 iter ind hyp 2} for the case $j + 1$.
    To obtain the second inequality of \eqref{eq:rnf2 iter ind hyp 2}, we apply the second inequality of  \eqref{eq:rnf2 one step application} with $g$ replaced by $\CT^j g$ and use the induction hypothesis and \eqref{eq:rnf2 iter ind step 1} to obtain
    \begin{equation*}
        \begin{aligned}
            |\CT^{j+1} g|_{a_{j+1}, b_{j+1}} <& |\CT^j g|_{a_j, b_j} + \eu^{-2\pi M} \normm{\CT^j g - \langle \CT^j g \rangle_q^N}_{a_j, b_j} \\
            <& \sum_{k=0}^{j} \eu^{-2\pi M k} |g|_{a,b} + \eu^{-2\pi M (j+1)} \normm{g - \langle g \rangle_q^N}_{a,b} 
            \leq \sum_{k=0}^{j+1} \eu^{-2\pi M k} |g|_{a,b}.
        \end{aligned}
    \end{equation*}
    The induction is finished.
    Let us now take $g = \wt g_q$ for the analytic function $\wt g_q$ obtained from Proposition~\ref{prop:rnf1} and take $j = q$ and $ g_q := \CT^q \wt g_q$.
    We readily get 
    \begin{equation*}
        \begin{aligned}
            |\{ g_q\}_q|_{a', b'} =& |\{\CR \CT^{q-1} \wt g_q\}_q|_{a',b'} \leq 2|\CR \CT^{q-1} \wt g_q|_{a',b'}  
            < 2\eu^{-2\pi M q} \normm{\wt g_q - \langle \wt g_q \rangle_q^\fN}_{a,b}\\
            \leq& 2\eu^{-2\pi M q} |\wt g_q|_{a,b}  \lesssim \eu^{-2\pi M q} \verts{\omega - \frac{p}{q}}^2
        \end{aligned}
    \end{equation*}
    and 
    \begin{equation*}
        \begin{aligned}
            |g_q|_{a',b'} < \frac{|\wt g_q|_{a,b}}{(1 - \eu^{-2\pi M})} \lesssim \verts{\omega - \frac{p}{q}}^2.
        \end{aligned}
    \end{equation*}
    We have proved \eqref{eq:rnf2 iter concl 1} and \eqref{eq:rnf2 iter concl 2}.
    We now show \eqref{eq:rnf2 iter res compare}.
    Since $N = \kappa q \leq \fN$ by the last inequality in \eqref{eq:rnf2 assump unpacked}, Equation \eqref{eq:rnf2 iter ind hyp 1} implies
    \begin{equation*}
        \langle g_q \rangle_q^{N} - \langle \wt g_q \rangle_q^{N} = \sum_{j = 1}^{q} \langle \CR \CT^{j - 1} \wt g_q\rangle_q^{N}.
    \end{equation*}
    By the first inequality in \eqref{eq:rnf2 iter ind hyp 2}, we have $|\CR\CT^{j-1} \wt g_q|_{a_j, b_j} < \eu^{-2\pi Mj} \normm{\wt g_q - \langle \wt g_q \rangle_q^\fN}_{a,b}$. 
    Thus, fixing $y = 0$ and substituting the definition of $a_j$, we find
    \begin{equation*}
        \begin{aligned}
            \|\langle \CR \CT^{j - 1} \wt g_q(\cdot , 0)\rangle_q^{N}\| 
            \leq & \sum_{|k| \leq N,\ q \mid k} |\CR\CT^{j-1} \wt g_q|_{a_j, b_j} \eu^{-2\pi a_j |k|} \eu^{2\pi |k|}\\
            = & \sum_{|k| \leq N,\ q \mid k} \eu^{-2\pi Mj} \normm{\wt g_q - \langle \wt g_q \rangle_q^\fN}_{a,b} \eu^{2\pi (1/q + j/q^2) |k|}\\
            \leq & \Big(\frac{2N}{q} + 1\Big) \eu^{-2\pi Mj} \normm{\wt g_q - \langle \wt g_q \rangle_q^\fN}_{a,b} \eu^{2\pi \frac{2N}{q}}\\
            \leq & (2\kappa + 1) \eu^{-2\pi Mj} \normm{\wt g_q - \langle \wt g_q \rangle_q^\fN}_{a,b} \eu^{4\pi \kappa}\\
            \lesssim& \kappa \eu^{4\pi \kappa} \eu^{-2\pi Mj} \normm{\wt g_q - \langle \wt g_q \rangle_q^\fN}_{a,b}.
        \end{aligned}
    \end{equation*}
    Summing over $1 \leq j \leq q$ and using $M \geq 1$, we obtain
    \begin{equation}\label{eq:rnf2 res compare prelim}
        \|\langle g_q(\cdot, 0) \rangle_q^{N} - \langle \wt g_q(\cdot, 0) \rangle_q^{N}\| \lesssim \kappa \eu^{4\pi \kappa} \normm{\wt g_q - \langle \wt g_q \rangle_q^\fN}_{a,b}.
    \end{equation}
    On the other hand, we may write 
    $$
        \wt g_q - \langle \wt  g_q\rangle_q^\fN  = \{\wt g_q\}_q + \langle \wt g_q + h\rangle_q^{> \fN} - \langle h \rangle^{>\fN},
    $$ 
    where 
    \begin{equation*}
        h(x,y) =  \frac{1}{2}\Big(\frac{p}{q} - \omega + y\Big)\Big(\frac{p}{q} - \omega - y\Big)\frac{\langle \dot f_1 \rangle_q(x)}{1 + \langle f_1 \rangle_q(x)}.
    \end{equation*}
    Then, by the first line in \eqref{eq:rnf1 residue},
    we have 
    \begin{equation*}
        \normm{\{\wt g_q\}_q + \langle \wt g_q + h\rangle_q^{> \fN}}_{a, b} \leq |\wt g_q + h|_{a,b} \lesssim q^{\nu_0} \verts{\omega - \frac{p}{q}}^3 \lesssim q^{\nu_0 - \nu}  \verts{\omega - \frac{p}{q}}^2
    \end{equation*}
    and, since the $\|\cdot \|$-norm of $\langle \dot f_1 \rangle_q(1 + \langle f_1 \rangle_q)\inv$ is bounded, by the choice of $\fN$ and the last line of \eqref{eq:rnf2 assump unpacked},
    \begin{equation*}
        \normm{\langle h \rangle^{>\fN}}_{a, b} \lesssim \eu^{-2\pi (1 - a)\fN}\normm{\langle h \rangle^{>\fN}}_{1, b} \lesssim \eu^{-2\pi q^{n-1}} \verts{\omega - \frac{p}{q}}^2 \lesssim q^{\nu_0 - \nu}\verts{\omega - \frac{p}{q}}^2.
    \end{equation*}
    It follows that
    \begin{equation}\label{eq:rnf2 orig res compare}
        \begin{aligned}
            \normm{\langle \wt  g_q\rangle_q^\fN - \wt g_q}_{a, b} 
            \lesssim q^{\nu_0 - \nu}\verts{\omega - \frac{p}{q}}^2.
        \end{aligned}
    \end{equation}
    Combining with \eqref{eq:rnf2 res compare prelim} we obtain \eqref{eq:rnf2 iter res compare}.

    It now remains to show \eqref{eq:rnf2 iter psiphi}.
    Let us denote by $\id + \Sigma_j: \T_{a_j} \times \D_{b_j} \to \T_{a_{j-1}} \times \D_{b_{j-1}}$ the coordinate transform obtained from the $j$-th iteration of the Lemma.
    Let us write $\Sigma_j(x,y) = (\sigma_{1,j}(x,y), \sigma_{2,j}(x,y))$ for some analytic $\sigma_{1,j } , \sigma_{2,j}: \T_{a_j} \times \D_{b_j} \to \C$.
    By \eqref{eq:rnf2 one step psiphi norm}, \eqref{eq:rnf2 iter ind hyp 1}, \eqref{eq:rnf2 iter ind hyp 2} and \eqref{eq:rnf2 orig res compare},
    \begin{equation*}
        \begin{aligned}            
            |\sigma_{1,j }|_{a_j, b_j} + |\sigma_{2,j }|_{a_j, b_j} <& q^2 \fN \normm{\CT^{j-1} g_q - \langle \CT^{j-1} g_q\rangle_q^\fN}_{a_{j-1}, b_{j-1}} \\
            \leq& q^2 \fN |\CR \CT^{j-2} g_q|_{a_{j-1}, b_{j-1}}\\
            \leq& q^2 \fN \eu^{-2\pi M(j-1)}\normm{g_q - \langle g_q\rangle_q^\fN}_{a,b}\\
            \leq&  \eu^{-2\pi M(j-1)}q^{2 + n + \nu_0 - \nu} \verts{\omega - \frac{p}{q}}^2.
    \end{aligned}
    \end{equation*}
    On the other hand, by composing $(\id + \Sigma):= (\id + \Sigma_1) \circ (\id + \Sigma_2) \circ \cdots \circ (\id + \Sigma_q)$, we find that 
    \begin{equation*}
        \Sigma = \sum_{j = 1}^{q} \Sigma_j \circ (\id + \Sigma_{j+1}) \circ (\id + \Sigma_{j+2}) \circ \cdots \circ (\id + \Sigma_q).
    \end{equation*}
    Therefore, writing $\Sigma(x,y) = (\sigma_1(x,y), \sigma_2(x,y))$, we have 
    \begin{equation*}
        |\sigma_1|_{a', b'} + |\sigma_2|_{a',b'} \leq \sum_{j = 1}^{q} |\sigma_{1,j}|_{a_j, b_j} +  |\sigma_{2,j}|_{a_j, b_j} 
        < \sum_{j = 1}^{q} \eu^{-2\pi M(j-1)} q^{2 + n + \nu_0 - \nu} \verts{\omega - \frac{p}{q}}^2.
    \end{equation*}
    For concreteness, we may fix $n = 3$. 
    Then the proof of \eqref{eq:rnf2 iter psiphi} is finished by summing the geometric series in $j$.
\end{proof}

\subsection{Step 3: control $g^*_q$ using exactness and the end of the proof}\label{sec:rnf part 3}

We are now in position to prove Theorem~\ref{thm:rnf}.
So far, we  have essentially obtained all the required results in the statement of Theorem \ref{thm:rnf} with the exception of part (3). 
The proof of part (3) will be the main focus of this section.
The main idea for the proof of part (3) is to use the exactness property of the original standard map and the fact that the resonant coordinates do not `distort' the standard area form by much. 
Then, using an elementary geometric property of exact symplectic map, we can make use of the estimate of the zero-average part $g^\bullet_q$ to estimate the average part $g^*_q$.
We remark that a similar idea was also used in \cite{Martín_2016}.

\begin{proof}[Proof of Theorem \ref{thm:rnf}]
    Fix $\omega \in \R \setminus \Q$ and $\phi: \T_1 \to \C$ analytic with $\|\dddot\phi\| \leq 1$.
    Let $\kappa \geq 2$ and $\nu > \nu_0 := 64$ be given as in the statement of Theorem~\ref{thm:rnf}.
    Consider $q \gg 1$ such that the arithmetic condition \eqref{eq:om arith cond} holds and let $p$ be the corresponding minimizer of $|q\omega - p|$ over $p \in \Z$.
    Since the condition \eqref{eq:om arith cond} is stronger than the arithmetic condition \eqref{eq:rnf1 arithm} of Proposition~\ref{prop:rnf1}, taking $q$ sufficiently large ensures that Proposition~\ref{prop:rnf1} applies and we accordingly obtain a coordinate transform $\Theta$ and a map 
    \begin{equation*}
        \begin{aligned}
            \wt F_q = \Theta\inv \circ F \circ \Theta: \T_{1 - 1/q} \times \D_{|\omega - p/q|/2} \to& \T_{1} \times \C\\
            (x,y) \mapsto& \Big(x + \frac{p}{q} + y,\ y + \wt g_q(x,y)\Big)
        \end{aligned}
    \end{equation*}
    satisfying the conclusions of Proposition~\ref{prop:rnf1}.

    Increasing $q$ if necessary, we may then apply Proposition~\ref{prop:rnf2} with $M = \kappa +1$ (and $\nu$, $\kappa$ as given in the theorem statement) to obtain a second coordinate transform $\id + \Sigma$ and a map
    \begin{equation*}
        \begin{aligned}
            F_q = (\id + \Sigma)\inv \circ \wt F_q \circ (\id + \Sigma): \T_{1 - 2/q} \times \D_{|\omega - p/q|/3} \to& \T_{1} \times \C\\
            (x,y) \mapsto& \Big(x + \frac{p}{q} + y,\ y + g_q(x,y)\Big)
        \end{aligned}
    \end{equation*}
    satisfying the conclusions of Proposition~\ref{prop:rnf2}.
    We now prove that the change of coordinates $H_q := \Theta \circ (\id + \Sigma)$, the map $F_q$ and the function $g_q$ satisfy the conclusions of Theorem~\ref{thm:rnf}.

    Let $a = 1 - 3/q$ and $b = |\omega - p/q|/5$ as in the theorem statement.
    
    Let $h_{q,1}$ and $h_{q,2}$ be defined as in \eqref{eq:h12}. By \eqref{eq:th12} and $\Sigma(x,y) = (\sigma_1(x,y), \sigma_2(x,y))$, we have 
    \begin{equation*}
        h_{q,1} = \sigma_1 + \theta_1 \circ (\id + \Sigma), 
        \qquad h_{q,2} =  \sigma_2 + \theta_2 \circ (\id + \Sigma).
    \end{equation*}
    Since $(\id + \Sigma)(\T_{a} \times \D_{b}) \subset \T_{1 - 1/q} \times \D_{|\omega - p/q|/2}$, it follows from \eqref{eq:th12 est} and \eqref{eq:rnf2 iter psiphi} as well as the choice $\nu \geq \nu_0 + 6$ that 
    \begin{equation*}
        |h_{q,1}|_{a,b} \lesssim q\verts{\omega - \frac{p}{q}}; \qquad 
        |h_{q,2}|_{a,b} < \frac{1}{10} \verts{\omega - \frac{p}{q}}
    \end{equation*}
    for $q$ sufficiently large.
    We have obtained \eqref{eq:h12 est}.

    Next, from the conclusion \eqref{eq:rnf2 iter concl 1} of Proposition~\ref{prop:rnf2} we have immediately  \eqref{eq:rnf concl 1}.
    Next, to obtain \eqref{eq:rnf concl 2}, we use \eqref{eq:rnf2 iter concl 1}, \eqref{eq:rnf2 iter concl 2} and the choices $N = \kappa q$, $M = \kappa / 2$ to get 
    \begin{equation*}
        \begin{aligned}
            |g_q - \langle g_q \rangle_q^N|_{1/q,b} \leq&
            |\{g_q\}_q|_{1/q,b} + |\langle g_q \rangle_q^{>N}|_{1/q,b}\\
            \lesssim& |\{g_q\}_q|_{1-2/q, |\omega - p/q|/3} + \sum_{|j| > N, \ q \mid j} |g_q|_{1 - 2/q, |\omega - p/q|/3} \eu^{-2\pi (1 - 2/q - 1/q)|j|}\\
            \lesssim&  \eu^{-2\pi M q} \verts{\omega - \frac{p}{q}}^2 + \sum_{|j| > N, \ q \mid j}  \eu^{-2\pi(1-3/q) |j|} \verts{\omega - \frac{p}{q}}^2\\
            \lesssim& \eu^{-2\pi M q} \verts{\omega - \frac{p}{q}}^2 + \eu^{6\pi \kappa}\eu^{-2\pi (\kappa+1) q} \verts{\omega - \frac{p}{q}}^2
            \lesssim \eu^{6\pi \kappa} \eu^{-2\pi (\kappa+1) q} \verts{\omega - \frac{p}{q}}^2.
        \end{aligned}
    \end{equation*}
    We have proven \eqref{eq:rnf concl 2}.

    We shall postpone the proof of \eqref{eq:gq star} to the end of this section.
    To prove \eqref{eq:rnf res compare}, we denote 
    \begin{equation*}
        \CQ(x,y) =  \wt g_q(x , y) + \frac{1}{2} \Big(\frac{p}{q} - \omega + y\Big) \Big(\frac{p}{q} - \omega - y\Big) \frac{\langle \dot f_1 \rangle_q(x)}{1 + \langle f_1 \rangle_q(x)}.
    \end{equation*}
    Then, by the first inequality of \eqref{eq:rnf1 residue}, we have $|\CQ|_{1-1/q, |\omega - p/q|/2} \lesssim q^{\nu_0} |\omega - p/q|^3$.
    It follows from the definition of the $\|\cdot \|$-norm and the choice $N = \kappa q$ that
    \begin{equation*}
        \begin{aligned}
            \norm{\crochet{\CQ(\cdot, 0)}_q^N} =& \sum_{\substack{q \mid k \\ |k| \leq N}} |\hat \CQ(0)_k| \eu^{2\pi |k|}  
        \leq \sum_{\substack{q \mid k \\ |k| \leq N}} | \CQ|_{1 - 1/q, |\omega - p/q|/2}  \eu^{-2\pi (1 - 1/q) |k|} e^{2\pi |k|}  \\
        \leq& \Big(\frac{2N}{q} + 1 \Big) | \CQ|_{1 - 1/q, |\omega - p/q|/2}  \eu^{2\pi \frac{N}{q}} \\
            \lesssim& \kappa \eu^{2\pi \kappa} q^{\nu_0} \verts{\omega - \frac{p}{q}}^3 .
        \end{aligned}
    \end{equation*}
    Combining with \eqref{eq:rnf2 iter res compare}, we obtain
    \begin{equation*}
        \begin{aligned}
            \norm{\crochet{  g_q(\cdot , 0) + \frac{1}{2} \Big(\frac{p}{q} - \omega \Big)^2 \frac{\langle \dot f_1 \rangle_q}{1 + \langle f_1 \rangle_q}}_q^{N}} \leq & \norm{\langle g_q(\cdot , 0)\rangle_q^{N} - \langle \wt g_q(\cdot , 0)\rangle_q^{N}} + \norm{\crochet{\CQ(\cdot, 0)}_q^{N}} \\
            \lesssim& \kappa \eu^{4\pi \kappa} q^{\nu_0 - \nu} \verts{\omega - \frac{p}{q}}^2 + \kappa \eu^{2\pi \kappa} q^{\nu_0} \verts{\omega - \frac{p}{q}}^3\\
            \lesssim& \kappa \eu^{4\pi \kappa} q^{\nu_0 - \nu} \verts{\omega - \frac{p}{q}}^2.
        \end{aligned}
    \end{equation*}
    We have proven \eqref{eq:rnf res compare}.
    To prove \eqref{eq:rnf S compare}, 
    we first note that, by combining the second inequality of \eqref{eq:rnf1 residue} with the estimate $\|f_1\| < 1/18$ from \eqref{eq:f1 norm}, we have $|\wt S_q|_{1 - 1/q, |\omega - p/q|/2} \lesssim \verts{\omega - p/q}^2$.
    Since $S_q = S \circ H_q = \wt S_q \circ (\id + \Sigma)$, by \eqref{eq:rnf2 iter psiphi} and Cauchy's estimate we have
    \begin{equation*}
        \begin{aligned}
            |S_q - \wt S_q|_{a,b} 
            \leq& |\sigma_2|_{a,b} \verts{\frac{\partial \wt S_q}{\partial y}}_{1-2/q, |\omega - p/q|/3} + |\sigma_1|_{a,b} \verts{\frac{\partial \wt S_q}{\partial x}}_{1-2/q, |\omega - p/q|/3} \\
            \lesssim& q |\sigma_2|_{a,b} \verts{\wt S_q}_{1-1/q, |\omega - p/q|/2} + \verts{\omega - \frac{p}{q}}\inv |\sigma_1|_{a,b}  \verts{\wt S_q}_{1-1/q, |\omega - p/q|/2} \\
            \lesssim& q^{5 + \nu_0 - \nu} \verts{\omega - \frac{p}{q}}^3.
        \end{aligned}
    \end{equation*}
    We have assumed $q$ is sufficiently large so that the map $\id + \Sigma$ sends $\T_a \times \D_b$ into $\T_{1 - 2/q} \times \D_{|\omega - p/q|/3}$.
    Since $\nu > 5$, the conclusion \eqref{eq:rnf S compare} follows from the above estimate and the second line of \eqref{eq:rnf1 residue} by a triangle inequality.

    Finally, it remains to prove \eqref{eq:gq star}.
    Informally, the claim states that $g_q(\cdot, 0)$ is dominated by its zero average (oscillating) part $g_q^\bullet$.
    This is expected from the fact that the map $F_q$ preserves an area form close to the standard form $\du x \du y$:
    if the preserved area form were exactly $\du x \du y$, then $g_q^*$ would be identically zero, whereas for an area form close to $\du x \du y$, the average part $g_q^*$ is expected to be small compared to the oscillatory part $g_q^\bullet$.

    We make this argument precise as follows.
    Observe that, if $\phi$ is real-analytic, then the map $F_q$ is real-analytic and exact symplectic with the (nonstandard) symplectic 1-form $H^*_q(y \du x)$, whose exterior derivative is an area form $H^*_q(\du x \du y)$ preserved by $F_q$.
    By \eqref{eq:rnf1 area form}, the area form $H^*_q(\du x \du y)$ writes
    \begin{equation*}
        \begin{aligned}
            & (\Theta \circ (\id + \Sigma))^* (\du x \du y) \\
            =& (\id + \Sigma)^* \Big(\Big( \frac{1}{1 + \langle f_1\rangle_q } + \wt \eta \Big) \du x \du y \Big) \\
            =& \Big(\frac{1}{1+ \langle f_1 \rangle_q \circ (\id + \Sigma)}  + \wt \eta \circ (\id + \Sigma) \Big) \Big(1 + \frac{\partial \sigma_1}{ \partial x} +  \frac{\partial \sigma_2}{ \partial y} +  \frac{\partial \sigma_1}{ \partial x} \frac{\partial \sigma_2}{ \partial y} -  \frac{\partial \sigma_1}{ \partial y} \frac{\partial \sigma_2}{ \partial x} \Big) \du x \du y
        \end{aligned}
    \end{equation*}
    where $\wt \eta$, $\sigma_1$ and $\sigma_2$ are analytic functions satisfying \eqref{eq:rnf1 area form estimate} and \eqref{eq:rnf2 iter psiphi} respectively.
    This expression holds on the domain $\T_{1 - 2/q} \times \D_{|\omega - p/q|/3}$.
    In particular, 
    \begin{equation*}
        |\wt \eta \circ (\id + \Sigma)|_{1 - 2/q, |\omega - p/q|/3} \lesssim q^{\nu_0} \verts{\omega - \frac{p}{q}}.
    \end{equation*}
    By Cauchy's estimate and \eqref{eq:rnf2 iter psiphi},
    \begin{equation*}
        \verts{\frac{\partial \sigma_1}{ \partial x} +  \frac{\partial \sigma_2}{ \partial y} +  \frac{\partial \sigma_1}{ \partial x} \frac{\partial \sigma_2}{ \partial y} -  \frac{\partial \sigma_1}{ \partial y} \frac{\partial \sigma_2}{ \partial x}}_{1 - 3/q, |\omega - p/q|/4} \lesssim q^{5 + \nu_0 - \nu} \verts{\omega - \frac{p}{q}}.
    \end{equation*}
    Since $\langle f_1 \rangle_q = f_1^* + \langle f_1^\bullet \rangle_q$, we may write 
    \begin{equation*}
        \frac{1}{1 + \langle f_1 \rangle_q \circ (\id + \Sigma)} = \frac{1}{1 + f_1^*} - \frac{\langle f_1^\bullet \rangle_q \circ (\id + \Sigma)}{(1 + f_1^*)(1 + \langle f_1 \rangle_q \circ (\id + \Sigma))}.
    \end{equation*}
    Since the smallest order of nonzero Fourier coefficients of $\langle f_1^\bullet \rangle_q$ is at least $q$, using $\|f_1\| < 1/18$ from \eqref{eq:f1 norm} and the smallness of $\sigma_1$, the quantity $\langle f_1^\bullet \rangle_q \circ (\id + \Sigma)(x,y)$ for $\Im x = 0$ can be bounded  by $\lesssim \eu^{-2\pi q}$.
    The second term on the right-hand side in the above equation can then be estimated as
    \begin{equation*}
        \verts{\frac{\langle f_1^\bullet \rangle_q \circ (\id + \Sigma)}{(1 + f_1^*)(1 + \langle f_1 \rangle_q \circ (\id + \Sigma))}}_{0, |\omega - p/q|/4} \lesssim \eu^{-2\pi q}.
    \end{equation*}
    Here and henceforth, we denote by $\verts{\; \cdot\; }_{0,b}$ the supremum norm over the domain $\T \times \D_b$.
    Combining the above estimates and using $\eu^{-2\pi q} \ll q^{-\nu}$ and $|\omega - p/q| \leq q^{-\nu}$, we find that 
    \begin{equation}\label{eq:eta estimate}
        H_q^* (\du x \du y) = \Big(\frac{1}{1 + f_1^*} + \eta\Big) \du x \du y, \qquad \text{with} \qquad |\eta|_{0, |\omega - p/q|/4} \lesssim q^{\nu_0 - \nu}.
    \end{equation}
    By construction, the function $\eta: \T_{1-3/q} \times \D_{|\omega - p/q|/4} \to \C$ depends analytically on $\phi$.

    If $\phi$ is real-analytic, then the changes of coordinates are all real-analytic and therefore $\eta$ is real-analytic.
    By the standard property of exact symplectic maps,  the \textit{signed} area of the region bounded by the curve $\T \times \{y\}$ and its image under $F_q$ is zero for a fixed \textit{real} $y$, i.e.,
    \begin{equation}\label{eq:signed area = 0}
        \int_\T \int_y^{y + g_q(x,y)} \frac{1}{1 + f_1^*} + \eta(x,y)  \du y\du x = 0.
    \end{equation}
    Since $\eta$ is analytic on $\T_{1 - 3/q} \times \D_{|\omega - p/q|/4}$ and $|g_q|_{1 - 3/q, |\omega - p/q|/4} \ll |\omega - p/q|$ by \eqref{eq:rnf concl 1}, the left-hand side of the equation above is well-defined and analytic for $y$ in the complex disc $\D_b$ with $b = |\omega - p/q|/5$.

    Fix $y \in \D_b \cap \R$. Consider now the Fourier expansion of $\phi$ of the following form
    \begin{equation*}
        \phi(x) = \sum_{k \geq 1} a_k \sin(2\pi k x) + b_k \cos(2\pi k x)
    \end{equation*}
    where $a_k, b_k \in \C$.
    If $a_k, b_k \in \R$ for all $k$, then $\phi$ is real-analytic and the equation \eqref{eq:signed area = 0} holds.
    Suppose $a_k = b_k = 0$ for all but finitely many $k$'s.
    Then we may regard the left hand side of \eqref{eq:signed area = 0} as an analytic function in the finitely many variables $a_k, b_k$.
    Since the equation \eqref{eq:signed area = 0} holds for all real values of $a_k, b_k$, it must hold for all complex values of $a_k, b_k$ by analytic continuation.
    By approximating a general $\phi$ by its Fourier truncations and using the continuous dependence of $\eta$ and $g_q$ on $\phi$, we conclude that \eqref{eq:signed area = 0} holds for all analytic $\phi: \T_1 \to \C$ such that $\|\dddot\phi\| \leq 1$.

    Applying the above argument for each fixed real $y$ and using analytic continuation again, we deduce that \eqref{eq:signed area = 0} holds for all $y \in \D_b$.

    We now use \eqref{eq:signed area = 0} to control $g_q^*(y)$.
    Fixing $|y| < b$ and rewriting the left hand side of \eqref{eq:signed area = 0}, we have
    \begin{equation*}
        \frac{g^*_q(y)}{1 + f_1^*} + \int_\T \int_0^1 g_q(x,y) \eta(x,y + t g_q(x,y))  \du t \du x = 0
    \end{equation*}
    The second term on the left hand side can be bounded as
    \begin{equation*}
        \verts{\int_\T \int_0^1 g_q(x,y) \eta(x,y + t g_q(x,y))  \du t \du x}
        \leq  |\eta|_{0, |\omega - p/q|/4} \sup_{x \in \T} |g_q(x,y)|
    \end{equation*}
    Therefore, we obtain from \eqref{eq:eta estimate} that 
    \begin{equation*}
        |g_q^*(y)| \lesssim  q^{\nu_0 - \nu}  \sup_{x \in \T} |g_q(x,y)|.
    \end{equation*}
    Note that $g_q(x,y) = g_q^*(y) + g_q^\bullet(x,y)$. 
    Since $\nu_0 < \nu$, for $q$ sufficiently the quantity $q^{\nu_0 - \nu}$ is arbitrarily small. Hence we may write 
    \begin{equation*}
        |g_q^*(y)| \lesssim  q^{\nu_0 - \nu}  \sup_{x \in \T} |g_q^\bullet(x,y)|.
    \end{equation*}
    We have obtained \eqref{eq:gq star} and hence completed the proof of the theorem.
\end{proof}

\appendix

\section{Functions defined on $\CD_\lambda$}\label{sec:tri dom}
Recall that for $\lambda > 0$ we have defined the domain $\CD_{\lambda}$ to be
\begin{equation}\label{eq:tri dom recall}
    \CD_\lambda = \{(x,y) \in \T_1 \times \C \mid |\Im x| + \lambda\inv |y| < 1\}.
\end{equation}
If $u$ is a function defined on $\CD_{\lambda}$, we denote by $|u|^{(\lambda)}$ the supremum norm of $u$ over $\CD_{\lambda}$.
If $u$ is of class $\CC^r$ on $\CD_{\lambda}$, we set
\begin{equation*}
    |u|^{(\lambda, r)} := \sup\bigg\{ \biggl\vert\frac{\partial^{j+k} u}{\partial x^j \partial y^k}(x,y)\biggr\vert \ : \ (x,y) \in \CD_\lambda,\ 0 \le j+k \le r\bigg\}.
\end{equation*}

\begin{lemma}[Cauchy’s estimate]\label{lem:tri cauchy}
Let $\lambda>0$ and $\bar u:\CD_\lambda\to\C$ be analytic, and put $u(x,y)=y\,\bar u(x,y)$.
Then, for any $r\ge1$ and any $\lambda'\in(0,\lambda)$,
\begin{equation}\label{eq:tri cauchy}
  |u|^{(\lambda', r)} \;\le\; \frac{r+\lambda'}{1-\lambda'/\lambda}\;|\bar u|^{(\lambda,\,r-1)}.
\end{equation}
In particular,
\begin{equation}\label{eq:tri cauchy ord1}
  \biggl\vert\frac{\partial u}{\partial x}\biggr\vert^{(\lambda')} \le \frac{\lambda\lambda'}{\lambda-\lambda'}\,|\bar u|^{(\lambda)},
  \qquad
  \biggl\vert\frac{\partial u}{\partial y}\biggr\vert^{(\lambda')} \le \frac{\lambda}{\lambda-\lambda'}\,|\bar u|^{(\lambda)}.
\end{equation}
\end{lemma}

\begin{proof}
Throughout the proof, let us fix $(x,y)\in\CD_{\lambda'}$.

\medskip\noindent
\emph{Step 1: first derivatives.}
By \eqref{eq:tri dom recall}, for fixed $y$ the available $x$–strip in $\CD_\lambda$ has half-width $1-|y|/\lambda$ and in $\CD_{\lambda'}$ has half-width $1-|y|/\lambda'$. Cauchy's estimate in $x$ gives
\begin{equation}\label{eq:tri cauchy ord 1}
\begin{aligned}
  \biggl\vert y\,\frac{\partial\bar u}{\partial x}(x,y)\biggr\vert
  &\le\; |y|\Bigl(\Bigl(1-\frac{|y|}{\lambda}\Bigr)-\Bigl(1-\frac{|y|}{\lambda'}\Bigr)\Bigr)^{-1}
           |\bar u|^{(\lambda)} \;\le \frac{\lambda\lambda'}{\lambda-\lambda'}\,|\bar u|^{(\lambda)}.
\end{aligned}
\end{equation}
For fixed $x$, the available $y$–radii are $\lambda(1-|\Im x|)$ in $\CD_\lambda$ and $\lambda'(1-|\Im x|)$ in $\CD_{\lambda'}$. Cauchy's estimate in $y$ yields
\begin{align*}
  \biggl\vert y\,\frac{\partial\bar u}{\partial y}(x,y)\biggr\vert
  &\le\; \frac{|y|}{\lambda(1-|\Im x|)-\lambda'(1-|\Im x|)}\,
           |\bar u|^{(\lambda)} = \frac{\lambda'}{\lambda-\lambda'}\,|\bar u|^{(\lambda)}.
\end{align*}
Since $\partial_y u=\bar u+y\,\partial_y\bar u$, we obtain
\begin{equation*}
 \biggl\vert\frac{\partial u}{\partial y}\biggr\vert^{(\lambda')}
 \le |\bar u|^{(\lambda')} + \frac{\lambda'}{\lambda-\lambda'}\,|\bar u|^{(\lambda)}
 \le \frac{\lambda}{\lambda-\lambda'}\,|\bar u|^{(\lambda)},
\end{equation*}
and together with \eqref{eq:tri cauchy ord 1} this proves \eqref{eq:tri cauchy ord1}.

\medskip\noindent
\emph{Step 2: higher derivatives.}
Let $1\le j+k\le r$. Applying the previous estimates to
$\partial_x^{\,j-1}\partial_y^{\,k}\bar u$ (if $k=0$) or to
$\partial_x^{\,j}\partial_y^{\,k-1}\bar u$ (if $k\ge1$) gives
\begin{equation*}
  \biggl\vert y\,\frac{\partial^{j+k}\bar u}{\partial x^j \partial y^k}(x,y)\biggr\vert
  \le
  \begin{cases}
    \dfrac{\lambda\lambda'}{\lambda-\lambda'}\;
    \Bigl\vert \dfrac{\partial^{j-1}\bar u}{\partial x^{j-1}}\Bigr\vert^{(\lambda)}
      & \text{if } k=0,\\
    \dfrac{\lambda'}{\lambda-\lambda'}\;
    \Bigl\vert \dfrac{\partial^{j+k-1}\bar u}{\partial x^{j}\partial y^{k-1}}\Bigr\vert^{(\lambda)}
      & \text{if } k\ge1.
  \end{cases}
\end{equation*}
By the Leibniz rule,
\begin{equation*}
  \biggl\vert \frac{\partial^{j+k}u}{\partial x^{j}\partial y^{k}}(x,y) \biggr\vert
  = \biggl\vert y\,\frac{\partial^{j+k}\bar u}{\partial x^{j}\partial y^{k}}(x,y)
      + k\,\frac{\partial^{j+k-1}\bar u}{\partial x^{j}\partial y^{k-1}}(x,y) \biggr\vert
  \le \max\!\Bigl\{\frac{\lambda\lambda'}{\lambda-\lambda'},\; \frac{\lambda'}{\lambda-\lambda'}+r\Bigr\}
     \;|\bar u|^{(\lambda,\,r-1)}.
\end{equation*}
Since $\frac{\lambda\lambda'}{\lambda-\lambda'}=\frac{\lambda'}{1-\lambda'/\lambda}$ and
$\frac{\lambda'}{\lambda-\lambda'}+r < \frac{1}{1 - \lambda'/\lambda} + r < \frac{r}{1-\lambda'/\lambda}$, we get
\begin{equation*}
  \sup_{1\le j+k\le r}\biggl\vert \frac{\partial^{j+k}u}{\partial x^{j}\partial y^{k}}\biggr\vert^{(\lambda')}
  < \frac{r+\lambda'}{1-\frac{\lambda'}{\lambda}}\;|\bar u|^{(\lambda,\,r-1)}.
\end{equation*}

\medskip\noindent
\emph{Step 3: zeroth order.}
For $j=k=0$, using $|y|<\lambda'$ on $\CD_{\lambda'}$,
\begin{equation*}
  |u(x,y)| \le |y|\,|\bar u(x,y)|
  \le \lambda'\,|\bar u|^{(\lambda')}
  \le \frac{r+\lambda'}{1-\frac{\lambda'}{\lambda}}\;|\bar u|^{(\lambda,\,r-1)}.
\end{equation*}
Taking the supremum over $(x,y)\in\CD_{\lambda'}$ and $0\le j+k\le r$ yields \eqref{eq:tri cauchy}.
\end{proof}

\begin{proposition}[Composition]\label{prop:tri comp}
    Let $\lambda > \lambda' > 0$ and suppose $\bar \sigma_j : \CD_{\lambda'} \to \C$ is analytic for $j=1,2$.
    Define $\sigma_j(x,y) := y\,\bar \sigma_j(x,y)$ for $j=1,2$, set $\Sigma(x,y) := (\sigma_1(x,y), \sigma_2(x,y))$, and consider the map $\id + \Sigma: \CD_{\lambda'} \to \C^2$.
    Then:
    \begin{enumerate}
        \item If 
        \begin{equation}\label{eq:tri o cond}
            |\bar \sigma_1|^{(\lambda')} + \frac{1+|\bar \sigma_2|^{(\lambda')}}{\lambda} \;\le\; \frac{1}{\lambda'},    
        \end{equation}
        then $\id + \Sigma$ maps $\CD_{\lambda'}$ into $\CD_{\lambda}$.
        
        \item For all $r \ge 0$, there exists a combinatorial constant $\tilde C_r>0$ depending only on $r$ such that,
        whenever the composition $v \circ (\id + \Sigma)$ is well-defined on $\CD_{\lambda'}$ for an analytic $v: \CD_\lambda \to \C$,
        \begin{equation}\label{eq:tri comp norm}
            |v \circ (\id + \Sigma)|^{(\lambda', r)} \;\le\; \tilde C_r\, |v|^{(\lambda, r)}\,
            \bigl(1 + |\sigma_1|^{(\lambda', r)} + |\sigma_2|^{(\lambda', r)}\bigr)^{r}.
        \end{equation}
        
        \item Let $v_j: \CD_\lambda \to \C$ be analytic for $j=1,2$ such that $v_j \circ (\id + \Sigma)$ is well-defined on $\CD_{\lambda'}$,
        and suppose $\bar \sigma_j$ satisfies
        \begin{equation}\label{eq:sig-v implicit}
            \bar \sigma_j \;=\; v_j \circ (\id + \Sigma) \qquad (j=1,2).
        \end{equation}
        Then, for any $r \in \Z_{\ge 0}$, there exist constants $A_r,B_r\in\Z_{\ge0}$ and $C_r>0$ depending only on $r$ such that, for all $\lambda'' \in (0,\lambda')$,
        \begin{equation}\label{eq:tri bootstrap}
            |\bar \sigma_j|^{(\lambda'', r)} \;\le\;
            C_r \Bigl( \frac{r + \lambda'}{1 - \beta}\Bigr)^{A_r}
            \Bigl(1 + |v_1|^{(\lambda, r-1)} + |v_2|^{(\lambda, r-1)}\Bigr)^{B_r}
            |v_j|^{(\lambda, r)},
            \qquad \beta := \Bigl( \frac{\lambda''}{\lambda'}\Bigr)^{\!1/r},
        \end{equation}
        with the convention that $|v_j|^{(\lambda,-1)}:=0$.
    \end{enumerate}
\end{proposition}

\begin{proof}
\emph{(1)} Let $(x,y)\in\CD_{\lambda'}$, so $|\Im x| < 1 - |y|/\lambda'$.
Using $|\Im(x+\sigma_1)| \le |\Im x| + |\sigma_1|$ and $|\sigma_j|\le |y|\,|\bar\sigma_j|^{(\lambda')}$,
\begin{equation*}
    \begin{aligned}
          |\Im(x + \sigma_1(x,y))| + \frac{|y + \sigma_2(x,y)|}{\lambda}
            \;&\le\; |\Im x| + |y| \Bigl(|\bar\sigma_1|^{(\lambda')} + \frac{1 + |\bar\sigma_2|^{(\lambda')}}{\lambda}\Bigr)\\
  \;&<\; 1 - \frac{|y|}{\lambda'} + |y|\Bigl(|\bar\sigma_1|^{(\lambda')} + \frac{1 + |\bar\sigma_2|^{(\lambda')}}{\lambda}\Bigr).
    \end{aligned}
\end{equation*}
If \eqref{eq:tri o cond} holds, then the term between the parentheses is $\le \frac1{\lambda'}$, and the right-hand side is $<1$.
Hence $(x+\sigma_1(x,y),y+\sigma_2(x,y))\in\CD_\lambda$.

\smallskip
\emph{(2)} Differentiate $v(x+\sigma_1(x,y),\,y+\sigma_2(x,y))$ up to total order $r$ in $(x,y)$.
By the multivariate chain rule (Faà di Bruno), each derivative is a finite linear combination (with coefficients depending only on $r$) of terms
\begin{equation*}
  \frac{\partial^{\alpha+\beta} v}{\partial x^\alpha \partial y^\beta}\bigl(\varsigma_1(x,y), \varsigma_2(x,y)\bigr)
  \prod_{\ell=1}^{\alpha}\frac{\partial^{j_\ell+k_\ell}\varsigma_1}{\partial x^{j_\ell}\partial y^{k_\ell}}(x,y)\;
  \prod_{\ell=1}^{\beta}\frac{\partial^{j'_\ell+k'_\ell}\varsigma_2}{\partial x^{j'_\ell}\partial y^{k'_\ell}}(x,y),
\end{equation*}
where $\varsigma_j(x,y) =x+\sigma_j(x,y)$, and
$\sum_{\ell = 1}^{\alpha}(j_\ell+k_\ell)+\sum_{\ell = 1}^{\beta}(j'_\ell+k'_\ell)\le r$.
Since the composition is well-defined, we have $(\varsigma_1(x,y),\varsigma_2(x,y))\in \CD_\lambda$ and hence
$\bigl|\partial^{\alpha+\beta} v\bigr|^{(\lambda)}\le |v|^{(\lambda,r)}$.
On the other hand, each derivative of $\varsigma_j$ is bounded by $1+|\sigma_j|^{(\lambda',r)}$.
Taking suprema and absorbing the number of terms into $\tilde C_r$ yields \eqref{eq:tri comp norm}.

\smallskip
\emph{(3)} We argue by induction on $r\ge0$.

\underline{Base case $r=0$.} From \eqref{eq:sig-v implicit} we have $|\bar\sigma_j|^{(\lambda'')} \le |v_j|^{(\lambda)}$, so \eqref{eq:tri bootstrap} holds with $C_0=1$ and $A_0=B_0=0$.

\underline{Induction step.} Fix $r\ge1$ and $\lambda''\in(0,\lambda')$, and set $\beta=(\lambda''/\lambda')^{1/r}\in(0,1)$.
Applying \eqref{eq:tri comp norm} with $v=v_j$ and using Lemma~\ref{lem:tri cauchy} to estimate $\sigma_j=y\bar\sigma_j$ on $\CD_{\lambda''}$ from $\bar\sigma_j$ on the larger domain $\CD_{\lambda''/\beta}\subset\CD_{\lambda'}$, we get
\begin{equation*}
  |\bar\sigma_j|^{(\lambda'',r)}
  \;\le\; \tilde C_r\,|v_j|^{(\lambda,r)}\,
          \Bigl(1 + |\sigma_1|^{(\lambda'',r)} + |\sigma_2|^{(\lambda'',r)}\Bigr)^r
  \;\le\; \tilde C_r\,|v_j|^{(\lambda,r)}\,
          \Bigl(\frac{r+\lambda''}{1-\beta}\Bigr)^{\!r}\,
          \Bigl(1 + |\bar\sigma_1|^{(\lambda''/\beta,\,r-1)} + |\bar\sigma_2|^{(\lambda''/\beta,\,r-1)}\Bigr)^{\!r}.
\end{equation*}
Note that $\lambda''<\lambda''/\beta\le\lambda'$ and
$((\lambda''/\beta)/\lambda')^{1/(r-1)}=\beta$.
Thus the induction hypothesis at level $r-1$ (applied with $\lambda''$ replaced by $\lambda''/\beta$) yields
\begin{equation*}
  1 + |\bar\sigma_1|^{(\lambda''/\beta,\,r-1)} + |\bar\sigma_2|^{(\lambda''/\beta,\,r-1)}
  \;\le\; C_{r-1}\,\Bigl(\frac{r+\lambda''/\beta}{1-\beta}\Bigr)^{\!A_{r-1}}
          \Bigl(1 + |v_1|^{(\lambda, r-1)} + |v_2|^{(\lambda, r-1)}\Bigr)^{B_{r-1}+1}.
\end{equation*}
Combining and using $\lambda''/\beta < \lambda'$ gives
\begin{equation*}
  |\bar\sigma_j|^{(\lambda'',r)}
  \;\le\; \tilde C_r\,C_{r-1}^{\,r}\,
          \Bigl(\frac{r+\lambda'}{1-\beta}\Bigr)^{\,r+rA_{r-1}}\,
          \Bigl(1 + |v_1|^{(\lambda, r-1)} + |v_2|^{(\lambda, r-1)}\Bigr)^{\,r(B_{r-1}+1)}
          |v_j|^{(\lambda,r)}.
\end{equation*}
Thus \eqref{eq:tri bootstrap} holds with the recursive choices
\begin{equation*}
  A_r := r + rA_{r-1}, \qquad B_r := r(B_{r-1}+1), \qquad C_r := \tilde C_r\, C_{r-1}^{\,r},
\end{equation*}
and the stated base values.
\end{proof}

\begin{remark}[Explicit choice of $\tilde C_r$ in \eqref{eq:tri comp norm}]
One may take
\begin{equation*}
  \tilde C_0 := 1, \qquad \tilde C_r := r!\,2^{\,r}\quad (r\ge 1).
\end{equation*}
This follows from the multivariate Faà di Bruno expansion for $v\circ(\id+\Sigma)$:
an $r$-th $(x,y)$–derivative is a sum of at most $r!\,2^{r}$ terms, each bounded by
$|v|^{(\lambda,r)}(1+|\sigma_1|^{(\lambda',r)}+|\sigma_2|^{(\lambda',r)})^{r}$.
\end{remark}

\begin{proposition}[Inversion]\label{prop:tri inv}
    Let $\bar \xi_1, \bar \xi_2: \CD_\lambda \to \C$ be analytic and define $\xi_j(x,y) := y\,\bar \xi_j(x,y)$ and $\Xi(x,y) := (\xi_1(x,y), \xi_2(x,y))$.
    Suppose 
    \begin{equation}\label{eq:tri inv}
        0 < \lambda' < \frac{\bigl(1 - |\bar \xi_2|^{(\lambda)}\bigr)^2}{\;\frac{1}{\lambda} + |\bar \xi_1|^{(\lambda)}\bigl(2 - |\bar \xi_2|^{(\lambda)}\bigr)}.
    \end{equation}
    Then there exist analytic functions $\bar \sigma_1, \bar \sigma_2: \CD_{\lambda'} \to \C$ such that, with $\sigma_j(x,y) := y\,\bar \sigma_j(x,y)$ and $\Sigma(x,y) := (\sigma_1(x,y), \sigma_2(x,y))$, we have 
    \begin{equation}\label{eq:tri right inv}
        (\id + \Xi) \circ (\id + \Sigma) = \id \quad \text{on } \CD_{\lambda'}
    \end{equation}
    and, for any $\lambda'' \in (0, \lambda')$ such that $(\id + \Xi) (\CD_{\lambda''}) \subset \CD_{\lambda'}$,
    \begin{equation}\label{eq:tri left inv}
        (\id + \Sigma) \circ (\id + \Xi) = \id \quad \text{on } \CD_{\lambda'}.
    \end{equation}
    Moreover, for any $r \in \Z_{\ge 0}$ there exist constants $A_r,B_r\in\Z_{\ge0}$ and $C_r>0$ (depending only on $r$) such that, for any $\lambda'' \in (0,\lambda')$,
    \begin{equation}\label{eq:bar xi vs sig}
        1 + |\bar \sigma_1|^{(\lambda'', r)} + |\bar \sigma_2|^{(\lambda'', r)}
        \;\le\;  C_r \Bigl( \frac{r + \lambda'}{1 - \beta}\Bigr)^{A_r}
        \Biggl(\frac{1 + |\bar \xi_1|^{(\lambda, r-1)} + |\bar \xi_2|^{(\lambda, r-1)}}{1 - |\bar \xi_2|^{(\lambda)}}\Biggr)^{B_r}
        \Bigl(1 + |\bar \xi_1|^{(\lambda, r)} + |\bar \xi_2|^{(\lambda, r)}\Bigr),
    \end{equation}
    where $\beta := (\lambda''/\lambda')^{1/r}$.
\end{proposition}

\begin{proof}
Set $a:=|\bar \xi_1|^{(\lambda)}$ and $b:=|\bar \xi_2|^{(\lambda)}$ (so $b<1$ is implicit in \eqref{eq:tri inv}). From \eqref{eq:tri inv} one checks there exists $\lambda_*>0$ such that
\begin{equation}\label{eq:tri o lam'}
    \lambda' \;\le\; \frac{1-b}{\frac{1}{\lambda_*}+a}
    \qquad\text{and}\qquad
    \lambda_* \;<\; \frac{1-b}{\frac{1}{\lambda}+a}.
\end{equation}
(Indeed, the first inequality is equivalent to $\lambda_* \ge ((1-b)/\lambda' - a)^{-1}$, and the second one gives an upper bound; \eqref{eq:tri inv} is precisely the condition that the interval is nonempty.)

Equip $\C^2$ with the anisotropic norm $\|(x,y)\|^{(\lambda_*)}:=\lambda_*|x|+|y|$. For $(x,y),(x',y')\in\CD_{\lambda_*}$, Lemma~\ref{lem:tri cauchy} (with $(\lambda',\lambda)\leftarrow (\lambda_*,\lambda)$) and the mean value theorem give, for $j=1,2$,
\begin{equation*}
\begin{aligned}
    |\xi_j(x,y) - \xi_j(x',y')| 
    \leq& |\xi_j(x,y) - \xi_j(x,y')| + |\xi_j(x,y') - \xi_j(x',y')| \\
    \leq& |y - y'|\verts{\frac{\partial \xi_j}{\partial y}}^{(\lambda_*)}  + |x - x'|\verts{\frac{\partial \xi_j}{\partial x}}^{(\lambda_*)} \\
    \leq& |y - y'|\frac{\lambda}{\lambda - \lambda_*} |\bar \xi_j|^{(\lambda)}  + |x - x'|\frac{\lambda \lambda_*}{\lambda - \lambda_*} |\bar \xi_j|^{(\lambda)} \\
    =& \frac{\lambda}{\lambda - \lambda_*} |\bar \xi_j|^{(\lambda)}  \|(x - x', y - y')\|^{(\lambda_*)}.
\end{aligned}
\end{equation*}
Hence
\begin{equation*}
    \begin{aligned}
        \|\Xi(x,y) - \Xi(x',y')\|^{(\lambda_*)} 
        =& \lambda_* |\xi_1(x,y) - \xi_1(x',y')|  +  |\xi_2(x,y) - \xi_2(x',y')| \\
        \leq & \frac{\lambda}{\lambda-\lambda_*}\,\bigl(\lambda_* a + b\bigr)\,\|(x-x',y-y')\|^{(\lambda_*)}
    \end{aligned}
    \end{equation*}
By the second inequality in \eqref{eq:tri o lam'}, $\lambda_* a + b < 1 - \lambda_*/\lambda$, so the last line is $<\|(x-x',y-y')\|^{(\lambda_*)}$. Thus $\Xi$ is a contraction on $\CD_{\lambda_*}$ for the distance induced by $\|\cdot\|^{(\lambda_*)}$ and, in particular, $\id+\Xi$ is injective there. 
 Moreover, since the determinant of the Jacobian matrix of $\id + \Xi$ at $(x,y)$ writes 
    \begin{equation*}
        \det (d_{(x,y)} (\id + \Xi)) = \Big(1 + \frac{\partial \xi_1}{\partial x}(x,y) \Big)\Big(1 + \frac{\partial \xi_2}{\partial y} (x,y)\Big) - \frac{\partial \xi_1}{\partial y}(x,y) \frac{\partial \xi_2}{\partial x}(x,y),
    \end{equation*}
    by \eqref{eq:tri cauchy ord1} again,
    we have for $(x,y) \in \CD_{ \lambda_*}$ that 
    \begin{equation*}
        \begin{aligned}
            |\det (d_{(x,y)} (\id + \Xi))| 
            \geq& \Big(1 - \verts{\frac{\partial \xi_1}{\partial x} }_{ \lambda_*}\Big)\Big(1 - \verts{\frac{\partial \xi_2}{\partial y}}_{ \lambda_*} \Big) - \verts{\frac{\partial \xi_1}{\partial y} \frac{\partial \xi_2}{\partial x}}_{ \lambda_*}\\
            \geq& \Big(1 -  \frac{\lambda \lambda_*}{\lambda - \lambda_*} |\bar \xi_1|_{ \lambda}\Big)\Big(1 - \frac{\lambda}{\lambda - \lambda_*} |\bar \xi_2|_{ \lambda}\Big) - \frac{\lambda}{\lambda - \lambda_*} |\bar \xi_1|_{ \lambda} \frac{\lambda \lambda_*}{\lambda - \lambda_*} |\bar \xi_2|_{ \lambda}\\
            =& 1 - \frac{\lambda}{\lambda - \lambda_*} (\lambda_* |\bar \xi_1|_{ \lambda} + |\bar \xi_2|_{ \lambda}) > 0.
        \end{aligned}
    \end{equation*}
    Hence, $\id + \Xi$ is a diffeomorphism from $\CD_{ \lambda_*}$ onto its image.

Next, we show $\CD_{\lambda'} \subset (\id+\Xi)(\CD_{\lambda_*})$. It suffices to check that $(\id+\Xi)$ maps the boundary $\partial\CD_{\lambda_*}$ outside the interior of $\CD_{\lambda'}$ and then appeal to degree/invariance-of-domain for the injective homotopy $\id+t\Xi$, $t\in[0,1]$. Let $(x,y)\in\partial\CD_{\lambda_*}$, so $|\Im x|+|y|/\lambda_*=1$. Using $|\xi_1|\le |y|\,a$ and $|\xi_2|\le |y|\,b$, we estimate
\begin{equation*}
\begin{aligned}
  |\Im(x+\xi_1)| + \frac{|y+\xi_2|}{\lambda'}
  &\ge |\Im x| + \frac{|y|}{\lambda'} - |y|\,a - \frac{|y|}{\lambda'}\,b \\
  &\ge |\Im x| + \frac{|y|}{\lambda_*}\,\Bigl(\frac{\lambda_*}{\lambda'} - \lambda_* a - \frac{\lambda_*}{\lambda'}\,b\Bigr).
\end{aligned}
\end{equation*}
Writing $s:=|\Im x|$ and $t:=|y|/\lambda_*$ (so $s+t=1$) gives
\begin{equation*}
  s + (1-s)\,\lambda_*\Bigl(\frac{1}{\lambda'} - a - \frac{b}{\lambda'}\Bigr)
  \;\ge\; \min\!\Bigl\{1,\, \lambda_*\Bigl(\frac{1}{\lambda'} - a - \frac{b}{\lambda'}\Bigr)\Bigr\}.
\end{equation*}
By the first inequality in \eqref{eq:tri o lam'}, $\lambda_*\bigl(\frac{1}{\lambda'} - a - \frac{b}{\lambda'}\bigr)\ge 1$, hence the minimum is $1$. Thus $(\id+\Xi)(\partial\CD_{\lambda_*})$ lies outside the interior of $\CD_{\lambda'}$, and consequently $\CD_{\lambda'}\subset (\id+\Xi)(\CD_{\lambda_*})$. 
Since $\id+\Xi$ is injective on $\CD_{\lambda_*}$, we can define $\id+\Sigma: \CD_{\lambda'} \to \CD_{\lambda_*}$ satisfying \eqref{eq:tri right inv}.
By construction, $\Sigma=(\sigma_1,\sigma_2)$ satisfies $\sigma_j(x,0)=0$, so there are analytic $\bar \sigma_j$ with $\sigma_j(x,y)=y\,\bar \sigma_j(x,y)$.

Suppose $0 < \lambda'' < \lambda'$ such that $(\id + \Xi)(\CD_{\lambda''}) \subset \CD_{\lambda'}$.
Then restricting \eqref{eq:tri right inv} to $(\id + \Xi)(\CD_{\lambda''})$ yields $(\id + \Xi) \circ (\id + \Sigma) \circ (\id + \Xi) = (\id + \Xi)$ on $\CD_{\lambda''}$.
Since $(\id + \Sigma) \circ (\id + \Xi)(\CD_{\lambda''}) \subset \CD_{\lambda_*}$ by construction of $\Sigma$ and $(\id + \Xi)$ injective on $\CD_{\lambda_*}$, we obtain \eqref{eq:tri left inv}.

It remains to prove \eqref{eq:bar xi vs sig}. From $(\id+\Xi)\circ(\id+\Sigma)=\id$ we get $\Sigma + \Xi\circ(\id+\Sigma)=0$, which yields
\begin{equation*}
  \bar \sigma_j \;=\; -\,\frac{\bar \xi_j\circ(\id+\Sigma)}{1+\bar \xi_2\circ(\id+\Sigma)}\qquad(j=1,2).
\end{equation*}
Define $v_j(x,y):= -\bar \xi_j(x,y)/(1+\bar \xi_2(x,y))$. By repeated differentiation (quotient rule) one obtains, for each $r\ge0$,
\begin{equation*}
  |v_1|^{(\lambda,r)} + |v_2|^{(\lambda,r)}
  \;\le\; D_r\,
  \Biggl(\frac{1 + |\bar \xi_1|^{(\lambda,r-1)} + |\bar \xi_2|^{(\lambda,r-1)}}{1-|\bar \xi_2|^{(\lambda)}}\Biggr)^{r+1}
  \Bigl(|\bar \xi_1|^{(\lambda,r)} + |\bar \xi_2|^{(\lambda,r)}\Bigr)
\end{equation*}
for some constant $D_r>0$ depending only on $r$. Since $\bar \sigma_j = v_j\circ(\id+\Sigma)$, Proposition~\ref{prop:tri comp}(3) gives, for $\lambda''\in(0,\lambda')$ and $\beta=(\lambda''/\lambda')^{1/r}$,
\begin{equation*}
  1 + |\bar \sigma_1|^{(\lambda'',r)} + |\bar \sigma_2|^{(\lambda'',r)}
  \;\le\; C_r \Bigl(\frac{r+\lambda'}{1-\beta}\Bigr)^{A_r}
  \Bigl(1 + |v_1|^{(\lambda,r-1)} + |v_2|^{(\lambda,r-1)}\Bigr)^{B_r}
  \Bigl(1 + |v_1|^{(\lambda,r)} + |v_2|^{(\lambda,r)}\Bigr).
\end{equation*}
Substituting the bounds for $v_j$ and enlarging constants completes the proof of \eqref{eq:bar xi vs sig}.
\end{proof}

\section{A characterization of periodic orbits.}\label{sec:near int maps}

Let $\omega \in \R$ and $a,b>0$. 
Consider an analytic map
\begin{equation}\label{eq:near integrable maps def}
    F:\binom{x}{y}\mapsto 
    \begin{pmatrix}
        x + \omega + y + f(x,y)\\
        y + g(x,y)
    \end{pmatrix},
    \qquad f,g:\T_a\times \D_b\to\C.
\end{equation}

For iterates we write
\begin{equation}\label{eq:Fj-def}
  F^j(x,y) \;=\; \begin{pmatrix}
        x + j\omega + j y + f_j(x,y)\\
        y + g_j(x,y)
  \end{pmatrix},
\end{equation}
whenever the right-hand side is defined. From \eqref{eq:near integrable maps def},
\begin{equation}\label{eq:j comp ind}
    f_0=g_0=0,\quad f_1=f,\quad g_1=g,\qquad
    \begin{cases}
        f_{j+1} = f_j + g_j + f\circ F^j,\\
        g_{j+1} = g_j + g\circ F^j.
    \end{cases}
\end{equation}
The next lemma gives a domain of definition and sup-norm bounds for $f_j,g_j$.

\begin{lemma}\label{lem:composability}
Let $f,g$ be analytic on $\T_a\times \D_b$ and let $f_j,g_j$ be as in \eqref{eq:Fj-def}. 
Fix $a'<a$, $b'<b$ and $q \in \Z_{>0}$. If
\begin{equation}\label{eq:dom def composed residue conditions}
  \begin{aligned}
    a' \;+\; q\,b' \;+\; q\,|f|_{a,b} \;+\; \frac{q(q-1)}{2}\,|g|_{a,b} \;&\le a,\\
    b' \;+\; q\,|g|_{a,b} \;&\le b,
  \end{aligned}
\end{equation}
then, for $0 \le j \le q$, the functions $f_j,g_j$ are well-defined on $\T_{a'}\times \D_{b'}$ and
\begin{equation}\label{eq:composed residue norms}
  |g_j|_{a',b'} \le j\,|g|_{a,b},\qquad
  |f_j|_{a',b'} \le j\,|f|_{a,b} + \frac{j(j-1)}{2}\,|g|_{a,b}.
\end{equation}
\end{lemma}

\begin{proof}
Induction on $j$. The case $j=0$ is trivial since $F^0=\id$.  
Assume $0\le j<q$ and $(x_j,y_j):=F^j(x,y)$ is defined for every $(x,y) \in \T_{a'}\times\D_{b'}$ with \eqref{eq:composed residue norms}. 
Then
\begin{equation*}
  |y_j| \le b' + j\,|g|_{a,b},\qquad
  |\Im x_j|
  \le a' + j\,b' + j\,|f|_{a,b} + \frac{j(j-1)}{2}\,|g|_{a,b}.
\end{equation*}
By \eqref{eq:dom def composed residue conditions}, $|y_j|<b$ and $|\Im x_j|<a$, so the compositions in \eqref{eq:j comp ind} are well-defined; hence $f_{j+1}$ and $g_{j+1}$ are well-defined there. 
Taking suprema in \eqref{eq:j comp ind} and using $|f\circ F^j|_{a',b'}\le |f|_{a,b},$ and $|g\circ F^j|_{a',b'}\le |g|_{a,b}$, we obtain 
\begin{gather*} 
    |g_{j+1}|_{a',b'} \le |g_j|_{a',b'} + |g|_{a,b} \le (j+1)\,|g|_{a,b},\\ 
    |f_{j+1}|_{a',b'} \le |f_j|_{a',b'} + |g_j|_{a',b'} + |f|_{a,b} \le (j+1)\,|f|_{a,b} + \frac{j(j+1)}{2}\,|g|_{a,b}, 
\end{gather*} 
which is exactly \eqref{eq:composed residue norms} with $j$ replaced by $j+1$. This completes the induction
\end{proof}
From now on, set $$\omega = \frac{p}{q}$$ for some coprime $(p, q)$ with $q>0$, and assume $|f|_{a,b}$, $|g|_{a,b}$ are small enough that Lemma~\ref{lem:composability} applies for some $a'<a$, $b'<b$.  
In particular, $f_q$ and $g_q$ are defined on $\T_{a'}\times\D_{b'}$ via \eqref{eq:j comp ind}.
By \eqref{eq:Fj-def}, a point $(x,y)$ belongs to a $(p,q)$–periodic orbit of $F$ iff
\begin{equation*}
    \begin{cases}
        qy + f_q(x,y) = 0,\\
        g_q(x,y) = 0
    \end{cases}
\end{equation*}
for some lift of $f_q$ to the universal cover.
To solve the first equation we look for $\gamma:\T_{a'}\to\C$ such that 
\begin{equation*}
    \gamma(x) = -\frac{1}{q}\, f_q(x, \gamma(x)).
\end{equation*}
Then it suffices to find $x_0$ with
\begin{equation*}
    g_q(x_0,\gamma(x_0)) = 0.
\end{equation*}
The point $(x_0, \gamma(x_0))$ will be a $(p,q)$-orbit under $F$.
Let us formalize this argument  with quantitative details in the next lemma.

\begin{lemma}\label{lem:gam}
Consider a map $F$ of the form \eqref{eq:near integrable maps def} with $\omega = p/q$ and suppose, for some $a' < a$, that
\begin{equation}\label{eq:existence of gamma cond}
   b + |f|_{a,b} \le \frac{a- a'}{q}
   \quad \text{and}\quad 
   |f|_{a,b} + \frac{3q-1}{2}\,|g|_{a,b} < b.
\end{equation}
Put $b' := b - q|g|_{a,b}$ and $b'' = |f|_{a,b} + \frac{q-1}{2}|g|_{a,b}$.
Then all $(p,q)$-periodic orbits of $F|_{\T_{a'} \times \D_{b'}}$ lie in the strip $\T_{a'} \times \D_{b''}$.
More precisely, there exist analytic functions $\gamma,\Gamma:\T_{a'}\to\C$ depending analytically on $f$ and $g$ with
\begin{equation}\label{eq:gam sup}
    |\gamma|_{a'} \le  |f|_{a,b} + \frac{q-1}{2}|g|_{a,b},
    \qquad 
    |\Gamma|_{a'} \le q\,|g|_{a,b}
\end{equation}
such that the following holds.
For $(x,y)\in\T_{a'}\times\D_{b'}$, the point $(x,y)$ is a $(p,q)$–periodic orbit of $F$ iff
\begin{equation*}
  y=\gamma(x)\quad\text{and}\quad \Gamma(x)=0.
\end{equation*}
Moreover, with $\bar a = 2 q|f|_{a,b} + q(q-1)|g|_{a,b}$ and $\bar b = |f|_{a,b} + (3q-1)|g|_{a,b}/2$, we have 
\begin{equation}\label{eq:Gam approx}
    \verts{\Gamma - \sum_{k = 0}^{q-1} g\Big( \cdot + \frac{k}{q}, 0 \Big)}_{a'} \leq q \Big(\frac{\bar a}{a - a' - \bar a} + \frac{\bar b}{b - \bar b} \Big) |g|_{a,b} 
\end{equation}
and
\begin{equation}\label{eq:lyap bnd}
  \Bigl|\det\!\bigl(d_{(x,\gamma(x))} F^q - \id \bigr) + q\,\dot\Gamma(x)\Bigr|
  \leq \frac{\bar a}{2} \frac{|\dot \Gamma(x)|}{b' - |\gamma(x)|}.
\end{equation}
\end{lemma}

\begin{proof}
Define $b' = b - q|g|_{a,b}$ as in the statement. 
The second inequality of condition \eqref{eq:existence of gamma cond} imply $q|g|_{a,b}<b$, hence $b'>0$. Using the first inequality in \eqref{eq:existence of gamma cond} and definition of $b'$,
\begin{equation*}
a' + qb' + q|f|_{a,b} + \frac{q(q-1)}{2}|g|_{a,b} 
\leq a' + qb + q|f|_{a,b} \leq a,
\end{equation*}
so Lemma~\ref{lem:composability} applies, and $f_q,g_q$ are defined on $\T_{a'}\times\D_{b'}$ with
\begin{equation}\label{eq:fgq bounds}
    |f_q|_{a',b'} \le q|f|_{a,b} + \frac{q(q-1)}{2}|g|_{a,b},
    \qquad
    |g_q|_{a',b'} \le q|g|_{a,b}.
\end{equation}
By the second inequality in \eqref{eq:existence of gamma cond},
\begin{equation*}
 |f|_{a,b} + \frac{q-1}{2}|g|_{a,b}
 \;<\; b - q|g|_{a,b} \;=\; b',
\end{equation*}
so for each fixed $x\in\T_{a'}$, on $|z|=b'$ we have
\begin{equation*}
  \Bigl|\frac{1}{q}f_q(x,z)\Bigr|
  \;\le\; \frac{1}{q}|f_q|_{a',b'}
  \;<\; b' \;=\; |z|.
\end{equation*}
By Rouch\'e’s theorem, $z\mapsto z+\frac{1}{q}f_q(x,z)$ has exactly one zero in $\D_{b'}$, denoted $\gamma(x)$, depending analytically on $x$.  
The bound \eqref{eq:gam sup} on $|\gamma|_{a'}$ follows from $\gamma(x) = -f_q(x,\gamma(x))/q$ and  \eqref{eq:fgq bounds}.  
Define $\Gamma(x):=g_q(x,\gamma(x))$, so $|\Gamma|_{a'}\le |g_q|_{a',b'}\le q|g|_{a,b}$.

By construction, the lift of $F$ and $(x, \gamma(x))$ into the universal cover satisfies
\begin{equation}\label{eq:gam in F}
  F^q(x,\gamma(x)) = \bigl(x+p,\,\gamma(x)+\Gamma(x)\bigr),
\end{equation}
and uniqueness of the zero implies the first coordinate of $F^q(x,y')$ equals $x+p$ iff $y'=\gamma(x)$; thus $(x,y)$ is a $(p,q)$-orbit under $F$ iff $y=\gamma(x)$ and $\Gamma(x)=0$.

For \eqref{eq:Gam approx}, we have by \eqref{eq:j comp ind} that 
\begin{equation*}
    \Gamma(x) = g_q(x, \gamma(x)) = \sum_{j = 0}^{q-1} g \circ F^j(x, \gamma(x))
\end{equation*}
where, by \eqref{eq:Fj-def} with $\omega = p/q$,
\begin{equation*}
    g \circ F^j(x, \gamma(x)) = g(x + jp/q + j\gamma(x) + f_j(x,\gamma(x)), \gamma(x) + g_j(x,\gamma(x))).
\end{equation*}
Since $j \leq q$, the estimates \eqref{eq:gam sup} and \eqref{eq:composed residue norms} imply that
\begin{equation*}
    |j\gamma(x) + f_j(x,\gamma(x))| \leq \bar a; \qquad |\gamma(x) + g_j(x,\gamma(x))| \leq \bar b.
\end{equation*}
Hence, by the mean value theorem and Cauchy's estimate,
\begin{equation*}
    \begin{aligned}
        |g \circ F^j(x, \gamma(x)) - g(x + jp/q, 0)| 
        \leq  \verts{\frac{\partial g}{\partial x}}_{a' + \bar a, \bar b} \bar a + \verts{\frac{\partial g}{\partial y}}_{a' + \bar a, \bar b} \bar b
        \leq   \Big(\frac{\bar a}{a - a' - \bar a} + \frac{\bar b}{b - \bar b} \Big) |g|_{a,b}.
    \end{aligned}
\end{equation*}
Summing over $j$, we obtain \eqref{eq:Gam approx}.

For \eqref{eq:lyap bnd}, differentiate \eqref{eq:gam in F} to get
\begin{equation*}
  \frac{d}{dt}F^q(x+t,\gamma(x+t))\big|_{t=0}=(1,\dot\gamma(x)+\dot\Gamma(x)).
\end{equation*}
Also,
\begin{equation*} 
    \frac{\partial F^q}{\partial y}(x,\gamma(x)) = \Big(q + \frac{\partial f_q}{\partial y}(x, \gamma(x)), 1 + \frac{\partial g_q}{\partial y}(x, \gamma(x))\Big) 
\end{equation*}
In the basis $\{(1,\dot\gamma(x)),(0,1)\}$ the Jacobian is
\begin{equation*}
    d_{(x, \gamma(x))}F^q = \begin{pmatrix} 1 & q + \frac{\partial f_q}{\partial y}(x, \gamma(x)) \\ \dot\Gamma(x) & 1 + \frac{\partial g_q}{\partial y}(x, \gamma(x)) - \dot\gamma(x) \left( q + \frac{\partial f_q}{\partial y}(x, \gamma(x))\right) \end{pmatrix}.
\end{equation*}
A direct computation yields
\begin{equation*}
  \det\!\bigl(d_{(x,\gamma(x))}F^q-\id\bigr)+q\,\dot\Gamma(x)
  \;=\; -\,\dot\Gamma(x)\,\partial_y f_q(x,\gamma(x)).
\end{equation*}
Cauchy’s estimate and \eqref{eq:fgq bounds} give
\begin{equation*}
  |\partial_y f_q(x,\gamma(x))|
  \;\le\; \frac{|f_q|_{a',b'}}{\,b'-|\gamma(x)|\,}
  \;\le\; \frac{q|f|_{a,b}+\frac{q(q-1)}{2}|g|_{a,b}}{\,b'-|\gamma(x)|\,},
\end{equation*}
which implies \eqref{eq:lyap bnd}.
\end{proof}

\section{Properties of the Fourier-weighted norm}

Recall that for $u: \T_a \to \C$ analytic, the Fourier-weighted norm $\|u\|_a$ is defined to be 
\begin{equation*}
    \|u\|_a := \sum_{k \in \Z} |\hat u_k| \eu^{2\pi a |k|},
\end{equation*}
where $\hat u_k$ is the $k$th Fourier coefficient of $u$.
In this appendix, we collect some elementary properties of this norm used in the article.

\begin{lemma}\label{lem:banach-alg}
For fixed $a>0$, the space of analytic functions $f:\T_a \to \C$ with $\|f\|_a<\infty$ is a Banach algebra under the norm $\|\cdot\|_a$.
\end{lemma}

\begin{proof}
It suffices to show $\|fg\|_a \leq \|f\|_a \|g\|_a$.
Recall that the Fourier coefficient $\wh{fg}_n$ of the product function $fg$ can be given in terms of a convolution $\sum_{k \in \Z} \wh{f}_{n - k} \wh{g}_k$.
To prove the lemma, the key observation is the triangle inequality $|n| \le |n-k|+|k|$ and the following computation:
\begin{equation*}
\begin{aligned}
  \|fg\|_a
  &= \sum_{n\in\Z} |\wh{fg}_n|\, \eu^{2\pi a |n|}
   \le \sum_{n\in\Z}\sum_{k\in\Z} |\wh f_{\,n-k}|\,|\wh g_k|\, \eu^{2\pi a |n|} \\
  &\le \sum_{n\in\Z}\sum_{k\in\Z} |\wh f_{\,n-k}|\, \eu^{2\pi a |n-k|}\; |\wh g_k|\, \eu^{2\pi a |k|}
   = \sum_{k\in\Z} \Big(\sum_{n\in\Z} |\wh f_{\,n-k}|\, \eu^{2\pi a |n-k|}\Big)\, |\wh g_k|\, \eu^{2\pi a |k|} \\
  &\le \sum_{k\in\Z} \Big(\sum_{m\in\Z} |\wh f_m|\, \eu^{2\pi a |m|}\Big)\, |\wh g_k|\, \eu^{2\pi a |k|}
   = \|f\|_a \, \|g\|_a.
\end{aligned}
\end{equation*}
The interchange of sums is justified since all series involve non-negative terms and converge absolutely. This completes the proof of the lemma.
\end{proof}

\begin{lemma}\label{lem:fourier MVT}
    Let $\delta \in \R$. 
    For $u: \T_a \to \C$ analytic, define $\CD_\delta u(x):= u(x + \delta) - u(x)$.
    Then the function $\CD_\delta u$ satisfies
    \begin{align}
        \norm{\CD_\delta u}_a \leq& |\delta| \norm{\dot u}_a\\
        \norm{\CD_\delta u - \delta \dot u}_a  \leq& \frac{1}{2}\verts{\delta}^2 \norm{\ddot u}_a.
    \end{align}
\end{lemma}

\begin{proof}
    Since $|\eu^{\iu \theta} - 1| \leq |\theta|$ for all $\theta \in \R$ and $\hat{\dot u}_k = 2\pi \iu k \hat u_k$, 
    \begin{equation*}
        \begin{aligned}
            \norm{\CD_\delta u}_a = &  \sum_{k \in \Z} \left|\left(e^{2 \pi \iu k \delta} - 1\right) \hat u_k \right| e^{2 \pi a |k|} 
            \leq  \sum_{k \in \Z} |\delta|\left|2 \pi k\hat u_k \right| e^{2 \pi a |k|} 
            = |\delta| \norm{\dot u}_a.
        \end{aligned}
    \end{equation*}
    On the other hand, since $|e^{\iu \theta} - 1 - \iu \theta| \leq \theta^2/2$ for all $\theta \in \R$ and $\hat{\ddot u}_k = (2 \pi \iu k)^2 \hat u_k$,
    \begin{equation*}
        \begin{aligned}
            \norm{\CD_\delta u - \delta \dot u}_a =&  \sum_{k \in \Z} \left|\left(e^{2 \pi \iu k \delta} - 1 - 2 \pi \iu k \delta \right) \hat u_k \right| e^{2 \pi a |k|}\\
            \leq& \sum_{k \in \Z} \frac{1}{2} \verts{\delta}^2 |2\pi k|^2  \left| \hat u_k \right| e^{2 \pi a |k|} 
            = \frac{1}{2}\verts{\delta}^2 \norm{\ddot u}_a. 
        \end{aligned}
    \end{equation*}
    This completes the proof of the lemma.
\end{proof}

\begin{lemma}[Fourier separation lemma]\label{lem:Fourier sep}
    Suppose  $M > N$ and $u, f : \T_a \to \C$ are analytic functions. Then we have
    \begin{equation}\label{eq:Fourier sep}
        \norm{\crochet{f \crochet{u}_q}_q^N}_a 
        \leq  \norm{f}_{a} \norm{\crochet{u}_q^{M} +  \Big(\frac{2N}{q}+1\Big) \eu^{-4 \pi a (M-N)}  \crochet{u}_q^{\geq M}}_{a}.
    \end{equation}
\end{lemma}

\begin{proof}
    Using $|k| - |n| \leq |n-k| \leq |k| + |n|$, we have
       \begin{equation*}
        \begin{aligned}
            &\norm{\crochet{ f \crochet{u}_q^{\geq M}}_q^N}_{a}  = \sum_{\substack{q \mid n \\ |n| \leq N}}  \sum_{\substack{q \mid k \\ |k| \geq M}} \verts{\hat f_{n-k} \hat u_k} e^{2 \pi a |n|}
            \leq  \sum_{\substack{q \mid n \\ |n| \leq N}}  \sum_{\substack{q \mid k \\ |k| \geq M}} \verts{f}_{a} e^{-2\pi a |n - k|} \verts{\hat u_k} e^{2 \pi a |n|}\\
            \leq & \sum_{\substack{q \mid n \\ |n| \leq N}}  \sum_{\substack{q \mid k \\ |k| \geq M}} \verts{f}_{a} e^{-2\pi a |n - k|}  e^{2 \pi a (|n| - |k|)} \verts{\hat u_k}e^{2 \pi a |k|}
            \leq  \sum_{\substack{q \mid n \\ |n| \leq N}}  \sum_{\substack{q \mid k \\ |k| \geq M}} \verts{f}_{a} e^{-4\pi a (|k| - |n|)}  \verts{\hat u_k}e^{2 \pi a |k|}\\
            \leq & \Big(\frac{2N}{q}+1\Big)   \sum_{\substack{q \mid k \\ |k| \geq M}} \verts{f}_{a} e^{-4\pi a (M - N)}  \verts{\hat u_k}e^{2 \pi a |k|}
            \leq  \Big(\frac{2N}{q}+1\Big)    e^{-4\pi a (M - N)}  \verts{f}_{a} \norm{\crochet{u}_q^{\geq M}}_{a}.
        \end{aligned}
    \end{equation*}
    Thus, 
    \begin{equation*}
        \begin{aligned}
            \norm{\crochet{f \crochet{u}_q}_q^N}_a \leq& \norm{\crochet{f\crochet{u}_q^{M}}_q^N}_a  + \norm{\crochet{f \crochet{u}_q^{\geq M}}_q^N}_a \\
            \leq& \norm{f}_a \left(\norm{\crochet{u}_q^{M}}_a  + \Big(\frac{2N}{q}+1\Big)   e^{-4\pi a (M - N)} \norm{\crochet{u}_q^{\geq M}}_{a} \right)\\
            \leq& \norm{f}_a \norm{\crochet{u}_q^{M}  + \Big(\frac{2N}{q}+1\Big)   e^{-4\pi a (M - N)}\crochet{u}_q^{\geq M}}_{a} 
        \end{aligned}
    \end{equation*}
    which proves \eqref{eq:Fourier sep}.
    This completes the proof of the lemma.
\end{proof}

\medskip

\noindent \textbf{Acknowledgements.} The author would like to thank Vadim Kaloshin for initiating this project and for his guidance throughout its development. The author also thanks Abed Bounemoura, Bassam Fayad, Mathieu Helfter, Comlan Edmond Koudjinan, Illya Koval, Yi Pan, and Daniel Tsodikovich for many helpful discussions. The financial support of the ERC grant \textit{SPERIG \#885707} is gratefully acknowledged.

\medskip

\printbibliography

\end{document}